\documentclass[11pt,reqno,oneside,a4paper]{amsart}
\usepackage[british]{babel}

\usepackage{graphicx}
\usepackage{amscd}
\usepackage{amsmath}
\usepackage{amsfonts}
\usepackage{amssymb}
\usepackage{mathtools}
\usepackage{mathrsfs}
\usepackage{comment}
\usepackage{color}
\usepackage[usenames,dvipsnames,table,xcdraw]{xcolor}
\usepackage{pdfsync}
\usepackage{bbm}
\usepackage{dsfont}
\usepackage[colorlinks=true]{hyperref}
\usepackage{enumerate}
\usepackage{booktabs}
\usepackage{float}
\usepackage{wrapfig}
\usepackage{caption}
\usepackage{subcaption}
\usepackage{longtable}
\usepackage{listings}
\usepackage[T1]{fontenc}
\usepackage[scaled]{beramono}
\usepackage{tikz}
\usepackage{pgffor}
\usepackage[all]{xy}

\definecolor{trp}{rgb}{1,1,1}

\definecolor{red}{rgb}{1,0,.2}

\newtheorem{theorem}{Theorem}[section]
\theoremstyle{plain}

\newtheorem{conjecture}{Conjecture}

\newtheorem{corollary}[theorem]{Corollary}

\newtheorem{definition}[theorem]{Definition}
\newtheorem{example}[theorem]{Example}
\newtheorem{question}[theorem]{Question}

\newtheorem{lemma}[theorem]{Lemma}

\newtheorem{prop}[theorem]{Proposition}
\newtheorem{remark}[theorem]{Remark}

\numberwithin{equation}{section}

\newcommand{\R}{\mathbb{R}}

\newcommand{\ii}{\mathbf{i}}
\newcommand{\jj}{\mathbf{j}}

\newcommand{\iih}{\boldsymbol{\hat\imath}}

\newcommand{\supp}{\mathrm{supp}}

\usepackage{anysize}

\usepackage{caption}

\definecolor{blue}{rgb}{0,0,1}

\definecolor{red}{rgb}{1,0,.7}

\begin{document}
\title[The $L^q$ spectrum of self-affine measures on sponges]{The $L^q$ spectrum of self-affine measures on sponges 
}

\author{Istv\'an Kolossv\'ary}
\address{Istv\'an Kolossv\'ary, \newline University of St Andrews,  School of Mathematics and Statistics, \newline St Andrews, KY16 9SS, Scotland} \email{itk1@st-andrews.ac.uk}


\begin{abstract}
In this paper a sponge in $\mathbb{R}^d$ is the attractor of an iterated function system consisting of finitely many strictly contracting affine maps whose linear part is a diagonal matrix. A suitable separation condition is introduced under which a variational formula is proved for the $L^q$ spectrum of any self-affine measure defined on a sponge for all $q\in\mathbb{R}$. Apart from some special cases, even the existence of their box dimension was not proved before. Under certain conditions the formula has a closed form which in general is an upper bound. The Frostman and box dimension of these measures is also determined. The approach unifies several existing results and extends them to arbitrary dimensions. The key ingredient is the introduction of a novel pressure function which aims to capture the growth rate of box counting quantities on sponges. We show that this pressure satisfies a variational principle which resembles the Ledrappier--Young formula for Hausdorff dimension. 
\end{abstract}

\maketitle

\thispagestyle{empty}

\section{Introduction}\label{sec:intro}

The $L^q$ spectrum $T(\nu,q): \mathbb{R}\to\mathbb{R}$ of a compactly supported Borel probability measure $\nu$ quantifies the global fluctuations of $\nu$ and thus knowledge of it provides valuable information about the multifractal properties of $\nu$ and also about the dimension of its support, see Section~\ref{sec:IntroLqSpectrum}. As such, it is a basic tool in fractal geometry that has a rich literature concerning measures supported by different fractal sets.

It was shown by Peres and Solomyak~\cite{PeresSolomyak_Indiana00} that the $L^q$ spectrum of any self-conformal measure exists for $q>0$ and extended to graph-directed self-conformal measures by Fraser~\cite{Fraser_TAMS2016}. When the support is a self-similar set, a closed form expression for the $L^q$ spectrum is known~\cite{LauNgai_AdvMath99, RIEDI_JMAA95, Shmerkin_Annals19} under different separation conditions on the cylinder sets. We do not pursue this direction further since the focus of this paper is on the more general self-affine setting. 

Self-affine sets and measures are important building blocks in the study of smooth non-conformal dynamical systems and have thus gained a lot of attention lately. The study of these systems is more challenging than the conformal case and therefore there are far fewer results especially in dimensions $d\geq 3$. In one line of research, the $L^q$ spectrum of specific systems are considered. Feng and Wang~\cite{FengWang2005} calculated the $L^q$ spectrum of self-affine measures on the plane supported on attractors of iterated function systems given by orientation preserving diagonal matrices satisfying a suitable separation condition. This was extended by Fraser~\cite{Fraser_TAMS2016} to include reflections and rotations by 90 degrees. Ni and Wen~\cite{NiWen_LqGraphDirSelfAffine} considered a class of graph-directed self-affine measures. In higher dimensions, self-affine measures have only been studied on Bedford--McMullen (also called Sierpi\'nski) sponges by Olsen~\cite{Olsen_PJM98}, and~\cite{FraserOlsen_indiana, Olsen_RandomBMSponge} in a random setting. In the other direction, `generic' systems were considered in~\cite{BarralFeng_13,Falconer_GenericLq99}. The main objective of this paper is to build a general framework to study box counting quantities and in particular to determine the $L^q$ spectrum of self-affine measures supported on higher dimensional self-affine sponges where very little is known. These sets constitute a fundamental family of self-affine sets showcasing a number of interesting properties that set them apart from the `generic' systems.
 
\subsection*{Main contribution} In this paper the linear part of all the strictly contracting affine maps defining a sponge in $\mathbb{R}^d$ is a diagonal matrix. The \emph{separation of principal projections  condition} (SPPC) is introduced, see Definition~\ref{def:SPPC} and~\cite{FK_DimAMeasure_arxiv}, which gives extra grid alignment for the first level cylinder sets of the sponge. Roughly speaking, the entries of the diagonal matrices determine `relevant' orderings of the coordinates and the SPPC assumes that all orthogonal projections of the first level cylinders onto subspaces determined by these `relevant' orderings satisfy the, more familiar, open set condition. On the plane, the much studied  Lalley--Gatzouras~\cite{GatzourasLalley92} and Bara\'nski~\cite{BARANSKIcarpet_2007} (hence also Bedford--McMullen~\cite{Bedford84_phd,mcmullen84}) carpets are precisely the sets which satisfy the SPPC. Therefore, it naturally unifies the Lalley--Gatzouras and Bara\'nski classes, moreover, in higher dimensions it extends to a much wider class of sponges than simply these two classes.

The main result, see Theorem~\ref{thm:Lqmain}, states that if the self-affine measure $\nu_{\boldsymbol{\mu}}$ (defined by the probability vector $\boldsymbol{\mu}$) is fully supported on a self-affine sponge in $\mathbb{R}^d$ which satisfies the SPPC, then
\begin{equation*}
T(\nu_{\boldsymbol{\mu}},q) =  P(\boldsymbol{\psi}_q^{\boldsymbol{\mu}}) \quad\text{ for all } q\in\mathbb{R},
\end{equation*}
where $P$ is a novel pressure-like functional defined in~\eqref{eq:13} and $\boldsymbol{\psi}_q^{\boldsymbol{\mu}}$ is a family of potentials defined in~\eqref{eq:38} that depend on $q$ and $\boldsymbol{\mu}$. The key contribution is to use ideas from thermodynamic formalism to define $P$ in a way that is specifically tailored to capture the polynomial growth rate of box counting quantities such as the $L^q$ spectrum on sponges. The main technical result of the paper, see Theorem~\ref{thm:main1}, is to show that $P$ satisfies a variational principle. It shows resemblance to the Ledrappier--Young formula for Hausdorff dimension. However, since the box and Hausdorff dimension of such sponges is `typically' different, there is a clear distinction between the two variational principles. Generalising this variational principle further could be of independent interest.

We point out a few important aspects and advantages of our approach:
\begin{itemize}
\item the result for $T(\nu_{\boldsymbol{\mu}},q)$ is valid \emph{for all} $q\in\mathbb{R}$. Handling negative $q$ is known to be very challenging, in particular, in the non-conformal case we are only aware of the result of Olsen~\cite{Olsen_PJM98} about Bedford--McMullen sponges which are a very special case of the ones we consider. The potential $\boldsymbol{\psi}_q^{\boldsymbol{\mu}}$ is just a specific choice in our more general Theorem~\ref{thm:main1}.
\item The separation condition is weaker than the one considered in~\cite{Olsen_PJM98}.
\item The box dimension of the sponge is given by choosing $q=0$. Apart from the planar case, some three dimensional cases and Lalley--Gatzouras sponges~\cite{Raoetal_BoxCountingMeasure_arxiv}, even the box dimension of these sponges was not known before to exist.
\item Introducing `relevant' orderings of the coordinates is the key ingredient in the definition of $P$. The necessity of this is demonstrated on an example in Section~\ref{sec:3DBaranski}.
\end{itemize}
Section~\ref{sec:discussion} details related literature and includes two worked out examples showing how our approach is able to go beyond previous methods.

Further contribution is to calculate the Frostman and box dimension of any self-affine measure supported by a sponge satisfying the SPPC, see Theorem~\ref{thm:dimBFmeasure}. These dimensions give the slope of the asymptotes of the $L^q$ spectrum as $q$ tends to $+\infty$ and $-\infty$, respectively. To the best of our knowledge these have only been calculated for Bedford--McMullen sponges~\cite{Olsen_PJM98}.

We provide sufficient conditions under which the variational formula translates into a closed form expression, see Corollary~\ref{cor:main}. This is the case for sponges in the Lalley--Gatzouras class. In general, the closed form gives an upper bound for the pressure. A natural direction for further research could be to get a better understanding of the relationship between the variational formula and the closed form.

\subsubsection*{Structure of paper}
We continue the section with the formal introduction of the $L^q$ spectrum and then the self-affine sponges and measures. In Section~\ref{sec:01} we set up symbolic notation in order to define the pressure $P(\boldsymbol{\varphi})$ in~\eqref{eq:13} and state all our results regarding it. Section~\ref{sec:02} begins with the definition of the SPPC followed by the statements about the $L^q$ spectrum and the Frostman and box dimensions of the self-affine measure. Section~\ref{sec:discussion} gives further context to our results. Sections~\ref{sec:03} through~\ref{sec:proofFrostmanBox} contain the proofs of our results.

\subsection{The \texorpdfstring{$L^q$}{Lq} spectrum}\label{sec:IntroLqSpectrum}

A collection of closed balls $\{B(x_i,\delta)\}_i$ is a \emph{centred packing} of a set $F\subset \mathbb{R}^d$ if the balls are disjoint and all $x_i\in F$. Given a probability measure $\nu$ with compact support $\supp(\nu)$, for $\delta>0$ and $q\in\mathbb{R}$ let
\begin{equation*}
T_{\delta}(\nu,q) \coloneqq \sup \bigg\{\sum_{i} \big(\nu(B(x_{i}, \delta))\big)^{q} \mid \left\{B\left(x_{i}, \delta\right)\right\}_{i} \text { is a centred packing of } \supp(\nu) \bigg\}
\end{equation*}
and define the \emph{$L^q$ spectrum} of $\nu$ to be
\begin{equation*}
	T(\nu,q) \coloneqq \lim_{\delta\to 0} \frac{\log T_{\delta}(\nu,q)}{-\log \delta}
\end{equation*}
provided the limit exists, otherwise one takes lower and upper limits denoted by $\underline{T}(\nu,q)$ and $\overline{T}(\nu,q)$, respectively. Various definitions exist in the literature, see for example~\cite[Section~4]{PeresSolomyak_Indiana00} or~\cite[Section~1.1]{Fraser_TAMS2016} for some comparisons. The main reason for our choice is that $T(\nu,q)$ is well-defined for all $q\in\mathbb{R}$. Technical issues can arise for other definitions when $q<0$, see the remark after~\cite[proof of Proposition~3.1]{LauNgai_AdvMath99} or after~\cite[Proposition~2]{RIEDI_JMAA95}.  The \emph{$L^q$ dimension} of $\nu$ is the ratio
\begin{equation*}
D(\nu,q)\coloneqq \frac{T(\nu,q)}{1-q} \quad\text{ for } q\neq 1.
\end{equation*}
In case $q=1$, the \emph{entropy dimension} is used instead defined by 
\begin{equation*}
\dim_{\mathrm e} \nu \coloneqq \lim_{\delta\to 0} \frac{\inf  \sum_{B\in\mathcal{D}_{\delta}} \nu(B) \log(1/\nu(B)) }{-\log \delta}, 
\end{equation*}
where the infimum is taken over all finite Borel partitions of $\supp(\nu)$ with sets of diameter at most $\delta$. One takes lower and upper limits if the limit does not exist. Let $\dim_{\mathrm H}, \dim_{\mathrm B}$ and $\dim_{\mathrm P}$ denote the Hausdorff, box and packing dimensions, respectively, see~\cite{FalconerBook} for basic definitions. 

Knowledge of the $L^q$ spectrum of a measure provides valuable information about the measure and its support. It follows from the definitions that
\begin{equation*}
\underline{\dim}_{\mathrm B}\, \supp(\nu) = \underline{T}(\nu,0) \;\;\text{ and }\;\; \overline{\dim}_{\mathrm B} \,\supp(\nu) = \overline{T}(\nu,0).
\end{equation*}
Furthermore, if $\overline{T}(\nu,q)$ is differentiable at $q=1$, then Ngai~\cite{Ngai_97PAMS} showed that
\begin{equation*}
\dim_{\mathrm H} \nu = \dim_{\mathrm P} \nu = \dim_{\mathrm e} \nu = -\overline{T}'(1).
\end{equation*}
The value $-T(\nu,2)$ is often called the \emph{correlation dimension} or \emph{R\'enyi entropy}. The asymptotes of $T(\nu,q)$ as $q$ tends to $+ \infty$ and $-\infty$ are related to the Frostman and box dimension of the measure, respectively. Defined in~\cite{FFF_MinkowskiDimMeas_2020arxiv}, the \emph{Frostman dimension} of $\nu$ gives the decay rate of the ball with largest $\nu$ measure, more precisely,  
\begin{align*}
\dim_{\mathrm{F}} \nu\coloneqq \sup \{s \geq 0:\; &\text{there exists a constant } C \geq 1 \text { such that } \\
	&\left. \nu(B(x, \delta)) \leq C \delta^{s} \text { for all } x \in X \text{ and } 0<\delta<1\right\}
\end{align*}
and the dual notion of upper box (or Minkowski) dimension of $\nu$ is
\begin{align*}
\overline{\dim}_{\mathrm{B}}\, \nu\coloneqq \inf \{s \geq 0:\; &\text{there exists a constant } c>0 \text{ such that } \\
&\left. \nu(B(x, \delta)) \geq c \delta^{s} \text{ for all } x \in X \text{ and } 0<\delta<1 \right\}.
\end{align*}
For the lower box dimension of $\nu$, denoted $\underline{\dim}_{\mathrm{B}}\,\nu$, only a sequence $\delta_n\to 0$ needs to exist for which $\nu(B(x, \delta_n)) \geq c \delta_n^{s}$. If $\overline{\dim}_{\mathrm{B}}\,\nu=\underline{\dim}_{\mathrm{B}}\,\nu$, then the common value is called the \emph{box dimension} of $\nu$ denoted by $\dim_{\mathrm{B}}\,\nu$. Heuristically, if $q$ is a very large positive number then $T_\delta(\nu,q)$ is dominated by the ball(s) with largest mass, hence, $T_{\delta}(\nu,q)$ roughly behaves like $\delta^{q\cdot \dim_{\mathrm{F}}\nu}$ and one can expect $D(\nu,q)\to\dim_{\mathrm{F}}\nu$ as $q\to+\infty$. See~\cite[Proposition~4.2]{FFF_MinkowskiDimMeas_2020arxiv} for a precise statement of the dual claim that $D(\nu,q)\to\overline{\dim}_{\mathrm{B}}\,\nu$ as $q\to-\infty$. It was recently shown in~\cite{BJK_Chaos_IMRN22} that $\dim_{\mathrm{B}}\nu$ determines the convergence rate of the chaos game.

The $L^q$ spectrum is also intimately connected to multifractal analysis, see~\cite[Chapter 17]{FalconerBook} for some background. In one direction, the \emph{coarse multifractal spectrum} $f_C(\alpha):\, \mathbb{R}^+\to\mathbb{R}^+$ gives, roughly speaking, the power law exponent of the number of $\delta$-mesh cubes with $\nu$ measure approximately $\delta^{\alpha}$. Riedi~\cite{RIEDI_JMAA95} showed that the Legendre transform of $f_C(\alpha)$ is always equal to the $L^q$ spectrum, and vice-versa, if $T(\nu,q)$ is differentiable everywhere then its Legendre transform is equal to $f_C(\alpha)$ (otherwise it gives the convex hull of $f_C(\alpha)$). In the other direction, the \emph{fine multifractal spectrum} $f_H(\alpha)$ gives the Hausdorff dimension of the set of points in the support of $\nu$ with local dimension equal to $\alpha$. As a heuristic, it is said that the \emph{multifractal formalism} holds if $f_H(\alpha)$ is given by the Legendre transform of the $L^q$ spectrum. This fails in general, but was shown to hold for example for self-similar sets satisfying the strong separation condition~\cite{CawleyMauldin_Multifractal92, RIEDI_JMAA95}. Olsen introduced generalised Hausdorff measures to serve as an alternative to the $L^q$ spectrum~\cite{Olsen_AdvMath95} and showed that this formalism works for self-affine measures on Bedford--McMullen sponges~\cite{Olsen_PJM98}.

\subsection{Self-affine sponges and measures}\label{sec:IntroSponges}

Given a finite index set $\mathcal{I}$, an affine \emph{iterated function system} (IFS) on $\R^d$ is a finite family $\mathcal{F}=\{f_i\}_{i\in\mathcal{I}}$ of affine contracting maps $f_i:\, \R^d\to\R^d$ of the form $f_i(x)=A_i x+t_i$. The IFS determines a unique, non-empty compact set $F$, called the \emph{attractor}, that satisfies the relation
\begin{equation*}
	F = \bigcup_{i\in\mathcal{I}} f_i(F).
\end{equation*}
In case the linear part $A_i$ of each $f_i$ is a diagonal matrix with main diagonal $\big(a_i^{(1)},\ldots,a_i^{(d)}\big)$, we call $F$ a \emph{(self-affine) sponge}. For $1\leq n\leq d$ and $i\in\mathcal{I}$, let $\lambda_i^{(n)}\coloneqq|a_i^{(n)}|\in(0,1)$. Without loss of generality we assume that $f_i([0,1]^d)\subset [0,1]^d$ and that there is no $i\neq j$ such that $f_i(x)=f_j(x)$ for every $x\in[0,1]^d$.  We also assume that there exists $r_0=r_0(F)>0$ such that for every $1\leq n\leq d$ and $u\in\{0,1\}$ there exists $k_{u}^{(n)}\in\mathcal{I}$ such that
\begin{equation}\label{eq:19}
\mathrm{dist} \big( \{(x_1,\ldots,x_d)\in[0,1]^d:\, x_n=u \}, f_{k_{u}^{(n)}}([0,1]^d) \big) \geq r_0.
\end{equation}
In other words, for each face of the unit hypercube there is a map which sends the hypercube at least $r_0$ distance away from the face. Otherwise, $F$ is a subset of that face and is not `genuinely' $d$-dimensional.

More generally, $F$ can be referred to as a sponge also if the diagonal matrix is composed with a permutation matrix, see~\cite{Fraser12Boxlike} for $d=2$ and~\cite{fraserJurga2021AdvMath} for $d=3$. Sponges on the plane are generally called \emph{self-affine carpets} or \emph{box-like sets} and have a rich literature compared to the case $d\geq 3$. We give a more detailed account of relevant related literature in Section~\ref{sec:discussion}.

The orthogonal projections of $F$ onto the principal $n$-dimensional subspaces play a vital role in the arguments. Let $\mathcal{S}_d$ be the symmetric group on the set $\{1,\ldots,d\}$. For a permutation $\sigma=\{\sigma_1,\ldots,\sigma_d\}\in\mathcal{S}_d$ of the coordinates, let $E_n^{\sigma}$ denote the $n$-dimensional subspace spanned by the coordinate axes indexed by $\sigma_1,\ldots,\sigma_n$. Let $\Pi_n^{\sigma}:[0,1]^d\to E_n^\sigma$ be the orthogonal projection onto $E_n^{\sigma}$. For $n=d$, $\Pi_d^{\sigma}$ is simply the identity map. We say that $f_i$ and $f_j$ \emph{overlap exactly} on $E_n^{\sigma}$ if
\begin{equation*}
\Pi_n^{\sigma}(f_i(x))=\Pi_n^{\sigma}(f_j(x)) \;\text{ for every }\; x\in[0,1]^d.
\end{equation*}
Observe that if $f_i$ and $f_j$ overlap exactly on $E_n^{\sigma}$ then they also overlap exactly on $E_m^{\sigma}$ for all $1\leq m \leq n$ but may not overlap exactly on any $E_n^{\sigma'}$ for some other $\sigma'\in\mathcal{S}_d$. The definition of the separation condition we require is postponed to Definition~\ref{def:SPPC}.

Given an affine IFS $\mathcal{F}$ with attractor $F$ and a probability vector $\boldsymbol{\mu}=(\mu(i))_{i\in\mathcal{I}}$ with strictly positive entries, there exists a unique probability measure $\nu_{\boldsymbol{\mu}}$ fully supported by $F$ which satisfies 
\begin{equation*}
	\nu_{\boldsymbol{\mu}} = \sum_{i\in\mathcal{I}} \mu(i)\, \nu_{\boldsymbol{\mu}}\circ f_i^{-1}.
\end{equation*}
The \emph{self-affine measure} $\nu_{\boldsymbol{\mu}}$ has an equivalent characterisation as the push-forward of the Bernoulli measure by the natural projection from the symbolic space to the attractor. Formally, given $\boldsymbol{\mu}$, the Bernoulli measure on the symbolic space $\Sigma=\mathcal{I}^{\mathbb{N}}$ is the product measure $\widetilde{\nu}_{\boldsymbol{\mu}}=\boldsymbol{\mu}^{\mathbb{N}}$. The natural projection $\pi:\,\Sigma\to F$ is given by
\begin{equation*}
	\pi(\ii) =\pi(i_1,i_2,\ldots, i_k, \ldots) \coloneqq \lim_{k\to\infty} f_{i_1i_2\ldots i_k}(0),
\end{equation*} 
where $f_{i_1i_2\ldots i_k}=f_{i_1}\circ f_{i_2} \circ\ldots\circ f_{i_k}$. Then $\nu_{\boldsymbol{\mu}} = \widetilde{\nu}_{\boldsymbol{\mu}}\circ \pi^{-1}$.

\section{Variational principle for box counting quantities}\label{sec:01}

The classical variational principle for topological pressure, pioneered by the works of Ruelle~\cite{Ruelle_73TAMS} and Walters~\cite{Walters_75} is an essential tool in the thermodynamic formalism of dynamical systems. Given a dynamical system $(X,T)$, the topological pressure $P(T,\varphi)$ of a continuous potential $\varphi:\, X\to\mathbb{R}$ satisfies the variational principle
\begin{equation}\label{eq:16}
	P(T, \varphi)=\sup _{\nu \in \mathscr{M}^{T}(X)}\left(h_{\nu}(T)+\int_{X} \varphi\, \mathrm{d} \nu\right),
\end{equation}
where $\mathscr{M}^{T}(X)$ denotes the set of $T$-invariant Borel probability measures on $X$ and $h_{\nu}(T)$ is the measure-theoretic entropy of $\nu$ with respect to $T$, see~\cite{waltersBook} for definitions and background. More recently, motivated by the study of self-affine carpets and sponges, a more general weighted notion of pressure for factor maps between general topological dynamical systems was introduced \cite{BarralFeng_12, FengHuang_16, tsukamoto_2022}. Given $a_1>0$, $a_2\geq 0$ and two dynamical systems $(X,T)$ and $(Y,S)$ with a factor map $f$ between them (i.e. $f$ is a continuous surjection with $f\circ T = S\circ f$) there is a meaningful way to define the weighted pressure 	$P^{\left(a_{1}, a_{2}\right)}(T, \varphi)$ of the potential $\varphi$ such that the following
variational principle holds
\begin{equation}\label{eq:17}
	P^{\left(a_{1}, a_{2}\right)}(T, \varphi)=\sup _{\nu \in \mathscr{M}^{T}(X)}\left(a_{1} h_{\nu}(T)+a_{2} h_{\nu\circ f^{-1}}(S)+\int_{X} \varphi\, \mathrm{d} \nu\right).
\end{equation}
The formula can be extended to a sequence of factor maps. The definition of the pressure resembles the Hausdorff dimension. For example, for a particular choice of $(a_1,a_2)$ and $\varphi\equiv 0$, the Hausdorff dimension of a Bedford--McMullen carpet can be recovered from~\eqref{eq:17}. Olsen's formalism for multifractal analysis mentioned at the end of Section~\ref{sec:IntroLqSpectrum} is related to this weighted pressure. However, the sponges considered in this paper `typically' have different Hausdorff and box dimension. Therefore, these results can not be used directly to calculate the $L^q$ spectrum.

Instead, the main technical contribution of the paper is to set up a novel formalism that attempts to capture box counting quantities such as the $L^q$ spectrum. We keep the setting as simple as possible that still accommodates our goal. Generalising this formalism to more general contexts could be of independent interest.

\subsection{Symbolic setting}\label{subsec:11}

Recall $\mathcal{I}$ denotes the finite index set of the IFS $\mathcal{F}$ and $\Sigma = \mathcal{I}^{\mathbb{N}}$ is the space of all one-sided infinite words $\ii=i_1,i_2,\ldots$. For $\delta>0$, the \emph{$\delta$-stopping of $\ii\in\Sigma$ in the $n$-th coordinate} (for $n=1,\ldots,d$) is the unique integer $L_{\delta}(\ii,n)$ such that
\begin{equation}\label{eq:10}
\prod_{\ell=1}^{L_{\delta}(\ii,n)} \lambda_{i_\ell}^{(n)} \leq \delta < \prod_{\ell=1}^{L_{\delta}(\ii,n)-1}\lambda_{i_\ell}^{(n)}.
\end{equation}
We say that \emph{$\ii\in\Sigma$ is $\sigma$-ordered at scale $\delta$} if $L_{\delta}(\ii,\sigma_d)\leq L_{\delta}(\ii,\sigma_{d-1})\leq \ldots\leq L_{\delta}(\ii,\sigma_1)$, where to make the ordering unique, we use the convention that if $L_{\delta}(\ii,\sigma_n)=L_{\delta}(\ii,\sigma_{n-1})$ then $\sigma_n>\sigma_{n-1}$. We introduce $\Sigma_{\delta}^{\sigma}\coloneqq \{\ii\in\Sigma:\, \ii \text{ is } \sigma\text{-ordered at scale } \delta\}$, the set $\mathcal{A}_{\delta}\coloneqq \{ \sigma\in\mathcal{S}_d:\, \Sigma_{\delta}^{\sigma}\neq \emptyset\}\subseteq \mathcal{S}_d$  and let $\mathcal{A}\coloneqq \bigcup_{\delta>0}\mathcal{A}_{\delta}$. Since the $\sigma$-ordering is unique, the collection $\{\Sigma_{\delta}^{\sigma}:\, \sigma\in\mathcal{A}_{\delta}\}$ gives a partition of $\Sigma$ for every $\delta>0$.

For each permutation $\sigma=\{\sigma_1,\ldots,\sigma_d\}\in\mathcal{A}$ we define index sets $\mathcal{I}_d^{\sigma}\supseteq \mathcal{I}_{d-1}^{\sigma}\supseteq \ldots \supseteq \mathcal{I}_1^{\sigma}$ with $\mathcal{I}_d^{\sigma}\coloneqq \mathcal{I}$ as follows. Initially set $\mathcal{I}_d^{\sigma}= \mathcal{I}_{d-1}^{\sigma}= \ldots = \mathcal{I}_1^{\sigma}$. For $i<j$ ($i,j\in\mathcal{I}$), starting from $n=d-1$ and decreasing $n$, we check whether $f_i$ and $f_j$ overlap exactly on $E_n^{\sigma}$. If they do not overlap exactly for any $n$, then we move onto the next pair $(i,j)$, otherwise, we take the largest $n'$ for which $f_i$ and $f_j$ overlap exactly on $E_{n'}^{\sigma}$ and remove $j$ from $\mathcal{I}_{n'}^{\sigma},\mathcal{I}_{n'-1}^{\sigma},\ldots, \mathcal{I}_{1}^{\sigma}$ and then move onto the next pair $(i,j)$. The sets $\mathcal{I}_{d-1}^{\sigma}, \ldots, \mathcal{I}_1^{\sigma}$ are what remain after repeating this procedure for all pairs $i<j$. Further abusing notation, we denote by $\Pi_n^{\sigma}:\, \mathcal{I}\to\mathcal{I}_n^{\sigma}$ the `projection' of $j\in\mathcal{I}$ onto $\mathcal{I}_n^{\sigma}$, i.e.
\begin{equation*}
\Pi_n^{\sigma}j=i, \quad\text{ if } f_i \text{ and } f_j \text{ overlap exactly on } E_n^{\sigma} \text{ and } i\in\mathcal{I}_n^{\sigma}.
\end{equation*}
Defining $\Sigma_{n}^{\sigma}\coloneqq (\mathcal{I}_{n}^{\sigma})^{\mathbb{N}}$, we also let $\Pi_n^{\sigma}: \Sigma\to \Sigma_{n}^{\sigma}$ by acting coordinate wise, i.e. $\Pi_n^{\sigma}\ii=\Pi_n^{\sigma} i_1,\Pi_n^{\sigma}i_2,\ldots$. For completeness, let $\Pi_d^{\sigma}$ be the identity map on $\Sigma$. On each symbolic space $\Sigma_{n}^{\sigma}$ the dynamics is run by the left shift operator. Due to the coordinate wise definition, all maps $\Pi_n^{\sigma}$ commute with the left shift, hence all are factor maps.
 
We further partition each $\Sigma_{\delta}^{\sigma}$ into symbolic $\delta$-approximate cubes which play a crucial role in covering arguments of sponges. For two (finite or infinite) words $\ii$ and $\jj$, we denote the length of their longest common prefix by $|\ii\wedge \jj|= \min\{\ell:\, i_\ell\neq j_\ell\}-1$. The \emph{symbolic $\delta$-approximate cube} containing $\ii\in\Sigma_{\delta}^{\sigma}$ is
\begin{equation}\label{eq:29}
B_{\delta}(\ii)\coloneqq \left\{\jj \in \Sigma: \left|\Pi_{n}^{\sigma} \jj \wedge \Pi_{n}^{\sigma} \ii\right| \geq L_{\delta}(\ii, \sigma_{n}) \;\text { for every } 1 \leq n \leq d\right\}.
\end{equation}
Observe that if $\ii\in\Sigma_{\delta}^{\sigma}$, then for all $\jj\in B_{\delta}(\ii)$ also $\jj\in \Sigma_{\delta}^{\sigma}$. Thus, we define the $\sigma$-ordering of $B_{\delta}(\ii)$ with the $\sigma$-ordering of $\ii$ at scale $\delta$. As a result, the set $\mathcal{B}_{\delta}^{\sigma}$ of $\sigma$-ordered $\delta$-approximate cubes forms a partition of $\Sigma_{\delta}^{\sigma}$. The name comes from the fact that the image $\pi(B_{\delta}(\ii))\subseteq F$ lies within a cuboid of side lengths at most $\delta$ parallel to the coordinate axes. Finally, if $\ii\in\Sigma_{\delta}^{\sigma}$, then the surjectivity of the maps $\Pi_n^{\sigma}$ implies that $B_{\delta}(\ii)$ can be identified with a sequence of symbols of length $L_{\delta}(\ii,\sigma_1)$ of the form 
\begin{equation}\label{eq:11}
\left(\Pi_{n}^{\sigma} i_{L_{\delta}(\ii, \sigma_{n+1})+1}, \ldots,\Pi_{n}^{\sigma} i_{L_{\delta}(\ii, \sigma_{n})}\right)_{n=1}^{d} \in \bigtimes_{n=1}^d (\mathcal{I}_{n}^{\sigma})^{L_{\delta}(\ii,\sigma_n)-L_{\delta}(\ii,\sigma_{n+1})},
\end{equation}
where we set $L_{\delta}(\ii,\sigma_{d+1})\coloneqq0$. This will be crucial in determining the number of different approximate cubes with a fixed digit frequency.
 
\subsection{Topological pressure and variational principle}\label{subsec:12}

The main new ingredient is that rather than using just a single potential on $\Sigma$, we are working with a family of potentials $\boldsymbol{\varphi}=\{\varphi_n^{\sigma}\}_{\sigma,n}$ defined on $\{\Sigma_{n}^{\sigma}\}_{\sigma,n}$. In order to keep arguments simple, we let $\varphi_n^{\sigma}$ depend on $\ii\in\Sigma_n^{\sigma}$ only through $i_1$, i.e. $\varphi_n^{\sigma}$ is essentially defined on $\mathcal{I}_n^{\sigma}$. This is still sufficient for us to obtain results about the box dimension of sponges and the $L^q$ spectrum of self-affine measures defined on them, see Section~\ref{sec:ResLq} for statements.

For a fixed family of potentials
\begin{equation}\label{eq:22}
\boldsymbol{\varphi} = \{ \varphi_n^{\sigma}:\, \mathcal{I}_n^{\sigma}\to \mathbb{R}\, |\, \sigma\in\mathcal{A},\, 1\leq n\leq d\}
\end{equation}
and $\ii\in\Sigma_{\delta}^{\sigma}$, we define
\begin{equation*}
\Phi\left(B_{\delta}(\ii)\right)\coloneqq\sum_{n=1}^{d} \,
\sum_{\ell=L_{\delta}(\ii, \sigma_{n+1})+1}^{L_{\delta}(\ii, \sigma_{n})} \varphi_{n}^{\sigma}\left(\Pi_{n}^{\sigma} i_{\ell}\right)
\end{equation*}
to be the value of $\boldsymbol{\varphi}$ on $B_{\delta}(\ii)$ at scale $\delta$. Recalling that $\Sigma= \bigsqcup_{\sigma\in\mathcal{A}_{\delta}}\bigsqcup_{B\in\mathcal{B}_{\delta}^{\sigma}}B$ for every $\delta>0$, it is natural to introduce the topological pressure like quantities
\begin{equation}\label{eq:13}
\overline{P}(\boldsymbol{\varphi}) \coloneqq \limsup_{\delta \rightarrow 0} \frac{-1}{\log \delta}\, 
\log \bigg[\sum_{\sigma \in \mathcal{A}_{\delta}} \sum_{B_{\delta}(\ii) \in \mathcal{B}_{\delta}^{\sigma}} \exp \left[\Phi\left(B_{\delta}(\ii)\right)\right]\bigg]
\end{equation}
and $\underline{P}(\boldsymbol{\varphi})$ with a $\liminf_{\delta\to 0}$ instead. We state in Theorem~\ref{thm:main1} that $\overline{P}(\boldsymbol{\varphi})=\underline{P}(\boldsymbol{\varphi})$ for any choice of $\boldsymbol{\varphi}$ and denote this common limit by $P(\boldsymbol{\varphi})$.

We introduce additional notation. The \emph{Shannon entropy} $H(\mathbf{p})$ of a probability vector $\mathbf{p}=(p(i))_i$ is the sum $-\sum_ip(i)\log p(i)$. Fix $\sigma\in\mathcal{A}$. Let $\mathcal{P}_n^{\sigma}$ denote the set of probability vectors on $\mathcal{I}_n^{\sigma}$ (i.e. $\mathbf{p}\in\mathcal{P}_n^{\sigma}$ if $p(i)\geq 0$ for all $i\in\mathcal{I}_n^{\sigma}$ and $\sum_{i \in \mathcal{I}_{n}^{\sigma}}p(i)=1$). Define $\mathcal{P}^{\sigma}\coloneqq \mathcal{P}_d^{\sigma}\times \mathcal{P}_{d-1}^{\sigma}\times\ldots \times \mathcal{P}_1^{\sigma}$. An element of $\mathcal{P}^{\sigma}$ is $\mathbf{P}_{\!\sigma}=(\mathbf{p}_{\sigma_d},\ldots,\mathbf{p}_{\sigma_1})$, where $\mathbf{p}_{\sigma_n}=(p_{\sigma_n}(i))_{i\in\mathcal{I}_n^{\sigma}}$. For $1\leq n\leq m\leq d$ and $\mathbf{p}_{\sigma_m}\in\mathcal{P}_m^{\sigma}$, we denote the \emph{Lyapunov exponent} by
\begin{equation*}
\chi_n^{\sigma}(\mathbf{p}_{\sigma_m}) \coloneqq -\sum_{i\in\mathcal{I}_{m}^{\sigma}} p_{\sigma_m}(i) \log \lambda_i^{(\sigma_n)}.
\end{equation*}
For a fixed $\mathbf{P}_{\!\sigma} \in \mathcal{P}^{\sigma}$, we define constants $C_{n}^{(d), \sigma}(\mathbf{P}_{\!\sigma})$ for $n=d,d-1,\ldots,1$ recursively as follows: let $C_d^{(d),\sigma}(\mathbf{P}_{\!\sigma})\coloneqq1/\chi_d^{\sigma}(\mathbf{p}_{\sigma_d})$ and
\begin{equation}\label{eq:15}
C_{n}^{(d), \sigma}(\mathbf{P}_{\!\sigma})\coloneqq \bigg(\! 1-\!\sum_{m=n+1}^{d} C_m^{(d),\sigma}(\mathbf{P}_{\!\sigma})\cdot \chi_n^{\sigma}(\mathbf{p}_{\sigma_m})\! \bigg) \frac{1}{\chi_n^{\sigma}(\mathbf{p}_{\sigma_n})}.
\end{equation}
Note that $C_n^{(d),\sigma}(\mathbf{P}_{\!\sigma})$ may be negative for $n< d$ and depends on $\mathbf{P}_{\!\sigma}$ only through $\chi_\ell^{\sigma}(\mathbf{p}_{\sigma_m})$ for $n\leq \ell\leq m\leq d$. Of particular importance is the subset
\begin{equation*}
\mathcal{Q}^{\sigma}\coloneqq \{ \mathbf{P}_{\!\sigma} \in \mathcal{P}^{\sigma}:\, C_n^{(d),\sigma}(\mathbf{P}_{\!\sigma})\geq 0 \text{ for all } 1\leq n\leq d \}.
\end{equation*}
In fact, we will show that $\sigma\in\mathcal{A}$ if and only if $\mathcal{Q}^{\sigma}\neq \emptyset$, see Lemma~\ref{lem:4}. Slightly abusing notation for the integral, we write
\begin{equation*}
\int\! \varphi_n^{\sigma}\,\mathrm{d}\mathbf{p}_{\sigma_n} \coloneqq \sum_{i \in \mathcal{I}_{n}^{\sigma}} p_{\sigma_{n}}(i) \cdot \varphi_{n}^{\sigma}(i).
\end{equation*}
Our main technical result, proved in Section~\ref{sec:04}, is the following variational principle for $P(\boldsymbol{\varphi})$.

\begin{theorem}\label{thm:main1}
For any family of potentials $\boldsymbol{\varphi}$ as in~\eqref{eq:22} the limit $P(\boldsymbol{\varphi})$ exists, moreover, 
\begin{equation}\label{eq:12}
P(\boldsymbol{\varphi}) = \max _{\sigma \in \mathcal{A}}\, \sup _{\mathbf{P}_{\!\sigma} \in \mathcal{Q}^{\sigma}}\, \sum_{n=1}^{d} 
C_{n}^{(d), \sigma}(\mathbf{P}_{\!\sigma}) \cdot\bigg(H(\mathbf{p}_{\sigma_{n}}) + \int\! \varphi_n^{\sigma}\,\mathrm{d}\mathbf{p}_{\sigma_n}\bigg).
\end{equation}
Let $t_{\sigma}(\mathbf{P}_{\!\sigma})=t(\mathbf{P}_{\!\sigma}) = t(\mathbf{p}_{\sigma_d};\ldots;\mathbf{p}_{\sigma_1})$ denote the sum in~\eqref{eq:12} for any $\mathbf{P}_{\!\sigma} \in \mathcal{P}^{\sigma}$.
\end{theorem}
Formula~\eqref{eq:12} for $P(\boldsymbol{\varphi})$ clearly shows resemblance to the classical~\eqref{eq:16} and weighted~\eqref{eq:17} variational principle, but the differences are also apparent. Most notably, the supremum is taken over each coordinate separately for the different orderings rather than optimising over a single vector on $\mathcal{I}$ with its projections onto the subsets $\mathcal{I}_n^{\sigma}$. The interpretation of the formula is that for each $\sigma\in\mathcal{A}$ there is a \emph{dominant type} which `carries the pressure' for that ordering and determines the polynomial growth rate of $\sum_{B_{\delta}(\ii) \in \mathcal{B}_{\delta}^{\sigma}} \exp \left[\Phi\left(B_{\delta}(\ii)\right)\right]$. This rate is given by the sum in~\eqref{eq:12}, where for each coordinate $1\leq n\leq d$ the constant $C_{n}^{(d), \sigma}(\mathbf{P}_{\!\sigma})$ is related to the length of the block $\left(\Pi_{n}^{\sigma} i_{L_{\delta}(\ii, \sigma_{n+1})+1}, \ldots,\Pi_{n}^{\sigma} i_{L_{\delta}(\ii, \sigma_{n})}\right)$, which is where the restriction of $\mathbf{P}_{\!\sigma} \in \mathcal{Q}^{\sigma}$ comes into play. Furthermore, $H(\mathbf{p}_{\sigma_{n}})$ comes from the number of approximate cubes with this type and the `integral' is the contribution of $\boldsymbol{\varphi}$. Finally, the largest dominant type determines $P(\boldsymbol{\varphi})$.

If $\boldsymbol{\varphi}=\boldsymbol{0}$, i.e. $\varphi_n^{\sigma}\equiv 0$ for every $\sigma\in\mathcal{A}$ and $1\leq n\leq d$, then $P(\boldsymbol{0})$ gives the box dimension of any sponge satisfying the separation of principal projections condition, see Definition~\ref{def:SPPC} and Theorem~\ref{thm:Lqmain}. With another appropriate choice of $\boldsymbol{\varphi}$, see~\eqref{eq:38}, the pressure translates to the `symbolic' $L^q$ spectrum of $\widetilde{\nu}_{\boldsymbol{\mu}}$ which is then related to the actual $L^q$ spectrum of $\nu_{\boldsymbol{\mu}}$ under the same separation condition, see Theorem~\ref{thm:Lqmain}. We can thus see that the big advantage of this approach is that it unifies different arguments of numerous previous results and at the same time generalises them naturally to arbitrary dimensions.

For practical purposes, having a closed form formula for $P(\boldsymbol{\varphi})$ would be preferred over having to characterise the supremum over $\mathcal{Q}^{\sigma}$. We give a closed form which is always an upper bound for $P(\boldsymbol{\varphi})$ and equal to it in some instances. We define real numbers $T_0^{\sigma}\coloneqq 0, T_1^{\sigma},\ldots,T_d^{\sigma}$ recursively, where $T_n^{\sigma}=T_n^{(d),\sigma}(\boldsymbol{\varphi})$ is the unique solution to the equation
\begin{equation}\label{eq:14}
	\sum_{i\in\mathcal{I}_{n}^{\sigma}} \underbrace{ e^{\varphi_{n}^{\sigma}(i)} \prod_{\ell=1}^n \big( \lambda_i^{(\sigma_{\ell})} \big)^{T_{\ell}^{\sigma}-T_{\ell-1}^\sigma} }_{=:\, p_{\sigma_n}^{\ast}(i)} =1,
\end{equation}
and $\mathbf{P}_{\!\sigma}^{\ast}=(\mathbf{p}_{\sigma_d}^{\ast},\ldots,\mathbf{p}_{\sigma_1}^{\ast})\in \mathcal{P}^{\sigma}$, where $\mathbf{p}_{\sigma_n}^{\ast}=(p_{\sigma_n}^{\ast}(i))_{i\in\mathcal{I}_n^{\sigma}}$. 
\begin{prop}\label{prop:main}
For any $\sigma\in\mathcal{A}$ the supremum $\sup _{\mathbf{P}_{\!\sigma} \in \mathcal{P}^{\sigma}} t(\mathbf{P}_{\!\sigma}) = t(\mathbf{P}_{\!\sigma}^{\ast})=T_d^{\sigma}$.
\end{prop}
The proposition is proved in Section~\ref{sec:03}. Theorem~\ref{thm:main1} and Proposition~\ref{prop:main} imply the following. 
\begin{corollary}\label{cor:main}
The upper bound $P(\boldsymbol{\varphi}) \leq \max _{\sigma \in \mathcal{A}}\, T_d^{\sigma}$ holds for all $\boldsymbol{\varphi}$. If $\sigma\in\mathcal{A}$ is such that $\mathbf{P}_{\!\sigma}^{\ast}\in \mathcal{Q}^{\sigma}$, then $\sup _{\mathbf{P}_{\!\sigma} \in \mathcal{Q}^{\sigma}} t(\mathbf{P}_{\!\sigma}) = t(\mathbf{P}_{\!\sigma}^{\ast})=T_d^{\sigma}$. Furthermore, if $\omega\in\mathcal{A}$ is such that $\mathbf{P}_{\!\omega}^{\ast}\in \mathcal{Q}^{\omega}$ and $T_d^{\omega}=\max _{\sigma \in \mathcal{A}}\, T_d^{\sigma}$, then 
\begin{equation}\label{eq:18}
	P(\boldsymbol{\varphi}) = \max _{\sigma \in \mathcal{A}}\, T_d^{\sigma}.
\end{equation}
In particular, if $\mathcal{A}=\{\sigma\}$, then $P(\boldsymbol{\varphi}) = T_d^{\sigma}$. 
\end{corollary}

It is immediate that $\#\mathcal{A}=1$ if and only if there is a $\sigma\in\mathcal{S}_d$ such that 
\begin{equation}\label{eq:21}
	0<\lambda_i^{(\sigma_d)}\leq \lambda_i^{(\sigma_{d-1})}\leq \ldots\leq \lambda_i^{(\sigma_1)}<1 \;\text{ for every } i\in\mathcal{I}.
\end{equation}
In this case we say that the sponge $F$ satisfies the \emph{coordinate ordering condition with ordering $\sigma$}. 

The value of $T_d^{\sigma}$ can be calculated by numerically solving the $d$ equations in~\eqref{eq:14}. If $\mathbf{P}_{\!\sigma}^{\ast}\in \mathcal{Q}^{\sigma}$, then $\mathbf{P}_{\!\sigma}^{\ast}$ is the dominant type for that particular ordering $\sigma\in\mathcal{A}$. However, if $\mathbf{P}_{\!\sigma}^{\ast}\notin \mathcal{Q}^{\sigma}$, then characterising the dominant type is a difficult non-linear optimisation problem with non-linear constraint. It is also not clear how $\sup _{\mathbf{P}_{\!\sigma} \in \mathcal{Q}^{\sigma}} t(\mathbf{P}_{\!\sigma})$ and $\sup _{\mathbf{P}_{\!\omega} \in \mathcal{Q}^{\omega}} t(\mathbf{P}_{\!\omega})$ relate to each other for two different orderings $\sigma,\omega\in\mathcal{A}$. Nevertheless, the dominant type which gives the value of $P(\boldsymbol{\varphi})$ can be thought of as the `equilibrium state' of the system. Getting a better understanding of when $\mathbf{P}_{\!\sigma}^{\ast}\in \mathcal{Q}^{\sigma}$ seems a subtle issue and is a natural direction for further study.

\begin{question}\label{ques:2}
Are there further easy to check sufficient and/or necessary conditions for $\mathbf{P}_{\!\sigma}^{\ast}\in \mathcal{Q}^{\sigma}$? More broadly, when does~\eqref{eq:18} hold? If $\mathbf{P}_{\!\sigma}^{\ast}\notin \mathcal{Q}^{\sigma}$, then is the supremum over $\mathcal{Q}^{\sigma}$ attained on the boundary of $\mathcal{Q}^{\sigma}$ (where $C_n^{(d),\sigma}(\cdot)=0$ for at least one $n\in\{1,\ldots,d\}$)? 
\end{question}

\begin{example}\label{ex:SelfSim}
The self-affine sponge $F$ is a \emph{self-similar set}, if for each $i\in\mathcal{I}$ there is $\lambda_i\in(0,1)$ such that $\lambda_i^{(n)}=\lambda_i$ for all $1\leq n\leq d$. Clearly, $L_{\delta}(\ii,n)=L_{\delta}(\ii,m)$ for all $1\leq  n\leq m\leq d$, so $\mathcal{A}=\{\mathrm{Id}\}$. Let $L_{\delta}(\ii)$ denote this common value. We have $B_{\delta}(\ii) = (i_1,\ldots,i_{L_{\delta}(\ii)})\in\mathcal{I}^{L_{\delta}(\ii)}$ and $\Phi\left(B_{\delta}(\ii)\right)= 
\sum_{\ell=1}^{L_{\delta}(\ii)} \varphi_{d}\left(i_{\ell}\right)$. Moreover, $\chi_{1}(\mathbf{p}_n)=\chi_{2}(\mathbf{p}_n)=\ldots=\chi_{n}(\mathbf{p}_n)$ for all $1\leq n\leq d$ giving $C_d^{(d)}(\mathbf{P})=1/\chi_d(\mathbf{p}_d)$ and $C_n^{(d)}(\mathbf{P})=0$ for all $1\leq n\leq d-1$. As a result, ~\eqref{eq:12} simplifies to
\begin{equation*}
P(\boldsymbol{\varphi}) = \sup _{\mathbf{p} \in \mathcal{P}_d}\, 
\frac{H(\mathbf{p}) + \int\! \varphi_d\,\mathrm{d}\mathbf{p}}{\chi_d(\mathbf{p})}.
\end{equation*}
Also, writing out~\eqref{eq:14} for $n=d$, we obtain
$\sum_{i\in\mathcal{I}} e^{\varphi_{d}(i)} \lambda_i^{T_d} =1$.
If $F$ satisfies the open set condition, i.e. $f_i((0,1)^d)\cap f_j((0,1)^d) = \emptyset$ for all $i\neq j$, then by taking $\boldsymbol{\varphi}=\boldsymbol{0}$, we recover the well-known fact that $\dim_{\mathrm B}F=T_d$, often called the \emph{similarity dimension}, which has the equivalent characterisation of maximising `entropy over Lyapunuv exponent'. For a fixed probability vector $\boldsymbol{\mu}$ on $\mathcal{I}$ and $q\in\R$ if $\varphi_{d}(i)=q\log \mu(i)$, then $T_d=T_d(\boldsymbol{\mu},q)$ is the $L^q$ spectrum of the self-similar measure $\nu_{\boldsymbol{\mu}}$.  
\end{example}

\subsubsection*{Main idea of proof} The key observation is that $\Phi(B_{\delta}(\ii))$ does not depend directly on the order of symbols in the symbolic representation~\eqref{eq:11} of $B_{\delta}(\ii)$, but rather just on the number of times a particular symbol $i\in\mathcal{I}_n^{\sigma}$ appears in the block $\left(\Pi_{n}^{\sigma} i_{L_{\delta}(\ii, \sigma_{n+1})+1}, \ldots,\Pi_{n}^{\sigma} i_{L_{\delta}(\ii, \sigma_{n})}\right)$. Therefore, we use digit frequencies to express $\Phi(B_{\delta}(\ii))$ and the `method of types' to count the number of different approximate cubes with given digit frequencies. As $\delta\to 0$, the set of different types becomes dense in the parameter space $\mathcal{Q}^{\sigma}$ ($\sigma\in\mathcal{A}$), however, the rate of growth of the number of different types is significantly smaller compared to the cardinality of a type. Hence, there is a type which `carries the pressure' at each scale $\delta$ and these types converge to the dominant type given by the variational principle~\eqref{eq:12}. While the general scheme is certainly not new in the dimension theory of dynamical systems, we are unaware of such a streamlined application in the context of determining box counting quantities.

\section{Application to the \texorpdfstring{$L^q$}{Lq} spectrum of self-affine sponges}\label{sec:02}

We begin by introducing the separation condition required for most of our bounds.
\begin{definition}\label{def:SPPC}
A self-affine sponge $F\subset [0,1]^d$ satisfies the \emph{separation of principal projections condition} (SPPC) if for every $\sigma\in\mathcal{A}$, $1\leq n\leq d$ and $i,j\in\mathcal{I}$,
\begin{equation}\label{eq:20}
	\text{either } f_i \text{ and } f_j \text { overlap exactly on } E_n^{\sigma} \text{ or } \Pi_n^{\sigma}\big(f_i((0,1)^d)\big)\cap \Pi_n^{\sigma}\big(f_j((0,1)^d)\big) = \emptyset.
\end{equation}
The sponge satisfies the very strong SPPC if $(0,1)^d$ can be replaced with $[0,1]^d$ in~\eqref{eq:20}.
\end{definition}
If~\eqref{eq:20} is only assumed for $n=d$, the rather weaker condition is known as the \emph{rectangular open set condition} in~\cite{FengWang2005, Fraser12Boxlike, Fraser_TAMS2016}. In particular, if $F$ is a self-similar set, recall Example~\ref{ex:SelfSim}, then the SPPC is equivalent to assuming~\eqref{eq:20} only for $n=d$. The SPPC was introduced simultaneously in~\cite{FK_DimAMeasure_arxiv} where the Assouad and lower dimensions of the self-affine measure $\nu_{\boldsymbol{\mu}}$ were studied. In that case assuming the very strong SPPC is necessary while for all results in this paper the SPPC suffices. 

\begin{example}\label{ex:HigherDImCarpets}
The following are the natural generalisations of Bara\'nski~\cite{BARANSKIcarpet_2007}, Lalley--Gatzouras~\cite{GatzourasLalley92} and Bedford--McMullen~\cite{Bedford84_phd, mcmullen84} carpets to higher dimensions. Assume that $0<a_i^{(n)}<1$ for all $1\leq n\leq d$ and $i\in\mathcal{I}$.
\begin{enumerate}
\item A \emph{Bara\'nski sponge} $F\subset[0,1]^d$ satisfies that for all $\sigma\in\mathcal{S}_d$,
\begin{equation*}
	\text{either } f_i \text{ and } f_j \text {overlap exactly on } E_1^{\sigma} \text{ or } \Pi_1^{\sigma}\big(f_i((0,1)^d)\big)\cap \Pi_1^{\sigma}\big(f_j((0,1)^d)\big) = \emptyset.
\end{equation*}
In other words, the IFSs generated on the coordinate axes by indices $\mathcal{I}_1^{\sigma}$ satisfy the open set condition. This clearly implies the SPPC.
\item A \emph{Lalley--Gatzouras sponge} $F\subset[0,1]^d$ satisfies the SPPC and the coordinate ordering condition~\eqref{eq:21} for some $\sigma\in\mathcal{S}_d$, hence, $\mathcal{A}=\{\sigma\}$.
\item A \emph{Bedford--McMullen sponge} $F\subset[0,1]^d$ is a Bara\'nski sponge which satisfies the coordinate ordering condition (hence, is also a Lalley--Gatzouras sponge) and
\begin{equation*}
	\lambda_1^{(n)} = \lambda_2^{(n)} =\ldots = \lambda_N^{(n)} \;\text{ for all } 1\leq n\leq d.
\end{equation*}
\end{enumerate}
\end{example}

On the plane either $\#\mathcal{A}=1$ or $\#\mathcal{A}=2$, hence, the SPPC combines Lalley--Gatzouras (when $\#\mathcal{A}=1$) and (genuine) Bara\'nski carpets (when $\#\mathcal{A}=2$) into a unified framework in a natural way. Moreover, for dimensions $d\geq 3$ it is a wider class of sponges than simply the union of the Bara\'nski and Lalley--Gatzouras class. We give one example here and refer the interested reader to~\cite[Section~4]{FK_DimAMeasure_arxiv} for a complete characterisation of the sponges satisfying the SPPC in three dimensions. Assume for all $i\in\mathcal{I}$ that $0<\max\{ a_i^{(y)}, a_i^{(z)}\} < a_i^{(x)}<1$ and there exist $j,k\in\mathcal{I}$ such that $a_j^{(y)}< a_j^{(z)}$ and $a_k^{(y)}> a_k^{(z)}$. In this case it is easy to see that $\mathcal{A}=\{(x,y,z),(x,z,y)\}$, moreover, the projection onto both the $xy$ and $xz$-plane is a Lalley--Gatzouras carpet with $x$ being the dominant side. Projection onto $yz$-plane does not play a role.

\subsection{Results for \texorpdfstring{$L^q$}{Lq} spectrum}\label{sec:ResLq}

We define the family of potentials which leads us to the $L^q$ spectrum of self-affine measures. Let $\boldsymbol{\mu}=\big( \mu(i)\big)_{i\in\mathcal{I}}$ be a probability vector on $\mathcal{I}$ with strictly positive entries. For $\sigma\in\mathcal{A}$ and $1\leq n\leq d$, we define its `projection' to $\mathcal{I}_n^{\sigma}$ to be
\begin{equation}\label{eq:300}
\boldsymbol{\mu}_n^{\sigma}\coloneqq \big( \mu_n^{\sigma}(i) \big)_{i\in\mathcal{I}_n^{\sigma}}, \;\text{ where } \mu_n^{\sigma}(i)\coloneqq \sum_{j\in\mathcal{I}:\, \Pi_n^{\sigma}j=i}\mu(j).
\end{equation}
Hence, $\boldsymbol{\mu}_n^{\sigma}\in\mathcal{P}_n^{\sigma}$. For $q\in\mathbb{R}$, we introduce the family of potentials
\begin{equation}\label{eq:38}
\boldsymbol{\psi}_q^{\boldsymbol{\mu}} \coloneqq \{ \psi_{q,n}^{\boldsymbol{\mu},\sigma}:\, \mathcal{I}_n^{\sigma}\to \mathbb{R}\, |\, \sigma\in\mathcal{A},\, 1\leq n\leq d\}, \;\text{ where } \psi_{q,n}^{\boldsymbol{\mu},\sigma}(i) \coloneqq q\cdot \log \mu_n^{\sigma}(i).
\end{equation} 
It follows from Theorem~\ref{thm:main1} that the limit $P(\boldsymbol{\psi}_q^{\boldsymbol{\mu}})$ exists for all $\boldsymbol{\mu}$ and $q\in\mathbb{R}$. We prove in Lemma~\ref{lem:MeasureSymbCube} that with this choice $\exp \left[\Phi\left(B_{\delta}(\ii)\right)\right]=\big(\widetilde{\nu}_{\boldsymbol{\mu}}\left(B_{\delta}(\ii)\right)\big)^q$ for any approximate cube. Translating this to the $L^q$ spectrum of $\nu_{\boldsymbol{\mu}}$ leads us to our main result.

\begin{theorem}\label{thm:Lqmain}
Let $\nu_{\boldsymbol{\mu}}$ be a self-affine measure on the self-affine sponge $F$ which satisfies the SPPC. Then
\begin{equation*}
T(\nu_{\boldsymbol{\mu}},q) =  P(\boldsymbol{\psi}_q^{\boldsymbol{\mu}}) \quad\text{ for all } q\in\mathbb{R}.
\end{equation*}
In particular, the box dimension of $F$ exists and $\dim_{\mathrm B}F = \dim_{\mathrm P}F =P(\boldsymbol{\psi}_0^{\boldsymbol{\mu}})$.
\end{theorem}
\begin{remark}\label{rem:upperbound}
Observe from~\eqref{eq:38} that $\boldsymbol{\psi}_0^{\boldsymbol{\mu}}$ is independent of the choice of $\boldsymbol{\mu}$. The box and packing dimensions are equal because $F$ is compact and every open set intersecting $F$ contains a bi-Lipschitz image of $F$, see~\cite[Corollary 3.9]{FalconerBook}.

The theorem also gives a clear indication of how the $L^q$ spectrum can be non-differentiable at a point $\hat{q}$: the maximum in~\eqref{eq:12} is attained for a different $\sigma\in\mathcal{A}$ when $q\to\hat{q}^-$ than when $q\to\hat{q}^+$.  
\end{remark}

Adapting~\eqref{eq:14}, we define functions $T_n^{\boldsymbol{\mu},\sigma}(q):\mathbb{R}\to\mathbb{R}$ for $0\leq n\leq d$ recursively, by first setting $T_0^{\boldsymbol{\mu},\sigma}(q)\equiv 0$ and then defining  $T_n^{\boldsymbol{\mu},\sigma}(q)$ to be the unique solution to the equation
\begin{equation}\label{eq:31}
\sum_{i\in\mathcal{I}_{n}^{\sigma}} (\mu_n^{\sigma}(i))^q \prod_{\ell=1}^n \big( \lambda_i^{(\sigma_{\ell})} \big)^{T_{\ell}^{\boldsymbol{\mu},\sigma}(q)- T_{\ell-1}^{\boldsymbol{\mu},\sigma}(q)} =1.
\end{equation}
Combining Theorem~\ref{thm:Lqmain} and Corollary~\ref{cor:main} immediately give the following two statements.

\begin{corollary}\label{cor:Lqmain}
Let $\nu_{\boldsymbol{\mu}}$ be a self-affine measure on the self-affine sponge $F\subset \mathbb{R}^d$ that satisfies the SPPC. Then 
\begin{equation*}
T(\nu_{\boldsymbol{\mu}},q) \leq \max _{\sigma \in \mathcal{A}}\, T_d^{\boldsymbol{\mu},\sigma}(q) \quad\text{ for all } q\in\mathbb{R}.
\end{equation*}
\end{corollary}
A better understanding of Question~\ref{ques:2} would have direct implications on when $T(\nu_{\boldsymbol{\mu}},q) = \max _{\sigma \in \mathcal{A}}\, T_d^{\boldsymbol{\mu},\sigma}(q)$. Nevertheless, our results for the Lalley--Gatzouras class are more complete.
\begin{corollary}\label{cor:LGsponge}
If $F$ is a $\sigma$-ordered Lalley--Gatzouras sponge, then $T(\nu_{\boldsymbol{\mu}},q) = T_d^{\boldsymbol{\mu},\sigma}(q)$ for all $q\in\mathbb{R}$. Since $T_d^{\boldsymbol{\mu},\sigma}(q)$ is differentiable everywhere, the result from~\cite{Ngai_97PAMS} yields that
\begin{equation*}
	\dim_{\mathrm H} \nu_{\boldsymbol{\mu}} = \dim_{\mathrm P} \nu_{\boldsymbol{\mu}} = \dim_{\mathrm e} \nu_{\boldsymbol{\mu}} = -T'(\nu_{\boldsymbol{\mu}},1).
\end{equation*}
\end{corollary}
Implicit differentiation of $T_d^{\boldsymbol{\mu},\sigma}(q)$ gives the value of $T'(\nu_{\boldsymbol{\mu}},1)$. Theorem~\ref{thm:Lqmain} is proved in Section~\ref{sec:proofLq}.

\subsection{Box and Frostman dimension of self-affine measures}\label{sec:BoxFrostmanRes}

Given $\boldsymbol{\mu}$ and $\sigma\in\mathcal{A}$, we define two sequences of numbers $\overline{S}_0^{\boldsymbol{\mu},\sigma}\coloneqq0, \overline{S}_1^{\boldsymbol{\mu},\sigma},\ldots, \overline{S}_d^{\boldsymbol{\mu},\sigma}$ and $\underline{S}_0^{\boldsymbol{\mu},\sigma}\coloneqq0, \underline{S}_1^{\boldsymbol{\mu},\sigma},\ldots, \underline{S}_d^{\boldsymbol{\mu},\sigma}$ by
\begin{equation}\label{eq:39}
	\overline{S}_n^{\boldsymbol{\mu},\sigma} \coloneqq \overline{S}_{n-1}^{\boldsymbol{\mu},\sigma} +  \max_{i\in\mathcal{I}_n^{\sigma}} \frac{1}{\log \lambda_i^{(\sigma_n)}} \bigg( \log \mu_n^{\sigma}(i)  +  \sum_{m=1}^{n-1}\big(\overline{S}_{m-1}^{\boldsymbol{\mu},\sigma} - \overline{S}_{m}^{\boldsymbol{\mu},\sigma}\big) \log \big( \lambda_i^{(\sigma_{m})} \big) \bigg),
\end{equation}
and
\begin{equation}\label{eq:390}
	\underline{S}_n^{\boldsymbol{\mu},\sigma} \coloneqq \underline{S}_{n-1}^{\boldsymbol{\mu},\sigma} +  \min_{i\in\mathcal{I}_n^{\sigma}} \frac{1}{\log \lambda_i^{(\sigma_n)}} \bigg( \log \mu_n^{\sigma}(i)  +  \sum_{m=1}^{n-1}\big(\underline{S}_{m-1}^{\boldsymbol{\mu},\sigma} - \underline{S}_{m}^{\boldsymbol{\mu},\sigma}\big) \log \big( \lambda_i^{(\sigma_{m})} \big) \bigg),
\end{equation}
where the empty sum equals $0$ in case $n=1$. Let $\overline{k}_n^{\sigma}\in\mathcal{I}_n^{\sigma}$ denote any of the symbols which attain the maximum in~\eqref{eq:39} and $\underline{k}_n^{\sigma}\in\mathcal{I}_n^{\sigma}$ be any of the symbols which attain the minimum in~\eqref{eq:390}. Also let $\overline{\mathbf{K}}_{\sigma}\coloneqq (\overline{\mathbf{k}}_{\sigma_d},\ldots,\overline{\mathbf{k}}_{\sigma_1})$ and $\underline{\mathbf{K}}_{\sigma}\coloneqq (\underline{\mathbf{k}}_{\sigma_d},\ldots,\underline{\mathbf{k}}_{\sigma_1})$, where $\overline{\mathbf{k}}_{\sigma_n}$ denotes the degenerate probability vector on $\mathcal{I}_n^{\sigma}$ which puts all mass on $\overline{k}_n^\sigma$ and similarly $\underline{\mathbf{k}}_{\sigma_n}$ puts mass 1 on $\underline{k}_n^\sigma$. The quantity of interest now is
\begin{equation*}
S_{\sigma}(\mathbf{P}_{\!\sigma})=S(\mathbf{P}_{\!\sigma}) = S(\mathbf{p}_{\sigma_d};\ldots;\mathbf{p}_{\sigma_1}) \coloneqq -\sum_{n=1}^{d} 
C_{n}^{(d), \sigma}(\mathbf{P}_{\!\sigma}) \cdot \int\! \log \boldsymbol{\mu}_n^{\sigma}\,\mathrm{d}\mathbf{p}_{\sigma_n},
\end{equation*}
where $\int\! \log \boldsymbol{\mu}_n^{\sigma}\,\mathrm{d}\mathbf{p}_{\sigma_n} = \sum_{i \in \mathcal{I}_{n}^{\sigma}} p_{\sigma_n}(i)\cdot \log \mu_n^{\sigma}(i)$.
\begin{prop}\label{prop:maxminexponent}
For any $\sigma\in\mathcal{A}$ the supremum $\sup _{\mathbf{P}_{\!\sigma} \in \mathcal{P}^{\sigma}} S(\mathbf{P}_{\!\sigma}) = S(\overline{\mathbf{K}}_{\sigma})=\overline{S}_d^{\boldsymbol{\mu},\sigma}$ and the infimum $\inf_{\mathbf{P}_{\!\sigma} \in \mathcal{P}^{\sigma}} S(\mathbf{P}_{\!\sigma}) = S(\underline{\mathbf{K}}_{\sigma})=\underline{S}_d^{\boldsymbol{\mu},\sigma}$.
\end{prop}
\begin{theorem}\label{thm:dimBFmeasure}
Let $\nu_{\boldsymbol{\mu}}$ be a self-affine measure on the sponge $F\subset \mathbb{R}^d$ that satisfies the SPPC. Then
\begin{equation*}
\dim_{\mathrm F} \nu_{\boldsymbol{\mu}} = \min_{\sigma \in \mathcal{A}}\, \inf_{\mathbf{P}_{\! \sigma}\in\mathcal{Q}^{\sigma}} S(\mathbf{P}_{\!\sigma}) = \lim_{q\to +\infty} \frac{T(\nu_{\boldsymbol{\mu}},q)}{-q} \geq \max\Big\{ 0, \min_{\sigma \in \mathcal{A}} \underline{S}_d^{\boldsymbol{\mu},\sigma} \Big\},
\end{equation*}
and
\begin{equation*}
\dim_{\mathrm B} \nu_{\boldsymbol{\mu}} = \max_{\sigma \in \mathcal{A}}\, \sup_{\mathbf{P}_{\! \sigma}\in\mathcal{Q}^{\sigma}} S(\mathbf{P}_{\!\sigma}) = \lim_{q\to -\infty} \frac{T(\nu_{\boldsymbol{\mu}},q)}{-q} \leq \max_{\sigma \in \mathcal{A}} \overline{S}_d^{\boldsymbol{\mu},\sigma}.
\end{equation*}
In particular, if $F$ is a $\sigma$-ordered Lalley--Gatzouras sponge, then $\dim_{\mathrm F} \nu_{\boldsymbol{\mu}} = \underline{S}_d^{\boldsymbol{\mu},\sigma}$ and $\dim_{\mathrm B} \nu_{\boldsymbol{\mu}} = \overline{S}_d^{\boldsymbol{\mu},\sigma}$.
\end{theorem}
Proposition~\ref{prop:maxminexponent} is proved in Section~\ref{sec:03} and Theorem~\ref{thm:dimBFmeasure} in Section~\ref{sec:proofFrostmanBox}.

\section{Discussion and two worked out examples}\label{sec:discussion}

In this section, we give further context to our results by relating it to previous papers and demonstrate on two worked out examples how our approach tackles problems where earlier ones fell short. These examples can also help the reader get more comfortable with our notation. 

The closest related work is due to Olsen~\cite{Olsen_PJM98}, who amongst other things, calculated the $L^q$ spectrum of $\nu_{\boldsymbol{\mu}}$ supported on Bedford--McMullen sponges and also both asymptotes of the spectrum. To the best of our knowledge this is the only result in the non-conformal higher dimensional setting which additionally even handles the $q<0$ case. One slight drawback is that it assumes the VSSC which is equivalent to the very strong SPPC in our setting. Our approach allows us to weaken the separation condition to the SPPC while still obtaining the $L^q$ spectrum for the whole range of $q\in\mathbb{R}$, also $\dim_{\mathrm F} \nu_{\boldsymbol{\mu}}$ and $\dim_{\mathrm B} \nu_{\boldsymbol{\mu}}$ for a substantially larger class of sponges.   

Existing results for the $L^q$ spectrum on the plane restrict to $q\geq 0$ but allow for box-like sets outside the class of Lalley--Gatzouras and Bara\'nski carpets~\cite{FengWang2005, Fraser_TAMS2016} even with non-linear maps~\cite{falconer_fraser_lee_2021}. This is due in part to the fact that the $L^q$ spectrum of self-conformal IFSs on the line is known to exist~\cite{PeresSolomyak_Indiana00}, hence, the formulas on the plane can at least be stated depending on the $L^q$ spectrum of the projections onto the two coordinate axes. Assuming the SPPC, Theorem~\ref{thm:Lqmain} recovers the variational formula proved by Feng and Wang~\cite{FengWang2005}.  It follows from Fraser's work~\cite[Theorem 2.10 and 2.12]{Fraser_TAMS2016} that assuming the SPPC on the plane $T(\nu_{\boldsymbol{\mu}},q) = \max _{\sigma \in \mathcal{A}}\, T_d^{\boldsymbol{\mu},\sigma}(q)$ for $q\in(0,1]$ and $T(\nu_{\boldsymbol{\mu}},q)$ is differentiable at $q=1$. Uncovering the connection between the variational formula and the closed form expression is closely connected to Question~\ref{ques:2}. Already on the plane, this closed form expression need not hold for $q>1$ as was shown by the example presented in~\cite[Theorem 3.8]{FraserMorrisetal_Nonlin21} which we revisit in Section~\ref{sec:planarBaranski}.
\begin{question}\label{ques:1}
Is it true in higher dimensions as well that for a self-affine measure supported on a sponge satisfying the SPPC there is an interval of $q$ for which $T(\nu_{\boldsymbol{\mu}},q) = \max _{\sigma \in \mathcal{A}}\, T_d^{\boldsymbol{\mu},\sigma}(q)$? If so, does the interval include $q=0$? Is $T(\nu_{\boldsymbol{\mu}},q)$ always differentiable at $q=1$?
\end{question}

In case of the box dimension, Kenyon and Peres~\cite{KenyonPeres_ETDS96} calculated it for Bedford--McMullen sponges. The Lalley--Gatzouras class in arbitrary dimensions was also handled independently from our work in~\cite{Raoetal_BoxCountingMeasure_arxiv}. Recently, Fraser and Jurga~\cite{fraserJurga2021AdvMath} considered sponges in $d=3$ in the more general setting where each diagonal matrix can be composed with a permutation matrix. Amongst sponges which satisfy the SPPC, their main result only covers the Lalley--Gatzouras class. More importantly, they present an example in~\cite[Theorem 5.5]{fraserJurga2021AdvMath} which shows that their bounds are not applicable in general to the Bara\'nski class. In Section~\ref{sec:3DBaranski} we calculate the box dimension of this sponge and show the qualitative difference of our pressure compared to the one in~\cite{fraserJurga2021AdvMath}. Feng and Hu~\cite[Theorem 2.15]{FengHu09} considered diagonal systems with equal matrices.

Existing results on the plane go well beyond the SPPC, though it is still an open folklore conjecture that the box dimension of any self-affine set exists regardless of overlaps. It does not exist for all \emph{sub}-self-affine sets introduced in~\cite{KaenmakiVilppolainen_2010}, see the recent example of Jurga~\cite{Jurga_DimBsubselfaffine_arxiv}. Carpets satisfying the rectangular open set condition are covered in~\cite{FengWang2005, Fraser_TAMS2016}, so it would be particularly interesting to look at diagonal (and anti-diagonal) systems with overlaps. There has been some progress in this direction~\cite{fraser_shmerkin_2016,KolSimon_TriagGL2019,pardo-simon}, where the authors consider a carpet satisfying the SPPC and then shift complete rows and/or columns and give sufficient conditions under which $\dim_{\mathrm{B}}F$ does not drop, i.e. $\dim_{\mathrm{B}}F=\max_{\sigma \in \mathcal{A}}T_2^{\sigma}$. Assuming the SPPC to begin with ensures that $T_1^{\sigma}\leq 1$ and $T_2^{\sigma}\leq 2$ also for the shifted system. It makes sense to define $T_1^{\sigma}$ and $T_2^{\sigma}$ for general diagonal systems using the projections of the first level cylinders to the $x$ and $y$ coordinate axis. If $T_1^{\sigma}>1$ then it is appropriate to adjust the definition of $T_2^{\sigma}$ to the solution of the equation 
\begin{equation*}
\sum_{i\in\mathcal{I}}  \big( \lambda_i^{(\sigma_{1})} \big)^{\min\{ T_{1}^{\sigma},1\}} \big( \lambda_i^{(\sigma_{2})} \big)^{T_{2}^{\sigma}- \min \{ T_{1}^{\sigma},1\}} =1.
\end{equation*}
\begin{question}
Given an arbitrary diagonal system on the plane, under what overlapping conditions is it true that $\dim_{\mathrm{B}}F=\min\{ \max_{\sigma \in \mathcal{A}}T_2^{\sigma}, 2 \}$? Is it sufficient to assume the exponential separation condition introduced in~\cite{Hochman_Annals14} for both projected IFSs?
\end{question}

B\'ar\'any, Rams and Simon~\cite[Theorem B]{BRS_2016PAMS} partially answered the second question in the affirmative. Their result does not cover the case when $\min_{\sigma\in\mathcal{A}}\{T_1^{\sigma}\}>1$, in which case it is reasonable to suspect that the box dimension is equal to the affinity dimension introduced in~\cite{falconer_1988}.

The variational formula sheds some light on the differences between Hausdorff and box dimension. To illustrate this, consider the class of $\sigma$-ordered Lalley-Gatzouras sponges with $\mathbf{P}_{\!\sigma}=(\mathbf{p}_{\sigma_d},\ldots,\mathbf{p}_{\sigma_1})\in\mathcal{P}^{\sigma}$ such that $\mathbf{p}_{\sigma_n}= (\mathbf{p}_{\sigma_d})_{n}^{\sigma}$, i.e. $\mathbf{p}_{\sigma_n}$ is just the `projection' of $\mathbf{p}_{\sigma_d}$ onto $\mathcal{I}_n^{\sigma}$ defined in~\eqref{eq:300}. A simple induction argument shows that in this case $C_{n}^{(d), \sigma}(\mathbf{P}_{\!\sigma}) = 1/\chi_{n}^{\sigma}(\mathbf{p}_{\sigma_d})-1/\chi_{n+1}^{\sigma}(\mathbf{p}_{\sigma_d})\geq 0$ (due to the coordinate ordering property), hence,
\begin{equation*}
t(\mathbf{P}_{\!\sigma}) = \sum_{n=1}^{d} \bigg( 
\frac{1}{\chi_{n}^{\sigma}(\mathbf{p}_{\sigma_d})}-\frac{1}{\chi_{n+1}^{\sigma}(\mathbf{p}_{\sigma_d})} \bigg) \cdot H(\mathbf{p}_{\sigma_{n}}) = 
\sum_{n=1}^{d}  
\frac{H(\mathbf{p}_{\sigma_{n}})-H(\mathbf{p}_{\sigma_{n-1}})}{\chi_{n}^{\sigma}(\mathbf{p}_{\sigma_d})},
\end{equation*}
where $H(\mathbf{p}_{\sigma_{0}})\coloneqq 0$. By~\cite[Theorem 1.3]{feng2020LYformulaArxiv}, this is precisely the Hausdorff dimension of the self-affine measure $\nu_{\mathbf{p}_{\sigma_d}}$. This \emph{Ledrappier--Young formula} holds in much higher generality for measures on self-affine sets~\cite{BARANYAnti201788,BaranyRams_TransAMS18,feng2020LYformulaArxiv,FengHu09} and has been a key technical tool in recent advancements in the dimension theory of self-affine sets and measures, see \cite{BaranyHochmanRapaport,hochmanRapaport2021,MorrisShmerkin_TransAMS19,Rapaport_TransAMS18} to name a few.

In particular, Lalley and Gatzouras~\cite{GatzourasLalley92} proved on the plane the variational formula
\begin{equation}\label{eq:50}
\dim_{\mathrm H}F = \sup_{\mathbf{p}_{\sigma_2}\in\mathcal{P}_{2}^{\sigma}}\; t(\mathbf{p}_{\sigma_2};\, (\mathbf{p}_{\sigma_2})_1^{\sigma} ),
\end{equation}
which is attained by a unique choice of $\mathbf{p}_{\sigma_2}$. This is to be compared with 
\begin{equation*}
	\dim_{\mathrm B}F = \max_{(\mathbf{p}_{\sigma_2};\mathbf{p}_{\sigma_1})\in\mathcal{P}_2^{\sigma}\times\mathcal{P}_1^{\sigma}}\; t(\mathbf{p}_{\sigma_2};\, \mathbf{p}_{\sigma_1} ),
\end{equation*}
where the maximum is uniquely attained by $(\mathbf{p}_{\sigma_2}^{\ast};\mathbf{p}_{\sigma_1}^{\ast})$. Therefore, we see that
\begin{equation*}
	\dim_{\mathrm H} F = \dim_{\mathrm B} F \;\Longleftrightarrow\; (\mathbf{p}_{\sigma_2}^*)_{1}^{\sigma} = \mathbf{p}_{\sigma_1}^* \;\Longleftrightarrow\; \sum_{j\in\mathcal{I}_2:\, \Pi_1^{\sigma}j=i} \big(\lambda_j^{(\sigma_2)}\big)^{T_2^{\sigma}-T_1^{\sigma}} =1  \;\text{ for every } i\in\mathcal{I}_1^{\sigma}.
\end{equation*}
This is referred to as the \emph{uniform fibre} case in the literature. In stark contrast, the main result of Das and Simmons~\cite{das2017hausdorff} is that the analogue of the variational formula~\eqref{eq:50} does not necessarily hold in higher dimensions for shift invariant measures. In fact, the example they provide is a Lalley--Gatzouras sponge in $\R^3$. Instead, one needs to consider a wider class of measures, called \emph{pseudo-Bernoulli measures}, which are not invariant to obtain a similar variational principle. Our variational principle~\eqref{eq:12} can be thought of as a Ledrappier--Young like formula for box counting quantities on sponges satisfying the SPPC which holds regardless of the dimension. 
\begin{question}
	Does a Ledrappier--Young like formula hold more generally for the box dimension of self-affine sets on the plane? What about higher dimensions?
\end{question}

For $d=3$, suppressing $\sigma$ from the notation, the expression to be maximised for $\dim_{\mathrm B}F$ is
\begin{equation*}
	\frac{H(\mathbf{p}_3)}{\chi_3(\mathbf{p}_3)} + \left( 1-\frac{\chi_2(\mathbf{p}_3)}{\chi_3(\mathbf{p}_3)} \right)\frac{H(\mathbf{p}_2)}{\chi_2(\mathbf{p}_2)} + \left[ 1-\frac{\chi_1(\mathbf{p}_3)}{\chi_3(\mathbf{p}_3)} - \left( 1-\frac{\chi_2(\mathbf{p}_3)}{\chi_3(\mathbf{p}_3)} \right) \frac{\chi_1(\mathbf{p}_2)}{\chi_2(\mathbf{p}_2)}\right]\frac{H(\mathbf{p}_1)}{\chi_1(\mathbf{p}_1)},
\end{equation*}
over the vectors $(\mathbf{p}_3;\mathbf{p}_2;\mathbf{p}_1)\in\mathcal{P}_3\times\mathcal{P}_2\times\mathcal{P}_1$. The maximum is uniquely attained by $(\mathbf{p}_3^{\ast};\mathbf{p}_2^{\ast};\mathbf{p}_1^{\ast})$. The constants $C_{n}^{(d)}(\mathbf{P})$ can be similarly expressed in terms of Lyapunov exponents for $d>3$, however, the calculations get increasingly involved and cumbersome.

\subsection{A planar Bara\'nski carpet}\label{sec:planarBaranski}

In~\cite[Theorem 3.8]{FraserMorrisetal_Nonlin21} the authors considered a family of Bara\'nski carpets on the plane given by the two maps
\begin{equation*}
	f_{1}(x, y)=\begin{pmatrix}
		c & 0 \\
		0 & d
	\end{pmatrix}\begin{pmatrix}
		x \\
		y
	\end{pmatrix} \quad \text { and } \quad f_{2}(x, y)=\begin{pmatrix}
		d & 0 \\
		0 & c
	\end{pmatrix}\begin{pmatrix}
		x \\
		y
	\end{pmatrix}+\begin{pmatrix}
		1-d \\
		1-c
	\end{pmatrix}
\end{equation*}
with $c>d>0$ and $c+d\leq 1$. Let $\sigma=(1,2)$ and $\omega=(2,1)$ denote the two orderings on the plane. The maps are arranged so that $\mathcal{I}_1^{\sigma}=\mathcal{I}_2^{\sigma}=\mathcal{I}_1^{\omega}=\mathcal{I}_2^{\omega}=\mathcal{I}=\{1,2\}$. Thus, for any $\boldsymbol{\mu}=(u,1-u)$,  definition~\eqref{eq:31} gives that $T_1^{\boldsymbol{\mu},\sigma}(q) = T_2^{\boldsymbol{\mu},\sigma}(q)$ and $T_1^{\boldsymbol{\mu},\omega}(q) = T_2^{\boldsymbol{\mu},\omega}(q)$. Let $T_u^{\sigma}(q)$ and $T_u^{\omega}(q)$ denote these two values, respectively. See~\eqref{eq:44} for the explicit formula. If $u=1/2$, then symmetry of the system implies that $T_{1/2}^{\sigma}(q)=T_{1/2}^{\omega}(q)$. The authors of~\cite{FraserMorrisetal_Nonlin21} showed for this particular $\boldsymbol{\mu}=(1/2,1/2)$ that $T(\nu_{\boldsymbol{\mu}},q)\leq g(q)< T_{1/2}^{\sigma}(q)$ for all $q>1$, where $g(q)$ is given by~\cite[eq. (3.2)]{FraserMorrisetal_Nonlin21}. Moreover, $T(\nu_{\boldsymbol{\mu}},q)$ is differentiable at $q=1$, but not analytic in any neighbourhood of $q=1$. They ask~\cite[Question 3.10]{FraserMorrisetal_Nonlin21} how many derivatives does $T(\nu_{\boldsymbol{\mu}},q)$ have at $q=1$ for $\boldsymbol{\mu}=(1/2,1/2)$? We answer this now by giving an explicit formula for $T(\nu_{\boldsymbol{\mu}},q)$.

On one hand, we simplify their example by choosing $c=1/2$ and $d=1/4$ in order to make all calculations completely explicit. On the other hand, we handle all $\boldsymbol{\mu}=(u,1-u)$ in order to uncover an interesting phase transition by varying the parameter $u$. Due to symmetry, we assume without loss of generality that $u\in[1/2,1)$. Define $s$ to be the unique solution of $(1/2)^s+(1/4)^s=1$, i.e. $s=\log\big( (\sqrt{5}-1)/2\big) / \log(1/2)$.

\begin{prop}
The $L^q$ spectrum of the Bara\'nski carpet defined above is given by the following formula:
\begin{itemize}
\item if $u\in\big[\frac{1}{2}, \frac{1}{2^s}\big)$, then
\begin{equation*}
T(\nu_{\boldsymbol{\mu}},q) = 
\begin{cases*}
T_u^{\omega}(q) & if\, $q\leq 0$,\\
T_u^{\sigma}(q) & if\, $0<q\leq \frac{\log 2}{\log \frac{1-u}{u^2}}$,\\
\frac{2}{3}+\frac{\log( u(1-u))}{3 \log 2}\, q & if\, $q> \frac{\log 2}{\log \frac{1-u}{u^2}}$;
\end{cases*}
\end{equation*}
\item if $u\in \big[\frac{1}{2^s},1\big)$, then
\begin{equation*}
T(\nu_{\boldsymbol{\mu}},q) = 
\begin{cases*}
	T_u^{\omega}(q) & if\, $q\leq 0$,\\
	T_u^{\sigma}(q) & if\, $q>0$.
\end{cases*}
\end{equation*}
\end{itemize}
\end{prop} 
There is a point of non-differentiability at $q=0$ for every value of $u$. Moreover, if $u\in\big[\frac{1}{2}, \frac{1}{2^s}\big)$, then there is a further point of interest at $q=\frac{\log 2}{\log \frac{1-u}{u^2}}$, where $T(\nu_{\boldsymbol{\mu}},q)$ is differentiable but no further derivative exists. This answers~\cite[Question 3.10]{FraserMorrisetal_Nonlin21}. As $u\to (1/2)^s$, this phase transition ``escapes'' to $\infty$, explaining why it ``disappears'' for $u\geq (1/2)^s$.
\begin{proof}
Applying definition~\eqref{eq:31}, the function $T_u^{\sigma}(q)$ satisfies the equation
\begin{equation*}
u^q\cdot \Big(\frac{1}{2}\Big)^{T_u^{\sigma}(q)} + (1-u)^q\cdot \Big(\frac{1}{2}\Big)^{2T_u^{\sigma}(q)} = 1,
\end{equation*} 
from which after algebraic manipulations one obtains the explicit formula
\begin{equation}\label{eq:44}
T_u^{\sigma}(q) = \frac{-1}{\log 2} \left( q\cdot \log \Big(\frac{u}{1-u}\Big) + \log\left( \frac{1}{2}\sqrt{1+4\Big(\frac{1-u}{u^2}\Big)^q}  - \frac{1}{2}\right)\right). 
\end{equation} 
Moreover, $T_u^{\omega}(q) =T_{1-u}^{\sigma}(q)$. Some tedious calculations show that
\begin{equation*}
T_u^{\omega}(q)  \leq T_{u}^{\sigma}(q) \;\Longleftrightarrow\; q\in[0,1], \text{ with } T_u^{\omega}(q) = T_{u}^{\sigma}(q) \;\Longleftrightarrow\; q\in\{0,1\}.
\end{equation*}
The dominant types $\mathbf{P}_{\!\sigma,u}^{\ast}=(\mathbf{p}_{\sigma_2}^{\ast},\mathbf{p}_{\sigma_1}^{\ast})$ and $\mathbf{P}_{\!\omega,u}^{\ast}=(\mathbf{p}_{\omega_3}^{\ast},\mathbf{p}_{\omega_1}^{\ast})$ from~\eqref{eq:14} are
\begin{equation*}
\mathbf{p}_{\sigma_1}^{\ast} = \mathbf{p}_{\sigma_2}^{\ast} = \bigg( u^q\cdot \Big(\frac{1}{2}\Big)^{T_u^{\sigma}(q)}, (1-u)^q\cdot \Big(\frac{1}{2}\Big)^{2T_u^{\sigma}(q)} \bigg)
\end{equation*}
and
\begin{equation*}
\mathbf{p}_{\omega_1}^{\ast} = \mathbf{p}_{\omega_2}^{\ast} = \bigg( u^q\cdot \Big(\frac{1}{2}\Big)^{2T_u^{\omega}(q)}, (1-u)^q\cdot \Big(\frac{1}{2}\Big)^{T_u^{\omega}(q)} \bigg).
\end{equation*}
The main task is to determine when $\mathbf{P}_{\!\sigma,u}^{\ast}\in\mathcal{Q}^{\sigma}$ and $\mathbf{P}_{\!\omega,u}^{\ast}\in\mathcal{Q}^{\omega}$. Since $\mathbf{p}_{\sigma_1}^{\ast} = \mathbf{p}_{\sigma_2}^{\ast}$ and $\mathbf{p}_{\omega_1}^{\ast} = \mathbf{p}_{\omega_2}^{\ast}$, it is enough to consider types of the form $\big((r,1-r);(r,1-r)\big)$. Simple application of~\eqref{eq:15} yields that
\begin{equation*}
C_1^{(2),\sigma}(r) = \Big( \frac{1}{2-r} - \frac{1}{1+r} \Big)\frac{1}{\log 2} \geq 0 \;\Longleftrightarrow\; r\geq \frac{1}{2}
\end{equation*} 
and $C_1^{(2),\omega}(r) = -C_1^{(2),\sigma}(r)\geq 0 \;\Longleftrightarrow\; r\leq 1/2$. Therefore,
\begin{equation*}
\mathbf{P}_{\!\sigma,u}^{\ast}\in\mathcal{Q}^{\sigma} \;\;\Longleftrightarrow\;\; u^q\cdot \Big(\frac{1}{2}\Big)^{T_u^{\sigma}(q)} \geq \frac{1}{2} \;\;\text{ and }\;\; \mathbf{P}_{\!\omega,u}^{\ast}\in\mathcal{Q}^{\omega} \;\;\Longleftrightarrow\;\; (1-u)^q\cdot \Big(\frac{1}{2}\Big)^{T_u^{\omega}(q)} \geq \frac{1}{2}.
\end{equation*}
Using formula~\eqref{eq:44}, we obtain the following equivalences,
\begin{equation*}
\mathbf{P}_{\!\sigma,u}^{\ast}\in\mathcal{Q}^{\sigma} \;\;\Longleftrightarrow\;\; 
\begin{cases*}
q\leq \frac{\log 2}{\log \frac{1-u}{u^2}}, & if  $u\in\big[\frac{1}{2}, \frac{1}{2^s}\big]$\\
q\geq \frac{\log 2}{\log \frac{1-u}{u^2}}, & if  $u\in\big(\frac{1}{2^s},1\big)$
\end{cases*}
\end{equation*}
and
\begin{equation*}
\mathbf{P}_{\!\omega,u}^{\ast}\in\mathcal{Q}^{\omega} \;\;\Longleftrightarrow\;\; q\leq \frac{\log 2}{\log \frac{u}{(1-u)^2}}\; \text{ for every } u\in\Big[ \frac{1}{2},1\Big).
\end{equation*}
We can now determine $T(\nu_{\boldsymbol{\mu}},q)$ for $q\leq 1$. If $q\leq 0$, then $T_u^{\omega}(q) > T_{u}^{\sigma}(q)$ and $\mathbf{P}_{\!\omega,u}^{\ast}\in\mathcal{Q}^{\omega}$, hence, $T(\nu_{\boldsymbol{\mu}},q) = T_u^{\omega}(q)$. If $q\in[0,1]$, then $T_{u}^{\sigma}(q)\geq T_{u}^{\omega}(q)$ and $\mathbf{P}_{\!\sigma,u}^{\ast}\in\mathcal{Q}^{\sigma}$, hence, $T(\nu_{\boldsymbol{\mu}},q) = T_u^{\sigma}(q)$.

If $q>1$, then $\mathbf{P}_{\!\omega,u}^{\ast}\notin\mathcal{Q}^{\omega}$ for all $u\in[1/2,1)$. We abbreviate $\underline{r}* \log \underline{u} = r \log u + (1-r)\log(1-u)$. For fixed $q>1$ and $u\in[1/2,1)$, we need to maximise
\begin{equation*}
t(r) = \big( C_2^{(2),\omega} + C_1^{(2),\omega}\big) \big(-\underline{r}*\log\underline{r}+q\cdot \underline{r}* \log \underline{u}\big) = \frac{-\underline{r}*\log\underline{r}+q\cdot \underline{r}* \log \underline{u}}{(1+r)\log 2}
\end{equation*}
with respect to $r$ directly using types $\big((r,1-r);(r,1-r)\big)$ with $r\leq 1/2$. Elementary calculus shows that $t(r)$ is strictly increasing on $(0,1/2]$, so
\begin{equation*}
\sup _{\mathbf{P}_{\!\omega} \in \mathcal{Q}^{\omega}}\, t(\mathbf{P}_{\omega}) = \max_{r\in(0,1/2]} t(r) = t(1/2) = \frac{2}{3}+\frac{\log( u(1-u))}{3 \log 2}\cdot q.
\end{equation*}
If $1/2\leq u\leq (1/2)^s$ and $q\geq \log 2 / \log \frac{1-u}{u^2}\geq 1$, then $\mathbf{P}_{\!\sigma,u}^{\ast}\notin\mathcal{Q}^{\sigma}$. An analogous calculation shows that in this case as well $\sup _{\mathbf{P}_{\!\sigma} \in \mathcal{Q}^{\sigma}}\, t(\mathbf{P}_{\sigma}) = t(1/2)$. We leave it to the reader to check that
\begin{equation*}
\frac{2}{3}+\frac{\log( u(1-u))}{3 \log 2}\cdot q \leq \min\{ T_u^{\sigma}(q), T_u^{\omega}(q) \}
\end{equation*}
with equality with $T_u^{\sigma}(q)$ if and only if $q= \log 2 / \log \frac{1-u}{u^2}$ and equality with $T_u^{\omega}(q)$ if and only if $q= \log 2 /\log \frac{u}{(1-u)^2}$. The formula for $T(\nu_{\boldsymbol{\mu}},q)$ follows.
\end{proof}

\subsection{A Bara\'nski sponge in three dimensions}\label{sec:3DBaranski}

This example appeared in~\cite[Section 9]{fraserJurga2021AdvMath}. Let $0<1/N<c<b<a<d=1-b<1$ with $a+c<1$ and consider the affine IFS with maps $f_i(x)=A_i x +t_i$, where
\begin{align*}
A_i &= 	\mathrm{diag}(a,b,1/N), \quad t_i=(0,0,(i-1)/N) \;\text{ for } i=1,\ldots,N; \\
A_{N+1} &= 	\mathrm{diag}(c,d,1/N), \quad t_{N+1}=(1-c,b,0).
\end{align*} 
The attractor $F$ is a Bara\'nski sponge, recall Example~\ref{ex:HigherDImCarpets} and projection to the $xy$-plane is a Bara\'nski carpet.

Fraser and Jurga introduce a pressure function $\widehat{P}$ using `modified singular value functions' and show that the unique $s_0$ which satisfies $\widehat{P}(s_0)=1$ is always an upper bound for $\overline{\dim}_B F$ which can be strict for particular choices of parameters in this example. We now show why this happens. Their pressure in this example is
\begin{equation}\label{eq:41}
\widehat{P}(s) = N^{1-s} \cdot \max \big\{ N a^t b^{1-t}+c^t d^{1-t}, Nb+d \big\},
\end{equation}   
where $t$ satisfies $a^t+c^t=1$.

Now let us apply our notation and results. Since we are only interested in the box dimension, we simplify notation in~\eqref{eq:31} to $T_n^{\sigma}= T_n^{\boldsymbol{\mu},\sigma}(0)$ for $n=0,1,2,3$.  First observe that contraction along the $z$-axis is the strongest, hence, the only two orderings in $\mathcal{A}$ are $\sigma=(1,2,3)$ and $\omega=(2,1,3)$. Furthermore, $\mathcal{I}_1^{\sigma}=\mathcal{I}_2^{\sigma} = \{1,N+1\} = \mathcal{I}_1^{\omega}=\mathcal{I}_2^{\omega}$ and $\mathcal{I}_3^{\sigma}=\mathcal{I}=\mathcal{I}_3^{\omega}$. Applying~\eqref{eq:31}, we obtain $T_1^{\sigma}=t=T_2^{\sigma}$ and $T_1^{\omega}=1=T_2^{\omega}$, moreover, $T_3^{\sigma}$ and $T_3^{\omega}$ are the solutions to
\begin{equation}\label{eq:43}
N^{t-T_3^{\sigma}} \big( N a^t +c^t \big) =1 \;\text{ and }\; N^{1-T_3^{\omega}} \big( Nb+d \big) =1,
\end{equation}
respectively. From here, we get the closed forms
\begin{equation*}
T_3^{\sigma} = t+\frac{\log(N a^t+c^t)}{\log N} \;\text{ and }\; T_3^{\omega} = 1+\frac{\log(N b+d)}{\log N}.
\end{equation*}
Corollary~\ref{cor:Lqmain} implies that $\dim_B F\leq \max\{T_3^{\sigma},T_3^{\omega}\}$. Comparing~\eqref{eq:41} with~\eqref{eq:43}, some algebraic manipulations yield that $\max\{T_3^{\sigma},T_3^{\omega}\}\leq s_0$. More precisely, if $T_3^{\sigma}\leq T_3^{\omega}$, then $\max\{T_3^{\sigma},T_3^{\omega}\}= s_0$, however, if $T_3^{\sigma}> T_3^{\omega}$, then $\max\{T_3^{\sigma},T_3^{\omega}\}< s_0$ indicating that $s_0$ is not the correct value. Therefore, the qualitative difference between the two approaches is that distinguishing between the orderings is a necessary and crucial new feature of our method.

\begin{conjecture}
For the Bara\'nski sponge in this section, $\dim_{\mathrm{B}} F = \max\{ T_3^{\sigma},T_3^{\omega}\}$.
\end{conjecture} 

We give a sketch of a possible proof of this conjecture. It is straightforward to determine the dominant types $\mathbf{P}_{\!\sigma}^{\ast}=(\mathbf{p}_{\sigma_3}^{\ast},\mathbf{p}_{\sigma_2}^{\ast},\mathbf{p}_{\sigma_1}^{\ast})$ and $\mathbf{P}_{\!\omega}^{\ast}=(\mathbf{p}_{\omega_3}^{\ast},\mathbf{p}_{\omega_3}^{\ast},\mathbf{p}_{\omega_1}^{\ast})$ from~\eqref{eq:14},
\begin{equation*}
\mathbf{p}_{\sigma_1}^{\ast}=\mathbf{p}_{\sigma_2}^{\ast} = (a^t,c^t), \;\text{ and }\; \mathbf{p}_{\sigma_3}^{\ast} = \Big( \underbrace{\frac{a^t}{N a^t+c^t}\,,\ldots,\,\frac{a^t}{N a^t+c^t}}_{N \text{ times}},\frac{c^t}{N a^t+c^t} \Big),
\end{equation*}
moreover,
\begin{equation*}
\mathbf{p}_{\omega_1}^{\ast}=\mathbf{p}_{\omega_2}^{\ast} = (b,d), \;\text{ and }\; \mathbf{p}_{\omega_3}^{\ast} = \Big( \underbrace{\frac{b}{N b+d}\,,\ldots,\,\frac{b}{N b+d}}_{N \text{ times}},\frac{d}{N b+d} \Big).
\end{equation*}
The main task is to determine for which parameters $(a,b,c,N)$ is $\mathbf{P}_{\!\sigma}^{\ast}\in\mathcal{Q}^{\sigma}$ and $\mathbf{P}_{\!\omega}^{\ast}\in\mathcal{Q}^{\omega}$, i.e. when is  $C_n^{(3),\sigma}(\mathbf{P}_{\!\sigma}^{\ast})\geq 0$ for $n=1,2,3$ and same for $\omega$. This automatically holds for $C_3^{(3),\sigma}(\mathbf{P}_{\!\sigma}^{\ast})$ and also easy for $C_2^{(3),\sigma}(\mathbf{P}_{\!\sigma}^{\ast})$ since $1/N$ is the strongest contraction. It is much more cumbersome to check that $C_1^{(3),\sigma}(\mathbf{P}_{\!\sigma}^{\ast})\geq 0$ and $C_1^{(3),\omega}(\mathbf{P}_{\!\omega}^{\ast})\geq 0$. Both $C_1^{(3),\sigma}(\mathbf{P}_{\!\sigma}^{\ast})$ and $C_1^{(3),\omega}(\mathbf{P}_{\!\omega}^{\ast})$ are functions of $(a,b,c,N)$. To prove the conjecture, it is enough to verify the following three things:
\begin{enumerate}
\item $\{(a,b,c,N):\, \mathbf{P}_{\!\sigma}^{\ast}\notin\mathcal{Q}^{\sigma} \text{ and } \mathbf{P}_{\!\omega}^{\ast}\notin\mathcal{Q}^{\omega}\}=\emptyset$ (otherwise $\dim_B F < \max\{ T_3^{\sigma},T_3^{\omega}\}$);
\item if $(a,b,c,N)$ is such that $\mathbf{P}_{\!\sigma}^{\ast}\notin\mathcal{Q}^{\sigma}$, then $T_3^{\omega}\geq T_3^{\sigma}$;
\item if $(a,b,c,N)$ is such that $\mathbf{P}_{\!\omega}^{\ast}\notin\mathcal{Q}^{\omega}$, then $T_3^{\sigma}\geq T_3^{\omega}$;
\end{enumerate}
Verifying these seems possible but certainly tedious. Instead, we conducted an exhaustive search on the parameter space to see whether we can find a counterexample. Using \emph{Mathematica 13.1}, we chose $N=100$ up to $1000$ with increments of $50$, furthermore, $0.02\leq c\leq 0.49$, $c+0.01\leq b\leq 0.5$ and $b+0.01\leq a\leq 1-c-0.01$ all with increments of $0.01$. For all instances we found that all three conditions are true, supporting the conjecture. We note that $\{(a,b,c,N):\, \mathbf{P}_{\!\sigma}^{\ast}\notin\mathcal{Q}^{\sigma}\}\neq \emptyset$ and also $\{(a,b,c,N):\, \mathbf{P}_{\!\omega}^{\ast}\notin\mathcal{Q}^{\omega}\}\neq \emptyset$, so (2) and (3) are not empty statements.

\section{Proof of Proposition~\ref{prop:main} and~\ref{prop:maxminexponent}}\label{sec:03}

The proof of these two propositions follow a very similar argument, therefore, we present them side-by-side. Recall notation from Section~\ref{sec:01} and~\ref{sec:02}. In particular, 
\begin{equation}\label{eq:36}
t(\mathbf{P}_{\!\sigma}) = t(\mathbf{p}_{\sigma_d};\ldots;\mathbf{p}_{\sigma_1}) = \sum_{n=1}^{d} 
C_{n}^{(d), \sigma}(\mathbf{P}_{\!\sigma}) \cdot\bigg(H(\mathbf{p}_{\sigma_{n}}) + \int\! \varphi_n^{\sigma}\,\mathrm{d}\mathbf{p}_{\sigma_n}\bigg),
\end{equation}
and
\begin{equation}\label{eq:360}
S(\mathbf{P}_{\!\sigma}) = S(\mathbf{p}_{\sigma_d};\ldots;\mathbf{p}_{\sigma_1}) = -\sum_{n=1}^{d} 
C_{n}^{(d), \sigma}(\mathbf{P}_{\!\sigma}) \cdot \int\! \log \boldsymbol{\mu}_n^{\sigma}\,\mathrm{d}\mathbf{p}_{\sigma_n}.
\end{equation}
For $\sigma\in\mathcal{A}$, $1\leq n\leq d$ and $\mathbf{p}_{\sigma_{n}}\in\mathcal{P}_n^{\sigma}$ let
\begin{equation*}
f_{n}^{\sigma}(\mathbf{p}_{\sigma_{n}})\coloneqq \frac{1}{\chi_{n}^{\sigma}(\mathbf{p}_{\sigma_{n}})} \bigg(H(\mathbf{p}_{\sigma_{n}}) + \int\! \varphi_n^{\sigma}\,\mathrm{d}\mathbf{p}_{\sigma_n} - \sum_{k=1}^{n-1}\chi_k^{\sigma}(\mathbf{p}_{\sigma_n})\big( T_k^{\sigma}-T_{k-1}^{\sigma} \big)\bigg),
\end{equation*}
and
\begin{align*}
g_{n}^{\sigma}(\mathbf{p}_{\sigma_{n}}) &\coloneqq \frac{-1}{\chi_n^{\sigma}(\mathbf{p}_{\sigma_{n}})} \bigg( \int\! \log \boldsymbol{\mu}_n^{\sigma}\,\mathrm{d}\mathbf{p}_{\sigma_n} + \sum_{k=1}^{n-1}\chi_k^{\sigma}(\mathbf{p}_{\sigma_n})\big( \overline{S}_{k}^{\boldsymbol{\mu},\sigma}-\overline{S}_{k-1}^{\boldsymbol{\mu},\sigma} \big)\bigg), \\
h_{n}^{\sigma}(\mathbf{p}_{\sigma_{n}}) &\coloneqq \frac{-1}{\chi_n^{\sigma}(\mathbf{p}_{\sigma_{n}})} \bigg( \int\! \log \boldsymbol{\mu}_n^{\sigma}\,\mathrm{d}\mathbf{p}_{\sigma_n} + \sum_{k=1}^{n-1}\chi_k^{\sigma}(\mathbf{p}_{\sigma_n})\big( \underline{S}_{k}^{\boldsymbol{\mu},\sigma}-\underline{S}_{k-1}^{\boldsymbol{\mu},\sigma} \big)\bigg).
\end{align*}
For $n=1$ the empty sum is taken to equal 0. 

\begin{lemma}\label{lem:fnsigma}
For every $\sigma\in\mathcal{A}$ and $1\leq n\leq d$,
\begin{equation*}
\sup_{\mathbf{p} \in \mathcal{P}_n^{\sigma}} f_{n}^{\sigma}(\mathbf{p}) = f_{n}^{\sigma}(\mathbf{p}_{\sigma_{n}}^{\ast}) = T_n^{\sigma}-T_{n-1}^{\sigma},
\end{equation*}
moreover,
\begin{equation*}
\sup_{\mathbf{p} \in \mathcal{P}_n^{\sigma}} g_{n}^{\sigma}(\mathbf{p}) = g_{n}^{\sigma}(\overline{\mathbf{k}}_{\sigma_{n}}) = \overline{S}_{n}^{\boldsymbol{\mu},\sigma}-\overline{S}_{n-1}^{\boldsymbol{\mu},\sigma} \;\;\text{ and } \inf_{\mathbf{p} \in \mathcal{P}_n^{\sigma}} h_{n}^{\sigma}(\mathbf{p}) = h_{n}^{\sigma}(\underline{\mathbf{k}}_{\sigma_{n}}) = \underline{S}_{n}^{\boldsymbol{\mu},\sigma}-\underline{S}_{n-1}^{\boldsymbol{\mu},\sigma}.
\end{equation*}
\end{lemma}
\begin{proof}
Let $\mathbf{p}$ and $\mathbf{q}$ be two probability vectors of the same length with strictly positive entries. The \emph{Kullback--Leibler divergence} (or \emph{relative entropy}) of $\mathbf{p}$ with respect to $\mathbf{q}$ is
\begin{equation*}
	H(\mathbf{p}\| \mathbf{q}) \coloneqq \sum_{i} p_{i} \log (p_{i} / q_{i}).
\end{equation*}
It is asymmetric and $H(\mathbf{p}\| \mathbf{q})\geq 0$ with equality if and only if $\mathbf{p}=\mathbf{q}$.
	
Let $\mathbf{p} \in \mathcal{P}_n^{\sigma}$. Then using~\eqref{eq:14},
\begin{align*}
	f_{n}^{\sigma}(\mathbf{p})
	&= \frac{1}{\chi_{n}^{\sigma}(\mathbf{p})} \bigg( -\sum_{i\in\mathcal{I}_{n}^{\sigma}}p(i)\log \Big( p_{\sigma_{n}}^{\ast}(i) \frac{p(i)}{p_{\sigma_{n}}^{\ast}(i)} \Big) + \int\! \varphi_n^{\sigma}\,\mathrm{d}\mathbf{p} - \sum_{k=1}^{n-1}\chi_k^{\sigma}(\mathbf{p})\big( T_k^{\sigma}-T_{k-1}^{\sigma} \big)\bigg) \\
	&= T_n^{\sigma}-T_{n-1}^{\sigma} - \frac{H(\mathbf{p}\| \mathbf{p}_{\sigma_{n}}^{\ast})}{\chi_{n}^{\sigma}(\mathbf{p})} = T_n^{\sigma}-T_{n-1}^{\sigma} \;\Longleftrightarrow\; \mathbf{p} = \mathbf{p}_{\sigma_{n}}^{\ast},
\end{align*}
otherwise $f_{n}^{\sigma}(\mathbf{p})<f_{n}^{\sigma}(\mathbf{p}_{\sigma_{n}}^{\ast})=T_n^{\sigma}-T_{n-1}^{\sigma}$. Since $\mathbf{p}_{\sigma_{n}}^{\ast}$ is uniformly bounded away from the boundary of $\mathcal{P}_n^{\sigma}$ and $f_{n}^{\sigma}(\mathbf{p})$ is continuous in $\mathbf{p}$, these imply that $\sup_{\mathbf{p} \in \mathcal{P}_n^{\sigma}} f_{n}^{\sigma}(\mathbf{p}) = f_{n}^{\sigma}(\mathbf{p}_{\sigma_{n}}^{\ast})$.

The extreme value for $g_n^{\sigma}$ and $h_n^{\sigma}$ is even simpler. For every $i\in\mathcal{I}_n^{\sigma}$, the numerator and the denominator in both $g_n^{\sigma}$ and $h_n^{\sigma}$ are linear in $p_{\sigma_n}(i)$. Therefore, $g_n^{\sigma}$ and $h_n^{\sigma}$ considered as one variable functions of $p_{\sigma_n}(i)$ take their extreme values when $p_{\sigma_n}(i)$ is equal to $0$ or $1$. So it is enough to consider the degenerate probability vectors putting all mass on one of the coordinates of $\mathcal{I}_n^{\sigma}$. Out of these vectors, by definition, $\overline{\mathbf{k}}_{\sigma_{n}}$ maximises $g_n^{\sigma}$ while $\underline{\mathbf{k}}_{\sigma_{n}}$ minimises $h_n^{\sigma}$. To conclude, observe from~\eqref{eq:39} and~\eqref{eq:390} that $g_{n}^{\sigma}(\overline{\mathbf{k}}_{\sigma_{n}}) = \overline{S}_{n}^{\boldsymbol{\mu},\sigma}-\overline{S}_{n-1}^{\boldsymbol{\mu},\sigma}$ and $h_{n}^{\sigma}(\underline{\mathbf{k}}_{\sigma_{n}}) = \underline{S}_{n}^{\boldsymbol{\mu},\sigma}-\underline{S}_{n-1}^{\boldsymbol{\mu},\sigma}$.
\end{proof}

\begin{lemma}\label{lem:tpsigma}
For every $1\leq n\leq d-1$, $t(\mathbf{p}_{\sigma_d};\ldots;\mathbf{p}_{\sigma_{n+1}};\mathbf{p}_{\sigma_n}^{\ast};\ldots;\mathbf{p}_{\sigma_1}^{\ast})$ is equal to
\begin{equation}\label{eq:37}
	T_n^{\sigma} +\sum_{k=n+1}^{d} C_{k}^{(d), \sigma}(\mathbf{P}_{\!\sigma}) \cdot\bigg(H(\mathbf{p}_{\sigma_{k}}) + \int\! \varphi_k^{\sigma}\,\mathrm{d}\mathbf{p}_{\sigma_k} - \sum_{\ell=1}^{n} \chi_{\ell}^{\sigma}(\mathbf{p}_{\sigma_k})\big( T_{\ell}^{\sigma}-T_{\ell-1}^{\sigma} \big) \bigg),
\end{equation}
moreover, $t(\mathbf{p}_{\sigma_d}^{\ast};\ldots;\mathbf{p}_{\sigma_1}^{\ast})=T_d^{\sigma}$.
\end{lemma}

\begin{proof}
The proof goes by induction. First for $n=1$, using the definition of $C_{1}^{(d), \sigma}(\mathbf{P}_{\!\sigma})$ from~\eqref{eq:15}, $t(\mathbf{p}_{\sigma_d};\ldots;\mathbf{p}_{\sigma_1})$ is equal to
\begin{equation*}
	\underbrace{ \sum_{m=2}^{d} 
		C_{m}^{(d), \sigma}(\mathbf{P}_{\!\sigma}) \cdot\bigg(\! H(\mathbf{p}_{\sigma_{m}}) + \int\! \varphi_m^{\sigma}\,\mathrm{d}\mathbf{p}_{\sigma_m}\!\bigg)}_{\text{independent of }\mathbf{p}_{\sigma_1}} + 
	\underbrace{ \bigg(\! 1-\!\sum_{m=2}^{d} C_m^{(d),\sigma}(\mathbf{P}_{\!\sigma})\cdot \chi_1^{\sigma}(\mathbf{p}_{\sigma_m})\! \bigg) }_{\text{independent of }\mathbf{p}_{\sigma_1}} 
	\underbrace{ \frac{H(\mathbf{p}_{\sigma_{1}}) + \int\! \varphi_1^{\sigma}\,\mathrm{d}\mathbf{p}_{\sigma_1}}{\chi_1^{\sigma}(\mathbf{p}_{\sigma_1})} }_{=f_1^{\sigma}(\mathbf{p}_{\sigma_1})}.
\end{equation*}
From Lemma~\ref{lem:fnsigma} we know that $f_1^{\sigma}(\mathbf{p}_{\sigma_1})\leq f_{1}^{\sigma}(\mathbf{p}_{\sigma_{1}}^{\ast}) = T_1^{\sigma}$ and so
\begin{equation*}
	t(\mathbf{p}_{\sigma_d};\ldots;\mathbf{p}_{\sigma_1})\leq t(\mathbf{p}_{\sigma_d};\ldots;\mathbf{p}_{\sigma_2};\mathbf{p}_{\sigma_1}^{\ast}) =
	T_1^{\sigma} +\sum_{k=2}^{d} C_{k}^{(d), \sigma}(\mathbf{P}_{\!\sigma}) \cdot\bigg(\!H(\mathbf{p}_{\sigma_{k}}) + \int\! \varphi_k^{\sigma}\,\mathrm{d}\mathbf{p}_{\sigma_k} - \chi_{1}^{\sigma}(\mathbf{p}_{\sigma_k})T_1^{\sigma} \!\bigg),
\end{equation*}
proving the assertion for $n=1$ (recalling that $T_0^{\sigma}=0$).
	
Assume that~\eqref{eq:37} holds for $n-1$. Then using the definition of $C_{n}^{(d), \sigma}(\mathbf{P}_{\!\sigma})$ from~\eqref{eq:15} and the induction hypothesis, $t(\mathbf{p}_{\sigma_d};\ldots;\mathbf{p}_{\sigma_{n}};\mathbf{p}_{\sigma_{n-1}}^{\ast};\ldots;\mathbf{p}_{\sigma_1}^{\ast})$ is equal to 
\begin{multline*}
	\underbrace{ T_{n-1}^{\sigma} +\sum_{k=n+1}^{d} C_{k}^{(d), \sigma}(\mathbf{P}_{\!\sigma}) \cdot\bigg(H(\mathbf{p}_{\sigma_{k}}) + \int\! \varphi_k^{\sigma}\,\mathrm{d}\mathbf{p}_{\sigma_k} - \sum_{\ell=1}^{n-1} \chi_{\ell}^{\sigma}(\mathbf{p}_{\sigma_k})\big( T_{\ell}^{\sigma}-T_{\ell-1}^{\sigma} \big) \bigg) }_{\text{independent of }\mathbf{p}_{\sigma_n}} \\
	+
	\underbrace{ \bigg(\! 1-\!\sum_{m=n+1}^{d} C_m^{(d),\sigma}(\mathbf{P}_{\!\sigma})\cdot \chi_n^{\sigma}(\mathbf{p}_{\sigma_m})\! \bigg) }_{\text{independent of }\mathbf{p}_{\sigma_n}}
	\underbrace{ \frac{ H(\mathbf{p}_{\sigma_{n}}) + \int\! \varphi_n^{\sigma}\,\mathrm{d}\mathbf{p}_{\sigma_n} - \sum_{\ell=1}^{n-1}\chi_{\ell}^{\sigma}(\mathbf{p}_{\sigma_n})\big( T_{\ell}^{\sigma}-T_{\ell-1}^{\sigma} \big) }{\chi_n^{\sigma}(\mathbf{p}_{\sigma_n})} }_{=f_{n}^{\sigma}(\mathbf{p}_{\sigma_{n}})\leq f_{n}^{\sigma}(\mathbf{p}_{\sigma_{n}}^{\ast}) = T_n^{\sigma}-T_{n-1}^{\sigma} \;\text{ by Lemma~\ref{lem:fnsigma}}}.
\end{multline*}
Hence,
\begin{multline*}
	t(\mathbf{p}_{\sigma_d};\ldots;\mathbf{p}_{\sigma_{n}};\mathbf{p}_{\sigma_{n-1}}^{\ast};\ldots;\mathbf{p}_{\sigma_1}^{\ast})
	\leq
	t(\mathbf{p}_{\sigma_d};\ldots;\mathbf{p}_{\sigma_{n+1}};\mathbf{p}_{\sigma_{n}}^{\ast};\ldots;\mathbf{p}_{\sigma_1}^{\ast})
	= \\
	T_{n-1}^{\sigma} +T_n^{\sigma}-T_{n-1}^{\sigma} +\sum_{k=n+1}^{d} C_{k}^{(d), \sigma}(\mathbf{P}_{\!\sigma}) \cdot\bigg(H(\mathbf{p}_{\sigma_{k}}) + \int\! \varphi_k^{\sigma}\,\mathrm{d}\mathbf{p}_{\sigma_k} - \sum_{\ell=1}^{n} \chi_{\ell}^{\sigma}(\mathbf{p}_{\sigma_k})\big( T_{\ell}^{\sigma}-T_{\ell-1}^{\sigma} \big) \bigg),
\end{multline*} 
proving the assertion for $n\leq d-1$.
	
Finally, for $n=d$, we use that $C_{d}^{(d), \sigma}(\mathbf{P}_{\!\sigma}) = 1/\chi_d^{\sigma}(\mathbf{p}_{\sigma_d})$ to obtain
\begin{equation*}
	t(\mathbf{p}_{\sigma_d};\mathbf{p}_{\sigma_{d-1}}^{\ast};\ldots;\mathbf{p}_{\sigma_1}^{\ast}) = T_{d-1}^{\sigma} +
	\underbrace{ \frac{ H(\mathbf{p}_{\sigma_{d}}) + \int\! \varphi_d^{\sigma}\,\mathrm{d}\mathbf{p}_{\sigma_d} - \sum_{\ell=1}^{d-1}\chi_{\ell}^{\sigma}(\mathbf{p}_{\sigma_d})\big( T_{\ell}^{\sigma}-T_{\ell-1}^{\sigma} \big) }{\chi_d^{\sigma}(\mathbf{p}_{\sigma_d})} }_{=f_{d}^{\sigma}(\mathbf{p}_{\sigma_{d}})\leq f_{d}^{\sigma}(\mathbf{p}_{\sigma_{d}}^{\ast}) = T_d^{\sigma}-T_{d-1}^{\sigma} \;\text{ by Lemma~\ref{lem:fnsigma}}}.
\end{equation*}
To conclude, $t(\mathbf{p}_{\sigma_d};\mathbf{p}_{\sigma_{d-1}}^{\ast};\ldots;\mathbf{p}_{\sigma_1}^{\ast})\leq t(\mathbf{P}_{\!\sigma}^{\ast})=T_d^{\sigma}$.
\end{proof}

\begin{proof}[Proof of Proposition~\ref{prop:main}]
In the process of proving Lemma~\ref{lem:tpsigma}, we actually showed that for any $\sigma\in\mathcal{A}$ and $\mathbf{P}_{\!\sigma} \in \mathcal{P}^{\sigma}$,
\begin{equation*}
	t(\mathbf{P}_{\!\sigma}) \leq t(\mathbf{p}_{\sigma_d};\ldots;\mathbf{p}_{\sigma_2};\mathbf{p}_{\sigma_1}^{\ast}) \leq \ldots  \leq t(\mathbf{p}_{\sigma_d};\mathbf{p}_{\sigma_{d-1}}^{\ast};\ldots;\mathbf{p}_{\sigma_1}^{\ast}) \leq t(\mathbf{P}_{\!\sigma}^{\ast})=T_d^{\sigma},
\end{equation*}
with equality throughout if and only if $\mathbf{P}_{\!\sigma}=\mathbf{P}_{\!\sigma}^{\ast}$. Since $\mathbf{p}_{\sigma_{n}}^{\ast}$ is uniformly bounded away from the boundary of $\mathcal{P}_n^{\sigma}$ for every $1\leq n\leq d$ and $t(\mathbf{P}_{\!\sigma})$ is continuous in $\mathbf{P}_{\!\sigma}$, these imply that $\sup _{\mathbf{P}_{\!\sigma} \in \mathcal{P}^{\sigma}}\, t(\mathbf{P}_{\!\sigma}) = t(\mathbf{P}_{\!\sigma}^{\ast})=T_d^{\sigma}$.
\end{proof}

\begin{lemma}\label{lem:Spsigma}
For every $1\leq n\leq d-1$, $S(\mathbf{p}_{\sigma_d};\ldots;\mathbf{p}_{\sigma_{n+1}};\overline{\mathbf{k}}_{\sigma_n};\ldots;\overline{\mathbf{k}}_{\sigma_1})$ is equal to
\begin{equation}\label{eq:42}
	\overline{S}_n^{\boldsymbol{\mu},\sigma} -\sum_{k=n+1}^{d} C_{k}^{(d), \sigma}(\mathbf{P}_{\!\sigma}) \cdot\bigg( \int\! \log \boldsymbol{\mu}_k^{\sigma}\,\mathrm{d}\mathbf{p}_{\sigma_k} + \sum_{\ell=1}^{n} \chi_{\ell}^{\sigma}(\mathbf{p}_{\sigma_k})\big( \overline{S}_{\ell}^{\boldsymbol{\mu},\sigma}-\overline{S}_{\ell-1}^{\boldsymbol{\mu},\sigma} \big) \bigg),
\end{equation}
moreover, $S(\overline{\mathbf{k}}_{\sigma_d};\ldots;\overline{\mathbf{k}}_{\sigma_1})=\overline{S}_d^{\boldsymbol{\mu},\sigma}$. Similarly, $S(\mathbf{p}_{\sigma_d};\ldots;\mathbf{p}_{\sigma_{n+1}};\underline{\mathbf{k}}_{\sigma_n};\ldots;\underline{\mathbf{k}}_{\sigma_1})$ is equal to
\begin{equation*}
	\underline{S}_n^{\boldsymbol{\mu},\sigma} -\sum_{k=n+1}^{d} C_{k}^{(d), \sigma}(\mathbf{P}_{\!\sigma}) \cdot\bigg( \int\! \log \boldsymbol{\mu}_k^{\sigma}\,\mathrm{d}\mathbf{p}_{\sigma_k} + \sum_{\ell=1}^{n} \chi_{\ell}^{\sigma}(\mathbf{p}_{\sigma_k})\big( \underline{S}_{\ell}^{\boldsymbol{\mu},\sigma}-\underline{S}_{\ell-1}^{\boldsymbol{\mu},\sigma} \big) \bigg),
\end{equation*}
moreover, $S(\underline{\mathbf{k}}_{\sigma_d};\ldots;\underline{\mathbf{k}}_{\sigma_1})=\underline{S}_d^{\boldsymbol{\mu},\sigma}$.
\end{lemma}
\begin{proof}
We just sketch the proof since it is a very similar induction argument to the one in the proof of Lemma~\ref{lem:tpsigma}. First, $S(\mathbf{P}_{\!\sigma})$ is equal to
\begin{equation*}
\underbrace{ -\sum_{m=2}^{d} C_{m}^{(d), \sigma}(\mathbf{P}_{\!\sigma}) \cdot \int\! \log \boldsymbol{\mu}_m^{\sigma}\,\mathrm{d}\mathbf{p}_{\sigma_m}}_{\text{independent of }\mathbf{p}_{\sigma_1}} + 
	\underbrace{ \bigg(\! 1-\!\sum_{m=2}^{d} C_m^{(d),\sigma}(\mathbf{P}_{\!\sigma})\cdot \chi_1^{\sigma}(\mathbf{p}_{\sigma_m})\! \bigg) }_{\text{independent of }\mathbf{p}_{\sigma_1}} 
	\underbrace{ \frac{- \int\! \log  \boldsymbol{\mu}_1^{\sigma}\,\mathrm{d}\mathbf{p}_{\sigma_1}}{\chi_1^{\sigma}(\mathbf{p}_{\sigma_1})} }_{(\ast)},
\end{equation*}
where $(\ast)=g_1^{\sigma}(\mathbf{p}_{\sigma_1})\leq \overline{S}_1^{\boldsymbol{\mu},\sigma}$ by Lemma~\ref{lem:fnsigma}. After rearranging, we obtain~\eqref{eq:42} for $n=1$.

Now assume that~\eqref{eq:42} holds for $n-1$. Then $S(\mathbf{p}_{\sigma_d};\ldots;\mathbf{p}_{\sigma_{n}};\overline{\mathbf{k}}_{\sigma_{n-1}};\ldots;\overline{\mathbf{k}}_{\sigma_1})$ equals
\begin{multline*}
\underbrace{ \overline{S}_{n-1}^{\boldsymbol{\mu},\sigma} -\sum_{k=n+1}^{d} C_{k}^{(d), \sigma}(\mathbf{P}_{\!\sigma}) \cdot\bigg( \int\! \log \boldsymbol{\mu}_k^{\sigma}\,\mathrm{d}\mathbf{p}_{\sigma_k} + \sum_{\ell=1}^{n-1} \chi_{\ell}^{\sigma}(\mathbf{p}_{\sigma_k})\big( \overline{S}_{\ell}^{\boldsymbol{\mu},\sigma}-\overline{S}_{\ell-1}^{\boldsymbol{\mu},\sigma} \big) \bigg) }_{\text{independent of }\mathbf{p}_{\sigma_n}} \\
	+
\underbrace{ \bigg(\! 1-\!\sum_{m=n+1}^{d} C_m^{(d),\sigma}(\mathbf{P}_{\!\sigma})\cdot \chi_n^{\sigma}(\mathbf{p}_{\sigma_m})\! \bigg) }_{\text{independent of }\mathbf{p}_{\sigma_n}}
\underbrace{ \frac{ - \int\! \log \boldsymbol{\mu}_n^{\sigma}\,\mathrm{d}\mathbf{p}_{\sigma_n} - \sum_{\ell=1}^{n-1}\chi_{\ell}^{\sigma}(\mathbf{p}_{\sigma_n})\big(\overline{S}_{\ell}^{\boldsymbol{\mu},\sigma}-\overline{S}_{\ell-1}^{\boldsymbol{\mu},\sigma} \big) }{\chi_n^{\sigma}(\mathbf{p}_{\sigma_n})} }_{=g_{n}^{\sigma}(\mathbf{p}_{\sigma_{n}})\leq g_{n}^{\sigma}(\overline{\mathbf{k}}_{\sigma_{n}}) = \overline{S}_{n}^{\boldsymbol{\mu},\sigma}-\overline{S}_{n-1}^{\boldsymbol{\mu},\sigma} \;\text{ by Lemma~\ref{lem:fnsigma}}}.
\end{multline*}
After rearranging, we again see that~\eqref{eq:42} holds for $n\leq d-1$.

Finally, for $n=d$, we use that $C_{d}^{(d), \sigma}(\mathbf{P}_{\!\sigma}) = 1/\chi_d^{\sigma}(\mathbf{p}_{\sigma_d})$ to obtain
\begin{equation*}
S(\mathbf{p}_{\sigma_d};\overline{\mathbf{k}}_{\sigma_{d-1}};\ldots;\overline{\mathbf{k}}_{\sigma_1}) = \overline{S}_{d-1}^{\boldsymbol{\mu},\sigma} +
	\underbrace{ \frac{ - \int\! \log \boldsymbol{\mu}_d^{\sigma}\,\mathrm{d}\mathbf{p}_{\sigma_d} - \sum_{\ell=1}^{d-1}\chi_{\ell}^{\sigma}(\mathbf{p}_{\sigma_d})\big( \overline{S}_{\ell}^{\boldsymbol{\mu},\sigma}-\overline{S}_{\ell-1}^{\boldsymbol{\mu},\sigma} \big) }{\chi_d^{\sigma}(\mathbf{p}_{\sigma_d})} }_{=g_{d}^{\sigma}(\mathbf{p}_{\sigma_{d}})\leq g_{d}^{\sigma}(\overline{\mathbf{k}}_{\sigma_{d}}) = \overline{S}_{d}^{\boldsymbol{\mu},\sigma}-\overline{S}_{d-1}^{\boldsymbol{\mu},\sigma} \;\text{ by Lemma~\ref{lem:fnsigma}}}.
\end{equation*}
To conclude, $S(\mathbf{p}_{\sigma_d};\overline{\mathbf{k}}_{\sigma_{d-1}};\ldots;\overline{\mathbf{k}}_{\sigma_1}) \leq S(\overline{\mathbf{K}}_{\sigma})=\overline{S}_d^{\boldsymbol{\mu},\sigma}$.

The proof for $S(\mathbf{p}_{\sigma_d};\ldots;\mathbf{p}_{\sigma_{n+1}};\underline{\mathbf{k}}_{\sigma_n};\ldots;\underline{\mathbf{k}}_{\sigma_1})$ is exactly the same except that $h_{n}^{\sigma}(\mathbf{p}_{\sigma_{n}})\geq h_{n}^{\sigma}(\underline{\mathbf{k}}_{\sigma_{n}}) = \underline{S}_{n}^{\boldsymbol{\mu},\sigma}-\underline{S}_{n-1}^{\boldsymbol{\mu},\sigma}$ is used instead of $g_n^{\sigma}$.
\end{proof}

\begin{proof}[Proof of Proposition~\ref{prop:maxminexponent}]
In the proof of Lemma~\ref{lem:Spsigma} we actually showed that 
\begin{equation*}
S(\mathbf{P}_{\!\sigma}) \leq S(\mathbf{p}_{\sigma_d};\ldots;\mathbf{p}_{\sigma_2};\overline{\mathbf{k}}_{\sigma_1}) \leq \ldots  \leq S(\mathbf{p}_{\sigma_d};\overline{\mathbf{k}}_{\sigma_{d-1}};\ldots;\overline{\mathbf{k}}_{\sigma_1}) \leq S(\overline{\mathbf{K}}_{\sigma})=\overline{S}_d^{\boldsymbol{\mu},\sigma}
\end{equation*}
and
\begin{equation*}
	S(\mathbf{P}_{\!\sigma}) \geq S(\mathbf{p}_{\sigma_d};\ldots;\mathbf{p}_{\sigma_2};\underline{\mathbf{k}}_{\sigma_1}) \geq \ldots  \geq S(\mathbf{p}_{\sigma_d};\underline{\mathbf{k}}_{\sigma_{d-1}};\ldots;\underline{\mathbf{k}}_{\sigma_1}) \geq S(\underline{\mathbf{K}}_{\sigma})=\underline{S}_d^{\boldsymbol{\mu},\sigma}.
\end{equation*}
\end{proof}

\section{Proof of Theorem~\ref{thm:main1} and Corollary~\ref{cor:main}}\label{sec:04}

\subsection{Preliminaries}\label{subsec:40}

Fix $\delta>0$ and consider any $\sigma\in\mathcal{A}_{\delta}$. Recall the symbolic representation~\eqref{eq:11} of a $\delta$-approximate cube $B_{\delta}(\ii)\in\mathcal{B}_{\delta}^{\sigma}$ is determined by the first $L_{\delta}(\ii,\sigma_1)$ symbols of $\ii$. We introduce the \emph{type of $\ii\in\Sigma_{\delta}^{\sigma}$ at scale $\delta$} (and also of $B_{\delta}(\ii)$) to be the $\#\mathcal{I}_d^{\sigma}+\#\mathcal{I}_{d-1}^{\sigma}+\ldots+\#\mathcal{I}_1^{\sigma}$ dimensional empirical vector
\begin{equation*}
\tau_{\delta}^{\sigma}(\ii) \coloneqq \big( \tau_{\delta}(\ii,\sigma_d)\,;\, \tau_{\delta}(\ii,\sigma_{d-1})\,;\,\ldots \,;\, \tau_{\delta}(\ii,\sigma_1) \big),
\end{equation*}
where for $1\leq n\leq d$ using the abbreviation $|\ii(\delta,n)|\coloneqq L_{\delta}(\ii,\sigma_n)-L_{\delta}(\ii,\sigma_{n+1})$,
\begin{equation*}
\tau_{\delta}(\ii,\sigma_n) \coloneqq \frac{1}{|\ii(\delta,n)|} \Big( \#\big\{L_{\delta}(\ii,\sigma_{n+1})+1\leq\ell\leq L_{\delta}(\ii,\sigma_n):\, \Pi_{n}^{\sigma}i_{\ell}=j\big\}  \Big)_{j\in\mathcal{I}_n^{\sigma}} .
\end{equation*}
Note that $\tau_{\delta}(\ii,\sigma_n)$ is an $\#\mathcal{I}_n^{\sigma}$ dimensional probability vector except when $L_{\delta}(\ii,\sigma_n)=L_{\delta}(\ii,\sigma_{n+1})$, then we set $\tau_{\delta}(\ii,\sigma_n)=(0,\ldots,0)$. The set of all possible $\sigma$-ordered types at scale $\delta$ is
\begin{equation*}
\mathcal{T}_\delta^{\sigma}\coloneqq\big\{\mathbf{P}=(\mathbf{p}_{\sigma_d};\mathbf{p}_{\sigma_{d-1}};\ldots; \mathbf{p}_{\sigma_1}): \text{ there exists } B_\delta(\ii)\in\mathcal{B}_\delta^{\sigma} \text{ such that } \mathbf{P}=\tau_\delta^{\sigma}(\ii)\big\}\subset \mathcal{P}^{\sigma},
\end{equation*}
and the \emph{type class of} $\mathbf{P}\in \mathcal{T}_\delta^{\sigma}$ is the set
\begin{equation*}
T_\delta^{\sigma}(\mathbf{P}) \coloneqq \big\{B_\delta(\ii)\in\mathcal{B}_\delta^{\sigma}:\, \tau_\delta^{\sigma}(\ii)=\mathbf{P}\big\}.
\end{equation*}
\begin{lemma}\label{lem:1}
Fix $\delta>0$ and $\sigma\in\mathcal{A}_{\delta}$. Then
\begin{equation}\label{eq:32}
\# \mathcal{T}_\delta^{\sigma} \leq \prod_{n=1}^d \Big( \max_{B_\delta(\ii)\in\mathcal{B}_\delta^{\sigma}} |\ii(\delta,n)|+1 \Big)^{\#\mathcal{I}_n^{\sigma}+1}.
\end{equation}
Moreover, for every $\mathbf{P}\in\mathcal{T}_{\delta}^{\sigma}$ and $\ii\in\Sigma_{\delta}^{\sigma}$ such that $\tau_\delta^{\sigma}(\ii)=\mathbf{P}$,
\begin{equation}\label{eq:33}
\exp\!\left[ \sum_{n=1}^d |\ii(\delta,n)| H(\mathbf{p}_{\sigma_n})\right]  \prod_{n=1}^d \big(|\ii(\delta,n)|+1\big)^{-\#\mathcal{I}_n^{\sigma}}  
\leq \#T_\delta^{\sigma}(\mathbf{P})  \leq \exp\!\left[ \sum_{n=1}^d |\ii(\delta,n)| H(\mathbf{p}_{\sigma_n})\right].
\end{equation}
\end{lemma}
\begin{proof}
For each $\mathbf{P}\in\mathcal{T}_{\delta}^{\sigma}$, $\mathbf{p}_{\sigma_n}$ is an $\#\mathcal{I}_n^{\sigma}$ dimensional vector with components belonging to the set $\{k/|\ii(\delta,n)|:\, 0\leq k\leq |\ii(\delta,n)|\}$. Moreover, $0\leq |\ii(\delta,n)|\leq \max_{B_\delta(\ii)\in\mathcal{B}_\delta^{\sigma}} |\ii(\delta,n)|$. Hence, a crude upper bound for the number of different $\mathbf{p}_{\sigma_n}$ is $( \max_{B_\delta(\ii)\in\mathcal{B}_\delta^{\sigma}} |\ii(\delta,n)|+1 )^{\#\mathcal{I}_n^{\sigma}+1}$. Multiplying for each coordinate $1\leq n\leq d$ gives the claim for $\# \mathcal{T}_\delta^{\sigma}$.
 
Let $\mathcal{I}$ be an arbitrary finite index set. It is well known from the method of types, see~\cite[Lemma 2.1.8]{DemboZeitouniLDP}, that
\begin{equation}\label{eq:02}
(n+1)^{-\#\mathcal{I}} e^{n H(\mathbf{p})} \leq \#\{(i_1,\ldots,i_n)\in\mathcal{I}^n: \text{ the type } \tau(i_1,\ldots,i_n)=\mathbf{p} \} \leq e^{n H(\mathbf{p})}.
\end{equation}
The claim now follows by applying~\eqref{eq:02} to each block $\left(\Pi_{n}^{\sigma} i_{L_{\delta}(\ii, \sigma_{n+1})+1}, \ldots,\Pi_{n}^{\sigma} i_{L_{\delta}(\ii, \sigma_{n})}\right)$ having type $\mathbf{p}_{\sigma_n}$ (for $1\leq n\leq d$).
\end{proof}

\begin{lemma}\label{lem:2}
Fix $\delta>0$, $\sigma\in\mathcal{A}_{\delta}$ and a type $\mathbf{P}=(\mathbf{p}_{\sigma_d};\mathbf{p}_{\sigma_{d-1}};\ldots; \mathbf{p}_{\sigma_1})\in\mathcal{T}_\delta^{\sigma}$. Then for every $1\leq n\leq d$,
\begin{equation}\label{eq:30}
-C_n^{(d),\sigma}(\mathbf{P}) \cdot\log \delta \leq L_{\delta}(\ii,\sigma_n)-L_{\delta}(\ii,\sigma_{n+1}) \leq -\Big(1+\frac{\log \lambda_{\min}}{\log \delta}\Big)\cdot C_n^{(d),\sigma}(\mathbf{P}) \cdot\log \delta,
\end{equation}
where $\ii\in\Sigma_{\delta}^{\sigma}$ is such that $\tau_\delta^{\sigma}(\ii)=\mathbf{P}$ and $\lambda_{\min}\coloneqq \min_{i,n}\lambda_i^{(n)}>0$. 
\end{lemma}

\begin{proof}
Recall the abbreviation $|\ii(\delta,n)|= L_{\delta}(\ii,\sigma_n)-L_{\delta}(\ii,\sigma_{n+1})$ and that $L_{\delta}(\ii,\sigma_{d+1})=0$. From the definition~\eqref{eq:10} of the $\delta$-stopping of $\ii\in\Sigma_{\delta}^{\sigma}$ in each coordinate $1\leq \sigma_n\leq d$,
\begin{equation*}
\prod_{m=n}^{d}\,\, \prod_{\ell=L_{\delta}(\ii,\sigma_{m+1})+1}^{L_{\delta}(\ii,\sigma_m)} \lambda_{i_{\ell}}^{(\sigma_n)} = \prod_{\ell=1}^{L_{\delta}(\ii,\sigma_n)} \lambda_{i_{\ell}}^{(\sigma_n)} \leq \delta < \lambda_{\min}^{-1}\cdot \prod_{m=n}^{d}\,\, \prod_{\ell=L_{\delta}(\ii,\sigma_{m+1})+1}^{L_{\delta}(\ii,\sigma_m)} \lambda_{i_{\ell}}^{(\sigma_n)} .
\end{equation*}
In particular, if $\tau_\delta^{\sigma}(\ii)=\mathbf{P}=(\mathbf{p}_{\sigma_d};\mathbf{p}_{\sigma_{d-1}};\ldots; \mathbf{p}_{\sigma_1})\in\mathcal{T}_\delta^{\sigma}$, then after taking logarithms
\begin{multline*}
\sum_{m=n}^{d} |\ii(\delta,m)|\cdot \chi_n^{\sigma}(\mathbf{p}_{\sigma_m}) 
= \sum_{m=n}^{d}\; |\ii(\delta,m)| \frac{-1}{|\ii(\delta,m)|}\sum_{\ell=L_{\delta}(\iih,m+1)+1}^{L_{\delta}(\iih,m)} \log \lambda_{i_{\ell}}^{(\sigma_n)} \\
\geq -\log \delta 
> \log \lambda_{\min} +\sum_{m=n}^{d} |\ii(\delta,m)|\cdot \chi_n^{\sigma}(\mathbf{p}_{\sigma_m}).
\end{multline*}
Expressing $|\ii(\delta,n)|$, we obtain
\begin{multline*}
\frac{-1}{\chi_n^{\sigma}(\mathbf{p}_{\sigma_n})} \bigg( \log \delta + \sum_{m=n+1}^{d} |\ii(\delta,m)|\cdot \chi_n^{\sigma}(\mathbf{p}_{\sigma_m}) \bigg) \\
\leq |\ii(\delta,n)| <
\frac{-1}{\chi_n^{\sigma}(\mathbf{p}_{\sigma_n})} \bigg( \log \delta + \log \lambda_{\min} + \sum_{m=n+1}^{d} |\ii(\delta,m)|\cdot \chi_n^{\sigma}(\mathbf{p}_{\sigma_m}) \bigg).
\end{multline*}
We continue by induction on decreasing $n$ starting from $n=d$. In this case
\begin{equation*}
\frac{-\log \delta}{\chi_d^{\sigma}(\mathbf{p}_{\sigma_d})} \leq |\ii(\delta,d)| < \frac{-\log \delta}{\chi_d^{\sigma}(\mathbf{p}_{\sigma_d})} \Big(1+\frac{\log \lambda_{\min}}{\log \delta}\Big), \;\text{ giving } C_d^{(d),\sigma}(\mathbf{P})=\frac{1}{\chi_d^{\sigma}(\mathbf{p}_{\sigma_d})}.
\end{equation*}
Next, we assume~\eqref{eq:30} for $m\in\{n+1,\ldots,d\}$ and prove the claim for $n\leq d-1$:
\begin{align*}
|\ii(\delta,n)| &< 
\frac{-1}{\chi_n^{\sigma}(\mathbf{p}_{\sigma_n})} \bigg(\! \log \delta + \log \lambda_{\min} - \sum_{m=n+1}^{d} \!\Big(1+\frac{\log \lambda_{\min}}{\log \delta}\Big)\cdot C_m^{(d),\sigma}(\mathbf{P}) \chi_n^{\sigma}(\mathbf{p}_{\sigma_m}) \cdot\log \delta \bigg) \\
&= \Big(1+\frac{\log \lambda_{\min}}{\log \delta}\Big)
\bigg(\! 1-\!\sum_{m=n+1}^{d} C_m^{(d),\sigma}(\mathbf{P})\cdot \chi_n^{\sigma}(\mathbf{p}_{\sigma_m})\! \bigg) \frac{-\log \delta}{\chi_n^{\sigma}(\mathbf{p}_{\sigma_n})}\\
&\stackrel{\eqref{eq:15}}{=} -\Big(1+\frac{\log \lambda_{\min}}{\log \delta}\Big)C_n^{(d),\sigma}(\mathbf{P}) \cdot\log \delta.
\end{align*}
The lower bound for $|\ii(\delta,n)|$ is the same without the $\log \lambda_{\min}$.
\end{proof}

\begin{lemma}\label{lem:4}
For any $\sigma\in\mathcal{S}_d$, we have $\sigma\in\mathcal{A}$ if and only if $\mathcal{Q}^{\sigma}\neq\emptyset$. Moreover, $\mathcal{T}_{\delta}^{\sigma}$ becomes dense in $\mathcal{Q}^{\sigma}$ as $\delta\to 0$.
\end{lemma}
\begin{proof}
If $\sigma\in\mathcal{A}$, then for some $\delta>0$ there exists a $\delta$-approximate cube $B_{\delta}(\ii)\in\mathcal{B}_{\delta}^{\sigma}$ which is $\sigma$-ordered and whose type $\tau_{\delta}^{\sigma}(\ii)\in\mathcal{T}_{\delta}^{\sigma}$. By Lemma~\ref{lem:2}, for this type $\tau_{\delta}^{\sigma}(\ii)$, we have $C_n^{(d),\sigma}(\tau_{\delta}^{\sigma}(\ii))\geq 0$ for all $1\leq n\leq d$, implying $\tau_{\delta}^{\sigma}(\ii)\in\mathcal{Q}^{\sigma}$.

Conversely, if $\mathbf{P}_{\!\sigma}\in\mathcal{Q}^{\sigma}$, then for $\delta$ small enough, we construct $\widetilde{\mathbf{P}}_{\delta}^{\sigma}=(\widetilde{\mathbf{p}}_{\sigma_d};\ldots;\widetilde{\mathbf{p}}_{\sigma_1})$, where $\widetilde{\mathbf{p}}_{\sigma_n}=(\widetilde{p}_{\sigma_n}(i))_{i\in\mathcal{I}_n^{\sigma}}$ is such that $\widetilde{\mathbf{P}}_{\delta}^{\sigma}\in\mathcal{T}_{\delta}^{\sigma}$, implying $\sigma\in\mathcal{A}$. Set $\widetilde{p}_{\sigma_n}(i) \coloneqq A_{\sigma_n}(i)/B_{\sigma_n}$, where $A_{\sigma_n}(i)$ for $i\in\mathcal{I}_n^{\sigma}\setminus\{1\}$ is the unique integer for which
\begin{equation}\label{eq:40}
\frac{A_{\sigma_n}(i)}{B_{\sigma_n}}\leq p_{\sigma_n}(i) < \frac{A_{\sigma_n}(i)+1}{B_{\sigma_n}} \;\text{ and }\; A_{\sigma_n}(1) = B_{\sigma_n} - \sum_{i\in\mathcal{I}_n^{\sigma}\setminus\{1\}}A_{\sigma_n}(i),
\end{equation}
moreover, $-C_n^{(d),\sigma}(\mathbf{P}_{\!\sigma}) \cdot\log \delta \leq B_{\sigma_n} \leq -C_n^{(d),\sigma}(\mathbf{P}_{\!\sigma}) \cdot\log \delta - C_n^{(d),\sigma}(\mathbf{P}_{\!\sigma})\cdot \log \lambda_{\min}$ is chosen by Lemma~\ref{lem:2} so that $\widetilde{\mathbf{P}}_{\delta}^{\sigma}\in\mathcal{T}_{\delta}^{\sigma}$. By construction, $|\widetilde{p}_{\sigma_n}(i)-p_{\sigma_n}(i)|=\mathcal{O}\big( (-\log \delta)^{-1} \big)$, in particular, $\widetilde{\mathbf{P}}_{\delta}^{\sigma}\to \mathbf{P}_{\!\sigma}$ coordinate-wise in every component as $\delta\to 0$. Since $\mathbf{P}_{\!\sigma}\in\mathcal{Q}^{\sigma}$ was arbitrary, we conclude that $\mathcal{T}_{\delta}^{\sigma}$ becomes dense in $\mathcal{Q}^{\sigma}$ as $\delta\to 0$.
\end{proof}

\begin{lemma}\label{lem:3}
Fix $\varepsilon_0>0$. There exists $\delta_0(\varepsilon_0)>0$ such that for all $\sigma\in\mathcal{A}$ and $\delta<\delta_0(\varepsilon_0)$ there exists $\widetilde{\mathbf{P}}_{\delta}^{\sigma}\in\mathcal{T}_{\delta}^{\sigma}$ for which
\begin{equation*}
t(\widetilde{\mathbf{P}}_{\delta}^{\sigma}) > \sup _{\mathbf{P}_{\!\sigma} \in \mathcal{Q}^{\sigma}} t(\mathbf{P}_{\!\sigma})-\varepsilon_0. 
\end{equation*}
\end{lemma}
\begin{proof}
Continuity of $t(\mathbf{P}_{\!\sigma})$ for every $\sigma\in\mathcal{A}$ implies that there exist $\widehat{\mathbf{P}}^{\sigma}\in\mathcal{Q}^{\sigma}$ such that $t(\widehat{\mathbf{P}}^{\sigma})>\sup _{\mathbf{P}_{\!\sigma} \in \mathcal{Q}^{\sigma}} t(\mathbf{P}_{\!\sigma})-\varepsilon_0/2$. For $\widehat{\mathbf{P}}^{\sigma}\in\mathcal{Q}^{\sigma}$ we construct $\widetilde{\mathbf{P}}_{\delta}^{\sigma}\in\mathcal{T}_{\delta}^{\sigma}$ as we did in~\eqref{eq:40}. By Lemma~\ref{lem:4} and continuity of $t(\mathbf{P}_{\!\sigma})$, we can choose $\delta_0(\varepsilon_0)>0$ such that $t(\widetilde{\mathbf{P}}_{\delta}^{\sigma}) > t(\widehat{\mathbf{P}}^{\sigma}) - \varepsilon_0/2$ for every $\delta<\delta_0(\varepsilon_0)$.
\end{proof}

\subsection{Proof of Theorem~\ref{thm:main1}}\label{subsec:31}

Recall the definition of $P(\boldsymbol{\varphi})$ from~\eqref{eq:13}. Fix $\delta>0$ and $\sigma\in\mathcal{A}_{\delta}$. For any type $\mathbf{P}=(\mathbf{p}_{\sigma_d};\mathbf{p}_{\sigma_{d-1}};\ldots; \mathbf{p}_{\sigma_1})\in\mathcal{T}_\delta^{\sigma}$, observe that all approximate cubes $B_{\delta}(\ii)$ in its type class $T_\delta^{\sigma}(\mathbf{P})$ have the same value for $\Phi\left(B_{\delta}(\ii)\right)$, namely,
\begin{equation*}
\Phi\left(B_{\delta}(\ii)\right) = \sum_{n=1}^{d} 
\sum_{\ell=L_{\delta}(\ii, \sigma_{n+1})+1}^{L_{\delta}(\ii, \sigma_{n})} \varphi_{n}^{\sigma}\left(\Pi_{n}^{\sigma} i_{\ell}\right) =
\sum_{n=1}^{d} |\ii(\delta,n)| \cdot \int\! \varphi_n^{\sigma}\,\mathrm{d}\mathbf{p}_{\sigma_n},
\end{equation*}
where recall $|\ii(\delta,n)|= L_{\delta}(\ii,\sigma_n)-L_{\delta}(\ii,\sigma_{n+1})$. Hence, grouping according to type class,
\begin{equation*}
Z_{\delta}^{\sigma}(\boldsymbol{\varphi})\coloneqq \sum_{B_{\delta}(\ii) \in \mathcal{B}_{\delta}^{\sigma}} \exp \left[\Phi\left(B_{\delta}(\ii)\right)\right] = 
\sum_{\mathbf{P}\in\mathcal{T}_\delta^{\sigma}} \# T_\delta^{\sigma}(\mathbf{P}) \cdot  \exp \left[ \sum_{n=1}^{d} |\ii(\delta,n)| \cdot \int\! \varphi_n^{\sigma}\,\mathrm{d}\mathbf{p}_{\sigma_n} \right],
\end{equation*}
where $\ii\in\Sigma_{\delta}^{\sigma}$ is such that $\tau_\delta^{\sigma}(\ii)=\mathbf{P}$.
Using Lemma~\ref{lem:1} and \ref{lem:2}, we bound $Z_{\delta}^{\sigma}(\boldsymbol{\varphi})$ from above:
\begin{align}
Z_{\delta}^{\sigma}(\boldsymbol{\varphi}) &\stackrel{\eqref{eq:33}}{\leq} \sum_{\mathbf{P}\in\mathcal{T}_\delta^{\sigma}}  \exp \left[ \sum_{n=1}^{d} |\ii(\delta,n)| \cdot \bigg( H(\mathbf{p}_{\sigma_n})+\int\! \varphi_n^{\sigma}\,\mathrm{d}\mathbf{p}_{\sigma_n} \bigg) \right] \nonumber \\
&\stackrel{\eqref{eq:32}}{\leq} \prod_{n=1}^d \Big( \max_{B_\delta(\ii)\in\mathcal{B}_\delta^{\sigma}} |\ii(\delta,n)|+1 \Big)^{\#\mathcal{I}_n^{\sigma}+1} \!\!\cdot  \max_{\mathbf{P}\in\mathcal{T}_\delta^{\sigma}} \exp\! \left[\sum_{n=1}^{d} |\ii(\delta,n)| \cdot \!\bigg(\! H(\mathbf{p}_{\sigma_n})+\!\int\! \varphi_n^{\sigma}\,\mathrm{d}\mathbf{p}_{\sigma_n} \bigg) \right] \nonumber \\
&\stackrel{\eqref{eq:30}}{\leq} \mathcal{O}\big( (-\log \delta)^{d(N+1)} \big) \cdot  \max_{\mathbf{P}\in\mathcal{T}_\delta^{\sigma}}\, \delta^{-\sum_{n=1}^{d}  \left(1+\mathcal{O}((-\log \delta)^{-1})\right) C_n^{(d),\sigma}(\mathbf{P}) \cdot \left( H(\mathbf{p}_{\sigma_n})+\int\! \varphi_n^{\sigma}\,\mathrm{d}\mathbf{p}_{\sigma_n} \right)  }, \label{eq:34}
\end{align}
and also from below:
\begin{align}
Z_{\delta}^{\sigma}(\boldsymbol{\varphi}) &\stackrel{\eqref{eq:33}}{\geq} \prod_{n=1}^d \big(|\ii(\delta,n)|+1\big)^{-\#\mathcal{I}_n^{\sigma}} \cdot
\max_{\mathbf{P}\in\mathcal{T}_\delta^{\sigma}}  \exp \left[ \sum_{n=1}^{d} |\ii(\delta,n)| \cdot \bigg( H(\mathbf{p}_{\sigma_n})+\int\! \varphi_n^{\sigma}\,\mathrm{d}\mathbf{p}_{\sigma_n} \bigg) \right]   \nonumber \\
&\stackrel{\eqref{eq:30}}{\geq} \mathcal{O}\big( (-\log \delta)^{-dN} \big) \cdot  \max_{\mathbf{P}\in\mathcal{T}_\delta^{\sigma}}\, \delta^{-\sum_{n=1}^{d}  \left(1+\mathcal{O}((-\log \delta)^{-1})\right) C_n^{(d),\sigma}(\mathbf{P}) \cdot \left( H(\mathbf{p}_{\sigma_n})+\int\! \varphi_n^{\sigma}\,\mathrm{d}\mathbf{p}_{\sigma_n} \right)  }. \label{eq:35}
\end{align}
Since $\delta^{-1}>1$, the type $\mathbf{P}\in\mathcal{T}_\delta^{\sigma}$ which maximises the expression is the same which maximises the sum in the exponent. 

Recall $t(\mathbf{P}_{\!\sigma})$ from~\eqref{eq:36}. We are now ready to bound the pressure from above,
\begin{align*}
\overline{P}(\boldsymbol{\varphi}) &= \limsup_{\delta \rightarrow 0} \frac{-1}{\log \delta}\, 
\log \bigg[\sum_{\sigma \in \mathcal{A}_{\delta}} Z_{\delta}^{\sigma}(\boldsymbol{\varphi}) \bigg] \leq \limsup_{\delta \rightarrow 0} \frac{-1}{\log \delta}\, 
\log \bigg[d! \cdot \max_{\sigma \in \mathcal{A}_{\delta}} Z_{\delta}^{\sigma}(\boldsymbol{\varphi}) \bigg] \\
&\stackrel{\eqref{eq:34}}{\leq} \limsup_{\delta \rightarrow 0}\, \max_{\sigma \in \mathcal{A}_{\delta}} \, \max_{\mathbf{P}_{\!\sigma}\in\mathcal{T}_\delta^{\sigma}}\,
\sum_{n=1}^{d}  \left(1+\mathcal{O}((-\log \delta)^{-1})\right) C_n^{(d),\sigma}(\mathbf{P}_{\!\sigma}) \cdot \bigg( H(\mathbf{p}_{\sigma_n})+\int\! \varphi_n^{\sigma}\,\mathrm{d}\mathbf{p}_{\sigma_n} \bigg) \\
&\leq \max_{\sigma \in \mathcal{A}}\, \sup_{\mathbf{P}_{\! \sigma}\in\mathcal{Q}^{\sigma}}\, t(\mathbf{P}_{\! \sigma}) \cdot \big( 1+ \lim _{\delta \rightarrow 0}\,\mathcal{O}((-\log \delta)^{-1}) \big) = \max_{\sigma \in \mathcal{A}}\, \sup_{\mathbf{P}_{\! \sigma}\in\mathcal{Q}^{\sigma}}\, t(\mathbf{P}_{\! \sigma}),
\end{align*}
where the last inequality holds because $\mathcal{A}_{\delta}\subseteq\mathcal{A}$ and $\mathcal{T}_\delta^{\sigma}\subset \mathcal{Q}^{\sigma}$. Similarly,
\begin{equation*}
\underline{P}(\boldsymbol{\varphi}) \geq \liminf_{\delta \rightarrow 0} \frac{\log\! \big[\!\max_{\sigma \in \mathcal{A}_{\delta}} Z_{\delta}^{\sigma}(\boldsymbol{\varphi}) \big]}{-\log \delta} 
\!\stackrel{\eqref{eq:35}}{\geq} \liminf_{\delta \rightarrow 0}\, \max_{\sigma \in \mathcal{A}_{\delta}} \, \max_{\mathbf{P}_{\!\sigma}\in\mathcal{T}_\delta^{\sigma}}\,
\left(1+\mathcal{O}((-\log \delta)^{-1})\right) t(\mathbf{P}_{\!\sigma}).
\end{equation*}
We are only interested in the limit as $\delta\to 0$, hence, we may assume that $\delta<\delta_0(\varepsilon_0)$ given by Lemma~\ref{lem:3}. Using the type $\widetilde{\mathbf{P}}_{\delta}^{\sigma}\in\mathcal{T}_{\delta}^{\sigma}$ constructed in Lemma~\ref{lem:3}, we conclude,
\begin{equation*}
\liminf_{\delta \rightarrow 0}\, \max_{\sigma \in \mathcal{A}_{\delta}} \, \max_{\mathbf{P}_{\!\sigma}\in\mathcal{T}_\delta^{\sigma}}
\left(1+\mathcal{O}((-\log \delta)^{-1})\right) t(\mathbf{P}_{\!\sigma})
\geq  \liminf_{\delta \rightarrow 0}\, \max_{\sigma \in \mathcal{A}_{\delta}} \, 
t(\widetilde{\mathbf{P}}_{\delta}^{\sigma})
\geq \max_{\sigma \in \mathcal{A}}\, \sup _{\mathbf{P}_{\!\sigma} \in \mathcal{Q}^{\sigma}} t(\mathbf{P}_{\!\sigma})-\varepsilon_0. 
\end{equation*}
Since $\varepsilon_0$ is arbitrary, this shows that $\overline{P}(\boldsymbol{\varphi})=\underline{P}(\boldsymbol{\varphi})$, implying that the limit $P(\boldsymbol{\varphi})$ exists and is equal to $\max_{\sigma \in \mathcal{A}}\, \sup _{\mathbf{P}_{\!\sigma} \in \mathcal{Q}^{\sigma}} t(\mathbf{P}_{\!\sigma})$, which concludes the proof of Theorem~\ref{thm:main1}.

\subsection{Proof of Corollary~\ref{cor:main}}

The upper bound $P(\boldsymbol{\varphi}) \leq \max _{\sigma \in \mathcal{A}}\, T_d^{\sigma}$ follows from Proposition~\ref{prop:main} since $\mathcal{Q}^{\sigma}\subseteq \mathcal{P}^{\sigma}$. If $\mathcal{A}=\{\sigma\}$ then $\mathcal{Q}^{\omega}=\emptyset$ for all $\omega\neq \sigma$ by Lemma~\ref{lem:4} which implies that $\mathcal{Q}^{\sigma}=\mathcal{P}^{\sigma}$. Hence, Proposition~\ref{prop:main} implies that in this case $P(\boldsymbol{\varphi}) = T_d^{\sigma}$. The proof is complete.

\section{Proof of Theorem~\ref{thm:Lqmain}}\label{sec:proofLq}

In what follows, we write $A\lesssim B$ if there exists a constant $c$ depending only on the sponge $F$ such that $A\leq cB$. Similarly, $A\gtrsim B$ if $A\geq cB$ and $A\approx B$ if $A\lesssim B$ and $A\gtrsim B$. For example, if $\ii$ is $\sigma$-ordered at scale $\delta$ with type $\tau_{\delta}^{\sigma}(\ii)$ then the conclusion of Lemma~\ref{lem:2} can be written as $L_{\delta}(\ii,\sigma_n)-L_{\delta}(\ii,\sigma_{n+1}) \approx -C_n^{(d),\sigma}(\tau_{\delta}^{\sigma}(\ii)) \cdot\log \delta$ for $1\leq n\leq d$. Recall, 	$\widetilde{\nu}_{\boldsymbol{\mu}}$ denotes the Bernoulli measure $\boldsymbol{\mu}^{\mathbb{N}}$ and $\nu_{\boldsymbol{\mu}}=\widetilde{\nu}_{\boldsymbol{\mu}}\circ \pi^{-1}$ is its push-forward. 

\begin{lemma}\label{lem:MeasureSymbCube}
Assume $\ii$ is $\sigma$-ordered at scale $\delta$. Then
\begin{equation*}
	\widetilde{\nu}_{\boldsymbol{\mu}}(B_{\delta}(\ii)) = \prod_{n=1}^{d} \prod_{\ell=L_{\delta}(\ii,\sigma_{n+1})+1}^{L_{\delta}(\ii,\sigma_{n})} \mu_n^{\sigma}(\Pi_n^{\sigma}i_{\ell}) 
	\approx \delta^{S(\tau_{\delta}^{\sigma}(\ii))}.
\end{equation*}
If $F$ satisfies the SPPC then $\nu_{\boldsymbol{\mu}}(\pi(B_{\delta}(\ii)))=\widetilde{\nu}_{\boldsymbol{\mu}}(B_{\delta}(\ii))$. 
\end{lemma}
\begin{proof}
We start with the first equality. From definition~\eqref{eq:29} of $B_{\delta}(\ii)$ it follows that an approximate cube is the disjoint union of level $L_{\delta}(\ii,\sigma_1)$ cylinder sets:
\begin{equation*}
\big\{ [j_1,\ldots,j_{L_{\delta}(\ii,\sigma_1)}]:\, \Pi_{n}^{\sigma}j_{\ell}=\Pi_n^{\sigma}i_{\ell} \text{ for } \ell=L_{\delta}(\ii,\sigma_{n+1})+1,\ldots,L_{\delta}(\ii,\sigma_n) \text{ and } 1\leq n\leq d \big\}.
\end{equation*}
For each such cylinder, $\widetilde{\nu}_{\boldsymbol{\mu}}([j_1,\ldots,j_{L_{\delta}(\ii,\sigma_1)}]) = \prod_{\ell=1}^{L_{\delta}(\ii,\sigma_1)}\mu(j_{\ell})$. Adding up and using multiplicativity, we obtain
\begin{equation*}
	\widetilde{\nu}_{\boldsymbol{\mu}}(B_{\delta}(\ii)) = \prod_{n=1}^{d}\, \prod_{\ell=L_{\delta}(\ii,\sigma_{n+1})+1}^{L_{\delta}(\ii,\sigma_{n})}\, \sum_{j\in\mathcal{I}:\, \Pi_n^{\sigma}j=\Pi_n^{\sigma}i_{\ell}}\mu(j) = \prod_{n=1}^{d}\, \prod_{\ell=L_{\delta}(\ii,\sigma_{n+1})+1}^{L_{\delta}(\ii,\sigma_{n})}\, \mu_n^{\sigma}(\Pi_n^{\sigma}i_{\ell}).
\end{equation*}
The second relation is a direct consequence of Lemma~\ref{lem:2}:
\begin{align*}
\prod_{n=1}^{d}\, \prod_{\ell=L_{\delta}(\ii,\sigma_{n+1})+1}^{L_{\delta}(\ii,\sigma_{n})}\, \mu_n^{\sigma}(\Pi_n^{\sigma}i_{\ell})
&= \prod_{n=1}^{d}\, \prod_{i\in\mathcal{I}_n^{\sigma}}\, \mu_n^{\sigma}(i)^{(L_{\delta}(\ii,\sigma_n)-L_{\delta}(\ii,\sigma_{n+1}))\cdot \tau_{\delta}(\ii,\sigma_n)(i)} \\
&\stackrel{\eqref{eq:30}}{\approx} \delta^{-\sum_{n=1}^{d}C_n^{(d),\sigma}(\tau_{\delta}^{\sigma}(\ii))\sum_{i\in\mathcal{I}_n^{\sigma}} \tau_{\delta}(\ii,\sigma_n)(i)\cdot \log \mu_n^{\sigma}(i)} = \delta^{S(\tau_{\delta}^{\sigma}(\ii))}.
\end{align*} 
A detailed argument for the last claim can be found in the proof of~\cite[Corollary~2.8]{BARANYAnti201788}. We present a sketch. Let $D\coloneqq \{ x\in F:\, \text{there exist } \ii\neq\jj\in\Sigma \text{ such that } x=\pi(\ii)=\pi(\jj) \}$. If $x\in D$ then the SPPC implies that $x$ must lie on the boundary $\partial f_{i_1\ldots i_n}([0,1]^d)$ of some cylinder set and so $D\subseteq \bigcup_{n=0}^{\infty}\bigcup_{i_1\ldots i_n}\partial f_{i_1\ldots i_n}([0,1]^d)$. It is easy to see that $\nu_{\boldsymbol{\mu}}(\partial[0,1]^d)=0$, therefore, $\nu_{\boldsymbol{\mu}}(D)=0$ which also implies $\nu_{\boldsymbol{\mu}}(\pi(B_{\delta}(\ii)))=\widetilde{\nu}_{\boldsymbol{\mu}}(B_{\delta}(\ii))$. 
\end{proof}
An immediate corollary of Lemma~\ref{lem:MeasureSymbCube} and definition~\eqref{eq:38} of the potential $\boldsymbol{\psi}_q^{\boldsymbol{\mu}}$ is that (when assuming the SPPC) for any approximate cube
\begin{equation*}
\exp \left[\Phi\left(B_{\delta}(\ii)\right)\right]=\big(\widetilde{\nu}_{\boldsymbol{\mu}}\left(B_{\delta}(\ii)\right)\big)^q = \big( \nu_{\boldsymbol{\mu}}(\pi(B_{\delta}(\ii))) \big)^q.
\end{equation*}
As a result, the pressure $P(\boldsymbol{\psi}_q^{\boldsymbol{\mu}})$ can be interpreted as the `symbolic $L^q$ spectrum' of $\widetilde{\nu}_{\boldsymbol{\mu}}$. It remains to transfer this result to the actual $L^q$ spectrum $T(\nu_{\boldsymbol{\mu}},q)$ of $\nu_{\boldsymbol{\mu}}$.

A Euclidean ball centred in $F$ can always be drawn around the image of an approximate cube since $\pi(B_{\delta}(\ii))$ is contained in a hypercube of side length $\delta$. In particular, for all $\ii\in\Sigma$ and $\delta>0$,
\begin{equation}\label{eq:65}
	\pi(B_{\delta}(\ii)) \subseteq B(\pi(\jj),\sqrt{d}\cdot \delta) \;\text{ with any } \jj\in B_{\delta}(\ii).
\end{equation}
However, it is not necessarily true that there exists a uniform constant $c$ such that for any choice of $\ii$, the image $\pi(B_{\delta}(\ii))$ contains a ball centred in $F$ with radius $c \delta$. With a small perturbation of $\ii$ this is possible. Recall our standing assumption~\eqref{eq:19}. We define an injective function $\alpha_{\delta}: \Sigma \to \Sigma$ as follows. If $\ii\in\Sigma_{\delta}^{\sigma}$ then $\alpha_{\delta}(\ii)=\alpha_{\delta}(i)_1,\alpha_{\delta}(i)_2,\ldots$ is defined by the sequence
\begin{multline}\label{eq:73}
i_{1}, \ldots, i_{L_{\delta}(\ii, \sigma_{d})}, k_0^{(\sigma_{d})}, k_1^{(\sigma_{d})}, \ldots, i_{L_{\delta}(\ii, \sigma_{n+1})+1}, \ldots, i_{L_{\delta}(\ii, \sigma_{n})}, k_0^{(\sigma_{n})}, k_1^{(\sigma_{n})}, \ldots,  \\
i_{L_{\delta}(\ii, \sigma_{2})+1}, \ldots, i_{L_{\delta}(\ii, \sigma_{1})}, k_0^{(\sigma_{1})}, k_1^{(\sigma_{1})}, i_{L_{\delta}(\ii,\sigma_{1})+1}, i_{L_{\delta}(\ii,\sigma_{1})+2},\ldots.
\end{multline}
In other words, the pair $k_0^{(\sigma_{n})}, k_1^{(\sigma_{n})}$ is inserted after $i_{L_{\delta}(\ii, \sigma_{n})}$ for each $n=d,d-1,\ldots,1$ (even if $L_{\delta}(\ii, \sigma_{n})=L_{\delta}(\ii, \sigma_{n+1})$), otherwise $\ii$ is left unchanged. This small perturbation of $\ii$ has two useful consequences given in the following lemma. Let $\overline{\pi(B_{\delta}(\ii))}$ denote the smallest axis parallel hyper-rectangle which contains $\pi(B_{\delta}(\ii))$. 

\begin{lemma}\label{lem:BallInCube}
For every $\delta>0$ small enough, $\sigma\in\mathcal{A}$ and $\ii\in\Sigma_{\delta}^{\sigma}$, 
\begin{equation}\label{eq:70}
\widetilde{\nu}_{\boldsymbol{\mu}}(B_{\delta}(\alpha_{\delta}(\ii))) \approx \widetilde{\nu}_{\boldsymbol{\mu}}(B_{\delta}(\ii)),
\end{equation}
moreover, there exists a constant $0<C_0=C_0(F)<1$ such that
\begin{equation}\label{eq:71}
B(\pi(\jj), C_0\cdot \delta) \subset \overline{\pi(B_{\delta}(\alpha_{\delta}(\ii)))} \;\text{ for every } \jj \text{ with } |\jj\wedge \alpha_{\delta}(\ii)| \geq L_{\delta}(\ii,\sigma_{1})+2d.
\end{equation}
\end{lemma} 

\begin{proof}
We begin with~\eqref{eq:70}. The insertion of $k_0^{(\sigma_{n})}, k_1^{(\sigma_{n})}$ implies that for each $1\leq n\leq d$,
\begin{equation}\label{eq:72}
	0\leq L_{\delta}(\ii,\sigma_{n}) - L_{\delta}(\alpha_{\delta}(\ii),\sigma_{n}) \leq 2(d-n)\frac{\log \lambda_{\min}}{\log \lambda_{\max}}.
\end{equation}
In particular, for any $1\leq m<n \leq d-1$, if $L_{\delta}(\ii,\sigma_{n-m})-L_{\delta}(\ii,\sigma_{n})> 2(d-n+m)\frac{\log \lambda_{\min}}{\log \lambda_{\max}}$, then $L_{\delta}(\alpha_{\delta}(\ii),\sigma_{n})<L_{\delta}(\alpha_{\delta}(\ii),\sigma_{n-m})$, hence, $\sigma_{n-m}$ still precedes $\sigma_{n}$ in the ordering of $\alpha_{\delta}(\ii)$ at scale $\delta$. Therefore, two coordinates $n,m\in\{1,\ldots,d\}$ can potentially switch their order in the ordering of $\ii$ and the ordering of $\alpha_{\delta}(\ii)$ at scale $\delta$ only if $|L_{\delta}(\ii,n)-L_{\delta}(\ii,m)|$ was smaller than a uniformly bounded constant (independent of $\ii$ and $\delta$). As a result, from Lemma~\ref{lem:MeasureSymbCube} it follows that calculating $\widetilde{\nu}_{\boldsymbol{\mu}}(B_{\delta}(\alpha_{\delta}(\ii)))$ involves multiplying the same terms as in $\widetilde{\nu}_{\boldsymbol{\mu}}(B_{\delta}(\ii))$ apart from a uniformly bounded number of terms (that come from potentially switching orders), hence, claim~\eqref{eq:70} follows.

To show~\eqref{eq:71} let $\mathcal{H}_u^{(n)}\coloneqq\{(x_1,\ldots,x_d)\in[0,1]^d:\, x_n=u \}$ and $\ii|k=i_1,\ldots,i_k$. Since there is a uniform upper bound on $L_{\delta}(\ii,\sigma_{n}) - L_{\delta}(\alpha_{\delta}(\ii),\sigma_{n})$ from~\eqref{eq:72}, it follows that the hyper-rectangle $f_{\alpha_{\delta}(\ii)|(L_{\delta}(\ii,\sigma_n)+2(d-n))}([0,1]^d)$ has height $\approx\delta$ in coordinate $\sigma_{n}$. Therefore, the repeated insertion of $k_0^{(\sigma_{n})}, k_1^{(\sigma_{n})}$ after $i_{L_{\delta}(\ii, \sigma_{n})}$ implies from~\eqref{eq:19} that
\begin{equation*}
	\mathrm{dist}\Big( f_{\alpha_{\delta}(\ii)|(L_{\delta}(\ii,\sigma_n)+2(d-n+1))}([0,1]^d),  \bigcup_{n\leq \ell\leq d}  f_{\alpha_{\delta}(\ii)|(L_{\delta}(\ii,\sigma_\ell)+2(d-\ell))} \big(\mathcal{H}_0^{(\sigma_{\ell})}\cup \mathcal{H}_1^{(\sigma_{\ell})}\big) \Big) \gtrsim r_0^2\cdot \delta.
\end{equation*}
In particular, for $n=1$, we obtain using~\eqref{eq:72} that there exist a uniform constant $C_0$ such that for every $\jj$ with $|\jj\wedge \alpha_{\delta}(\ii)| \geq L_{\delta}(\ii,\sigma_{1})+2d$,
\begin{equation*}
B(\pi(\jj), C_0\cdot \delta) \subset \bigcup_{1\leq \ell\leq d}  f_{\alpha_{\delta}(\ii)|(L_{\delta}(\ii,\sigma_\ell)+2(d-\ell))} \big(\mathcal{H}_0^{(\sigma_{\ell})}\cup \mathcal{H}_1^{(\sigma_{\ell})}\big) \subset \overline{\pi(B_{\delta}(\alpha_{\delta}(\ii)))}.
\end{equation*}
\end{proof}

\subsection{Proof of Theorem~\ref{thm:Lqmain}, upper bound}\label{sec:upper}

Let $\{B(x_{\ell},\delta)\}_{\ell}$ be a centred packing of the self-affine sponge $F$ satisfying the SPPC. Let $\mathcal{B}_{\delta}$ be the set of all symbolic $\delta$-approximate cubes and $\mathcal{B}_{\delta}^x$ be the set of those cubes whose image under $\pi$ intersect $B(x,\delta)\cap F$. Since each edge of $\overline{\pi(B_{\delta}(\ii))}$ has length at least $\lambda_{\min} \delta$, moreover, $\overline{\pi(B_{\delta}(\ii))}$ and $\overline{\pi(B_{\delta}(\jj))}$ may intersect only on their boundary due to the SPPC, it follows that there exists a constant $N_0=N_0(F)$ such that $\#\mathcal{B}_{\delta}^x\leq N_0$ uniformly in $x$ and $\delta$. We split the proof into two parts depending on whether $q$ is negative or not. Note that if $A\subseteq B$ then $(\nu(A))^q\leq (\nu(B))^q$ if $q\geq 0$ for any probability measure $\nu$ and $(\nu(A))^q\geq (\nu(B))^q$ if $q< 0$.

First assume $q\geq 0$. Then for all elements of the packing
\begin{equation*}
(\nu_{\boldsymbol{\mu}}\big(B(x_{\ell},\delta))\big)^q = \big(\widetilde{\nu}_{\boldsymbol{\mu}}(\pi^{-1}(B(x_{\ell},\delta)))\big)^q \leq \big(\widetilde{\nu}_{\boldsymbol{\mu}}( \mathcal{B}_{\delta}^{x_{\ell}} )\big)^q.
\end{equation*} 
Furthermore, if we restrict to $q\in[0,1]$ then also 
\begin{equation}\label{eq:64}
\big(\widetilde{\nu}_{\boldsymbol{\mu}}( \mathcal{B}_{\delta}^{x_{\ell}} )\big)^q \leq \sum_{B\in\mathcal{B}_{\delta}^{x_{\ell}}} (\widetilde{\nu}_{\boldsymbol{\mu}}(B))^q.
\end{equation}
Since $\{B(x_{\ell},\delta)\}_{\ell}$ is a packing, there is a uniform bound $N_1$ on the number of different $B(x_{\ell},\delta)$ any one $\delta$-approximate cube $B$ can intersect. Therefore,
\begin{equation*}
\sum_{\ell} \big(\nu_{\boldsymbol{\mu}}(B(x_{\ell}, \delta))\big)^{q} \leq \sum_{\ell} \sum_{B\in\mathcal{B}_{\delta}^{x_{\ell}}} (\widetilde{\nu}_{\boldsymbol{\mu}}(B))^q \leq N_1 \sum_{B\in\mathcal{B}_{\delta}} (\widetilde{\nu}_{\boldsymbol{\mu}}(B))^q.
\end{equation*}
By Lemma~\ref{lem:MeasureSymbCube} and Theorem~\ref{thm:main1} the right hand side after taking $\log$ and dividing by $-\log \delta$ tends to $P(\boldsymbol{\psi}_q^{\boldsymbol{\mu}})$ as $\delta\to 0$ giving the desired upper bound. If $q>1$ then~\eqref{eq:64} holds in the opposite direction, however, we still have $\lesssim$ by Jensen's inequality for convex functions with the implied constant depending on $\#\mathcal{B}_{\delta}^{x_{\ell}}$ and $q$. To conclude as above, we use the uniform upper bound $\#\mathcal{B}_{\delta}^{x_{\ell}}\leq N_0$. The proof is complete for $q\geq 0$.

Now assume $q<0$. This time we use~\eqref{eq:65} to inscribe an approximate cube within each ball of the packing. Specifically, let $\ii_{\ell}\in\Sigma$ satisfy $\pi(\ii_{\ell})=x_{\ell}$ (if there is more than one, choose arbitrarily). Then according to~\eqref{eq:65} we have $\pi(B_{\delta/\sqrt{d}}(\ii_{\ell})) \subseteq B(x_{\ell}, \delta)$ and
\begin{equation*}
\sum_{\ell} \big(\nu_{\boldsymbol{\mu}}(B(x_{\ell}, \delta))\big)^{q} \leq \sum_{\ell} \big(\nu_{\boldsymbol{\mu}}(\pi(B_{\delta/\sqrt{d}}(\ii_{\ell})))\big)^{q} \leq \sum_{\ell} \big(\widetilde{\nu}_{\boldsymbol{\mu}}(B_{\delta/\sqrt{d}}(\ii_{\ell}))\big)^{q} 
\leq \sum_{B\in\mathcal{B}_{\delta/\sqrt{d}}} (\widetilde{\nu}_{\boldsymbol{\mu}}(B))^q,
\end{equation*}
where the second inequality holds because $B_{\delta/\sqrt{d}}(\ii_{\ell})\subseteq \pi^{-1}(\pi(B_{\delta/\sqrt{d}}(\ii_{\ell})))$. The upper bound follows after taking $\log$ of each side, dividing by $-\log \delta$ and taking the limit as $\delta\to 0$.

\subsection{Proof of Theorem~\ref{thm:Lqmain}, lower bound}

We write $t_q^{\boldsymbol{\mu}}(\mathbf{P}_{\!\sigma})$ to indicate that in definition \eqref{eq:36} of $t(\mathbf{P}_{\!\sigma})$ we use the potential $\boldsymbol{\psi}_q^{\boldsymbol{\mu}}$. We use the dominant type that `carries' the pressure $P(\boldsymbol{\psi}_q^{\boldsymbol{\mu}})$ to obtain the lower bound. The proof is split into two parts again depending on whether $q$ is negative or not.

First assume $q\geq 0$. Fix $\varepsilon>0$ and chose (any) $\sigma\in\mathcal{A}$ which maximises $\sup_{\mathbf{P}_{\! \sigma}\in\mathcal{Q}^{\sigma}} t_q^{\boldsymbol{\mu}}(\mathbf{P}_{\!\sigma})$. By Lemma~\ref{lem:4} and~\ref{lem:3}, for every $\delta$ small enough there exists a type $\tau_{\delta}^{\sigma}\in\mathcal{T}_{\delta}^{\sigma}$ such that 
\begin{equation*}
t_q^{\boldsymbol{\mu}}(\tau_{\delta}^{\sigma})\geq P(\boldsymbol{\psi}_q^{\boldsymbol{\mu}})-\varepsilon.
\end{equation*}
From~\eqref{eq:65} it follows that $\pi(B_{\delta}(\ii)) \subseteq B(\pi(\ii),\sqrt{d}\cdot \delta)$ for every $B_{\delta}(\ii)\in T_{\delta}^{\sigma}(\tau_{\delta}^{\sigma})$ (the type class of $\tau_{\delta}^{\sigma}$). We claim that there exists a constant $0<c=c(F)\leq 1$ independent of $\delta$ and a subset $\mathcal{C}_{\delta}^{\sigma}\subseteq T_{\delta}^{\sigma}(\tau_{\delta}^{\sigma})$ with the property that $\#\mathcal{C}_{\delta}^{\sigma}\geq c\cdot \#T_{\delta}^{\sigma}(\tau_{\delta}^{\sigma})$ and the balls $B(\pi(\ii),\sqrt{d}\cdot \delta)$ are pairwise disjoint for $B_{\delta}(\ii)\in \mathcal{C}_{\delta}^{\sigma}$. This is true for the same reason why $\#\mathcal{B}_{\delta}^{x}\leq N_0$ in Section~\ref{sec:upper}. In this case $B(\pi(\ii),2\sqrt{d}\cdot \delta)$ intersects at most $\widetilde{N}_0$ different $\overline{\pi(B_{\delta}(\jj))}$. The subset $\mathcal{C}_{\delta}^{\sigma}$ is constructed inductively by picking an element $B_{\delta}(\ii)\in T_{\delta}^{\sigma}(\tau_{\delta}^{\sigma})$, placing it in $\mathcal{C}_{\delta}^{\sigma}$ and removing any $B_{\delta}(\jj)\in T_{\delta}^{\sigma}(\tau_{\delta}^{\sigma})$ such that $\overline{\pi(B_{\delta}(\jj))}\cap B(\pi(\ii),2\sqrt{d}\cdot \delta)\neq\emptyset$. The process is repeated until all $B_{\delta}(\ii)\in T_{\delta}^{\sigma}(\tau_{\delta}^{\sigma})$ have either been placed in $\mathcal{C}_{\delta}^{\sigma}$ or removed. At each step at most $\widetilde{N}_0$ elements are removed, hence, $\#\mathcal{C}_{\delta}^{\sigma}\geq (\widetilde{N}_0)^{-1}\cdot \#T_{\delta}^{\sigma}(\tau_{\delta}^{\sigma})$. The extra factor of 2 in the radius ensures that $\{ B(\pi(\ii),\sqrt{d}\cdot \delta):\, B_{\delta}(\ii)\in \mathcal{C}_{\delta}^{\sigma} \}$ is a centred packing of $F$ which satisfies
\begin{equation*}
\delta^{-P(\boldsymbol{\psi}_q^{\boldsymbol{\mu}})+\varepsilon} \leq \delta^{-t_q^{\boldsymbol{\mu}}(\tau_{\delta}^{\sigma})} \lesssim \sum_{B_{\delta}(\ii)\in\mathcal{C}_{\delta}^{\sigma}} \big(\widetilde{\nu}_{\boldsymbol{\mu}}(B_{\delta}(\ii))\big)^{q}
\leq \sum_{B_{\delta}(\ii)\in\mathcal{C}_{\delta}^{\sigma}} \big( \nu_{\boldsymbol{\mu}}(B(\pi(\ii),\sqrt{d}\cdot \delta)) \big)^{q} \leq T_{\sqrt{d}\cdot\delta}(\nu_{\boldsymbol{\mu}},q),
\end{equation*}
where the $\lesssim$ holds because $\#\mathcal{C}_{\delta}^{\sigma}\geq c\cdot \#T_{\delta}^{\sigma}(\tau_{\delta}^{\sigma})$. We obtained that $P(\boldsymbol{\psi}_q^{\boldsymbol{\mu}})-\varepsilon \leq T(\nu_{\boldsymbol{\mu}},q)$ for any $\varepsilon>0$, hence, the proof is complete for $q\geq 0$.

Now assume $q< 0$ and fix $\varepsilon>0$. We choose the type $\tau_{\delta}^{\sigma}\in\mathcal{T}_{\delta}^{\sigma}$ with $t_q^{\boldsymbol{\mu}}(\tau_{\delta}^{\sigma})\geq P(\boldsymbol{\psi}_q^{\boldsymbol{\mu}})-\varepsilon$ the same way. This time we want to inscribe balls within the image of each approximate cube $B_{\delta}(\ii)\in T_{\delta}^{\sigma}(\tau_{\delta}^{\sigma})$. This may not be possible, however, we can use the map $\alpha_{\delta}(\cdot)$ defined in~\eqref{eq:73} to obtain another set of approximate cubes with the nice properties given in Lemma~\ref{lem:BallInCube}. More specifically, consider the collection $\mathcal{C}_{\delta}^{\sigma}= \{ B_{\delta}(\alpha_{\delta}(\ii)):\, B_{\delta}(\ii)\in T_{\delta}^{\sigma}(\tau_{\delta}^{\sigma})\}$. Since $\alpha_{\delta}(\cdot)$ is an injection, it follows from the SPPC and~\eqref{eq:71} that $\{ B(\pi(\alpha_{\delta}(\ii)),C_0\cdot \delta):\, B_{\delta}(\ii)\in T_{\delta}^{\sigma}(\tau_{\delta}^{\sigma})\}$ is a centred packing of $F$. We use this packing to bound the $L^q$ spectrum from below
\begin{align*}
T_{C_0\cdot\delta}(\nu_{\boldsymbol{\mu}},q) &\geq \sum_{B_{\delta}(\alpha_{\delta}(\ii))\in\mathcal{C}_{\delta}^{\sigma}} \big( \nu_{\boldsymbol{\mu}}(B(\pi(\alpha_{\delta}(\ii)),C_0\cdot \delta)) \big)^{q} \stackrel{\eqref{eq:71}}{\geq} \sum_{B_{\delta}(\alpha_{\delta}(\ii))\in\mathcal{C}_{\delta}^{\sigma}} \big(\nu_{\boldsymbol{\mu}}(\pi(B_{\delta}(\alpha_{\delta}(\ii))))\big)^{q} \\
&= \sum_{B_{\delta}(\alpha_{\delta}(\ii))\in\mathcal{C}_{\delta}^{\sigma}} \big(\widetilde{\nu}_{\boldsymbol{\mu}}(B_{\delta}(\alpha_{\delta}(\ii)))\big)^{q}
\stackrel{\eqref{eq:70}}{\approx} \sum_{B_{\delta}(\ii)\in T_{\delta}^{\sigma}(\tau_{\delta}^{\sigma})} \big(\widetilde{\nu}_{\boldsymbol{\mu}}(B_{\delta}(\ii))\big)^{q}  \geq  \delta^{-P(\boldsymbol{\psi}_q^{\boldsymbol{\mu}})+\varepsilon},
\end{align*}
by the choice of $\tau_{\delta}^{\sigma}\in\mathcal{T}_{\delta}^{\sigma}$, which completes the proof of the lower bound.

\section{Proof of Theorem~\ref{thm:dimBFmeasure}}\label{sec:proofFrostmanBox}

Using Lemma~\ref{lem:MeasureSymbCube}, we give uniform bounds on the $\nu_{\boldsymbol{\mu}}$ measure of approximate cubes.
\begin{lemma}\label{lem:UniformBoundMeasure}
Assuming the SPPC, any symbolic $\delta$-approximate cube $B_{\delta}(\ii)$ satisfies
\begin{equation*}
\delta^{\max_{\sigma \in \mathcal{A}}\, \sup_{\mathbf{P}_{\! \sigma}\in\mathcal{Q}^{\sigma}} S(\mathbf{P}_{\!\sigma})} \lesssim \nu_{\boldsymbol{\mu}}(\pi(B_{\delta}(\ii))) \lesssim \delta^{\min_{\sigma \in \mathcal{A}}\, \inf_{\mathbf{P}_{\! \sigma}\in\mathcal{Q}^{\sigma}} S(\mathbf{P}_{\!\sigma})}.
\end{equation*}
\end{lemma}
\begin{proof}
From Lemma~\ref{lem:MeasureSymbCube} we know that $\nu_{\boldsymbol{\mu}}(\pi(B_{\delta}(\ii))) 
\approx \delta^{S(\tau_{\delta}^{\sigma}(\ii))}$ assuming $\ii$ is $\sigma$-ordered at scale $\delta$, where $\tau_{\delta}^{\sigma}(\ii)\in\mathcal{T}_{\delta}^{\sigma}$. From Lemma~\ref{lem:4} we also know that $\mathcal{T}_{\delta}^{\sigma}$ becomes dense in $\mathcal{Q}^{\sigma}$ as $\delta\to 0$. Therefore, $\inf_{\mathbf{P}_{\! \sigma}\in\mathcal{Q}^{\sigma}} S(\mathbf{P}_{\!\sigma}) \leq S(\tau_{\delta}^{\sigma}(\ii)) \leq \sup_{\mathbf{P}_{\! \sigma}\in\mathcal{Q}^{\sigma}} S(\mathbf{P}_{\!\sigma})$, completing the proof.
\end{proof}

In the following lemma, we write $t_q^{\boldsymbol{\mu}}(\mathbf{P}_{\!\sigma})$ to indicate that in definition~\eqref{eq:36} of $t(\mathbf{P}_{\!\sigma})$ we use the potential $\boldsymbol{\psi}_q^{\boldsymbol{\mu}}$.

\begin{lemma}\label{lem:asymptotes}
We have
\begin{equation*}
\lim_{q\to +\infty} \frac{-1}{q} \max_{\sigma \in \mathcal{A}}\, \sup_{\mathbf{P}_{\! \sigma}\in\mathcal{Q}^{\sigma}} t_q^{\boldsymbol{\mu}}(\mathbf{P}_{\!\sigma}) = \min_{\sigma \in \mathcal{A}}\, \inf_{\mathbf{P}_{\! \sigma}\in\mathcal{Q}^{\sigma}} S(\mathbf{P}_{\!\sigma})
\end{equation*}
and 
\begin{equation*}
\lim_{q\to -\infty} \frac{-1}{q} \max_{\sigma \in \mathcal{A}}\, \sup_{\mathbf{P}_{\! \sigma}\in\mathcal{Q}^{\sigma}} t_q^{\boldsymbol{\mu}}(\mathbf{P}_{\!\sigma}) = \max_{\sigma \in \mathcal{A}}\, \sup_{\mathbf{P}_{\! \sigma}\in\mathcal{Q}^{\sigma}} S(\mathbf{P}_{\!\sigma}).
\end{equation*}
\end{lemma}
\begin{proof}
The uniform bounds $0\leq H(\mathbf{p}_{\sigma_n})\leq \log \#\mathcal{I}$ and $0\leq C_n^{(d),\sigma}(\mathbf{P}_{\!\sigma}) \leq -1/\log \lambda_{\min}$ hold for all $\mathbf{P}_{\! \sigma}\in\mathcal{Q}^{\sigma}$. Using these, we can bound
\begin{equation*}
\max_{\sigma \in \mathcal{A}}\, \sup_{\mathbf{P}_{\! \sigma}\in\mathcal{Q}^{\sigma}} -q\cdot S(\mathbf{P}_{\!\sigma}) \leq  \max_{\sigma \in \mathcal{A}}\, \sup_{\mathbf{P}_{\! \sigma}\in\mathcal{Q}^{\sigma}} t_q^{\boldsymbol{\mu}}(\mathbf{P}_{\!\sigma}) \leq  \max_{\sigma \in \mathcal{A}}\, \sup_{\mathbf{P}_{\! \sigma}\in\mathcal{Q}^{\sigma}} -q\cdot S(\mathbf{P}_{\!\sigma}) + d \frac{\log \#\mathcal{I}}{-\log \lambda_{\min}}.
\end{equation*}
First assume $q>0$ and divide through by $-q$. We obtain that
\begin{equation*}
\min_{\sigma \in \mathcal{A}}\, \inf_{\mathbf{P}_{\! \sigma}\in\mathcal{Q}^{\sigma}} S(\mathbf{P}_{\!\sigma}) \geq \frac{-1}{q} \max_{\sigma \in \mathcal{A}}\, \sup_{\mathbf{P}_{\! \sigma}\in\mathcal{Q}^{\sigma}} t_q^{\boldsymbol{\mu}}(\mathbf{P}_{\!\sigma}) \geq \min_{\sigma \in \mathcal{A}}\, \inf_{\mathbf{P}_{\! \sigma}\in\mathcal{Q}^{\sigma}} S(\mathbf{P}_{\!\sigma}) - \frac{d}{q}\cdot  \frac{\log \#\mathcal{I}}{-\log \lambda_{\min}}.
\end{equation*}
Taking the limit as $q\to +\infty$ proves the first assertion.

Now assume $q<0$ and again divide through by $-q$. We now obtain that
\begin{equation*}
\max_{\sigma \in \mathcal{A}}\, \sup_{\mathbf{P}_{\! \sigma}\in\mathcal{Q}^{\sigma}} S(\mathbf{P}_{\!\sigma}) \leq \frac{-1}{q} \max_{\sigma \in \mathcal{A}}\, \sup_{\mathbf{P}_{\! \sigma}\in\mathcal{Q}^{\sigma}} t_q^{\boldsymbol{\mu}}(\mathbf{P}_{\!\sigma}) \leq \max_{\sigma \in \mathcal{A}}\, \sup_{\mathbf{P}_{\! \sigma}\in\mathcal{Q}^{\sigma}} S(\mathbf{P}_{\!\sigma}) - \frac{d}{q}\cdot  \frac{\log \#\mathcal{I}}{-\log \lambda_{\min}}.
\end{equation*}
Taking the limit as $q\to -\infty$ completes the proof.
\end{proof}

\begin{lemma}\label{lem:Frostman}
Let $\nu_{\boldsymbol{\mu}}$ be a self-affine measure on the sponge $F$ that satisfies the SPPC. Then
\begin{equation*}
\dim_{\mathrm F} \nu_{\boldsymbol{\mu}} = \min_{\sigma \in \mathcal{A}}\, \inf_{\mathbf{P}_{\! \sigma}\in\mathcal{Q}^{\sigma}} S(\mathbf{P}_{\!\sigma}).
\end{equation*}
\end{lemma}
\begin{proof}
Let $x\in F$ and $0<\delta<1$ be arbitrary. Recall from Section~\ref{sec:upper} that $\mathcal{B}_{\delta}^x$ denotes the set of those symbolic $\delta$-approximate cubes whose image under $\pi$ intersect $B(x,\delta)\cap F$. Using that $\#\mathcal{B}_{\delta}^x\approx 1$, we obtain from Lemma~\ref{lem:UniformBoundMeasure} that
\begin{equation*}
\nu_{\boldsymbol{\mu}}(B(x,\delta)\cap F) \lesssim \max_{B\in\mathcal{B}_{\delta}^x} \nu_{\boldsymbol{\mu}}(\pi(B)) \lesssim \delta^{\min_{\sigma \in \mathcal{A}}\, \inf_{\mathbf{P}_{\! \sigma}\in\mathcal{Q}^{\sigma}} S(\mathbf{P}_{\!\sigma})},
\end{equation*}    
which shows that $\dim_{\mathrm F} \nu_{\boldsymbol{\mu}} \geq \min_{\sigma \in \mathcal{A}}\, \inf_{\mathbf{P}_{\! \sigma}\in\mathcal{Q}^{\sigma}} S(\mathbf{P}_{\!\sigma})$.

For the other direction, fix $\varepsilon>0$ and choose (any) $\sigma\in\mathcal{A}$ which minimises $\inf_{\mathbf{P}_{\! \sigma}\in\mathcal{Q}^{\sigma}} S(\mathbf{P}_{\!\sigma})$. By Lemma~\ref{lem:4}, $\mathcal{T}_{\delta}^{\sigma}$ becomes dense in $\mathcal{Q}^{\sigma}$, moreover, $S(\mathbf{P}_{\!\sigma})$ is continuous in $\mathbf{P}_{\!\sigma}$, therefore, for every $\delta$ small enough there exists an $\ii\in\Sigma_{\delta/\sqrt{d}}^{\sigma}$ such that $\tau_{\delta/\sqrt{d}}^{\sigma}(\ii)\in\mathcal{T}_{\delta/\sqrt{d}}^{\sigma}$,
\begin{equation*}
\pi(B_{\delta/\sqrt{d}}(\ii))\subseteq B(\pi(\ii),\delta) \;\;\text{ and }\;\; S(\tau_{\delta/\sqrt{d}}^{\sigma}(\ii))\leq \inf_{\mathbf{P}_{\! \sigma}\in\mathcal{Q}^{\sigma}} S(\mathbf{P}_{\!\sigma})+\varepsilon.
\end{equation*}
As a result, Lemma~\ref{lem:UniformBoundMeasure} again implies that
\begin{equation*}
\nu_{\boldsymbol{\mu}}(B(\pi(\ii),\delta)\cap F) \gtrsim \nu_{\boldsymbol{\mu}}\big(\pi(B_{\delta/\sqrt{d}}(\pi(\ii)))\big) \approx \delta^{S(\tau_{\delta/\sqrt{d}}^{\sigma}(\ii))} \gtrsim \delta^{\min_{\sigma \in \mathcal{A}}\, \inf_{\mathbf{P}_{\! \sigma}\in\mathcal{Q}^{\sigma}} S(\mathbf{P}_{\!\sigma})+\varepsilon}.
\end{equation*} 
Since $\varepsilon>0$ was arbitrary, the proof is complete.
\end{proof}
 
\begin{lemma}\label{lem:box}
Let $\nu_{\boldsymbol{\mu}}$ be a self-affine measure on the sponge $F$ that satisfies the SPPC. Then
\begin{equation*}
	\dim_{\mathrm B} \nu_{\boldsymbol{\mu}} = \max_{\sigma \in \mathcal{A}}\, \sup_{\mathbf{P}_{\! \sigma}\in\mathcal{Q}^{\sigma}} S(\mathbf{P}_{\!\sigma}).
\end{equation*}
\end{lemma}
\begin{proof}
Let $x\in F$ and $0<\delta<1$ be arbitrary, furthermore, $\ii\in\Sigma$ such that $\pi(\ii)=x$. Then $\pi(B_{\delta/\sqrt{d}}(\ii))\subseteq B(x,\delta)\cap F$, hence, by Lemma~\ref{lem:UniformBoundMeasure},
\begin{equation*}
	\nu_{\boldsymbol{\mu}}(B(x,\delta)\cap F) \gtrsim \nu_{\boldsymbol{\mu}}(\pi(B_{\delta/\sqrt{d}}(\ii))) \gtrsim \delta^{\max_{\sigma \in \mathcal{A}}\, \sup_{\mathbf{P}_{\! \sigma}\in\mathcal{Q}^{\sigma}} S(\mathbf{P}_{\!\sigma})},
\end{equation*}    
which shows that $\overline{\dim}_{\mathrm B} \nu_{\boldsymbol{\mu}} \leq \max_{\sigma \in \mathcal{A}}\, \sup_{\mathbf{P}_{\! \sigma}\in\mathcal{Q}^{\sigma}} S(\mathbf{P}_{\!\sigma})$.

For the other direction, fix $\varepsilon>0$ and choose (any) $\sigma\in\mathcal{A}$ which maximises $\sup_{\mathbf{P}_{\! \sigma}\in\mathcal{Q}^{\sigma}} S(\mathbf{P}_{\!\sigma})$. For $\delta$ small enough there exists $\ii\in\Sigma_{\delta}^{\sigma}$ such that $S(\tau_{\delta}^{\sigma}(\ii))\geq \sup_{\mathbf{P}_{\! \sigma}\in\mathcal{Q}^{\sigma}} S(\mathbf{P}_{\!\sigma})-\varepsilon$. Using Lemma~\ref{lem:MeasureSymbCube} and~\ref{lem:BallInCube},
\begin{equation*}
\nu_{\boldsymbol{\mu}}(B(\pi(\alpha_{\delta}(\ii)),C_0\cdot \delta))\cap F) \leq \nu_{\boldsymbol{\mu}}(\pi(B_{\delta}(\alpha_{\delta}(\ii)))) \approx \nu_{\boldsymbol{\mu}}(\pi(B_{\delta}(\ii))) \lesssim \delta^{\max_{\sigma \in \mathcal{A}}\, \sup_{\mathbf{P}_{\! \sigma}\in\mathcal{Q}^{\sigma}} S(\mathbf{P}_{\!\sigma})-\varepsilon}.
\end{equation*}
Since $\varepsilon>0$ was arbitrary, the proof is complete.
\end{proof}

\begin{proof}[Proof of Theorem~\ref{thm:dimBFmeasure}]
The claims about $\dim_{\mathrm F} \nu_{\boldsymbol{\mu}}$ follow directly from Theorem~\ref{thm:Lqmain}, Proposition~\ref{prop:maxminexponent} and Lemmas~\ref{lem:asymptotes} and~\ref{lem:Frostman}. The claims about $\dim_{\mathrm B} \nu_{\boldsymbol{\mu}}$ follow directly from Theorem~\ref{thm:Lqmain}, Proposition~\ref{prop:maxminexponent} and Lemmas~\ref{lem:asymptotes} and~\ref{lem:box}. If $F$ is a $\sigma$-ordered Lalley--Gatzouras sponge, then $\mathcal{A}=\{\sigma\}$ and $\mathcal{Q}^{\sigma}=\mathcal{P}^{\sigma}$, so the claims follow from Proposition~\ref{prop:maxminexponent}.
\end{proof}

\subsection*{Acknowledgment}
The author was supported by a \emph{Leverhulme Trust Research Project Grant} (RPG-2019-034).

\bibliographystyle{abbrv}
\bibliography{biblio_methodtypes}

\end{document}